\documentclass[12pt,twoside]{book}
\usepackage{syntonly}

\usepackage{fancyhdr} 
\usepackage{exercise}

\usepackage{amsmath}
\usepackage{amssymb}
\usepackage{amsfonts}
\usepackage{graphicx}
\usepackage{amsthm}
\usepackage{enumerate}
\usepackage{lscape}
\usepackage{dsfont}
\usepackage{color}
\usepackage{multirow}
\usepackage{fancyhdr}
\usepackage{pgfplots}
\usepackage[T1]{fontenc}
\usepackage[applemac]{inputenc}

\usepackage{hyperref}
\hypersetup{colorlinks=true,linkcolor=black}

\usepackage{pgf,tikz}
\usetikzlibrary{arrows}

\numberwithin{Exercise}{section}

\newcommand{\R}{\mathbb{R}}                     
\newcommand{\p}{\mathbb{P}}     
\newcommand{\z}{\mathbb{Z}}

\newcommand{\nat}{\mathbb{N}}

\newcommand{\oo}{\mathbb{O}}
\newcommand{\CP}{\mathbb{C}\mathrm{P}}

\newcommand{\CH}{\mathbb{C}\mathrm{H}}

\newcommand{\isom}{\mathrm{Isom}}
\newcommand{\aut}{\mathrm{Aut}}

\newcommand{\ol}{\mathrm{Hol}}
\newcommand{\hilb}{\mathcal{H}}

\newcommand{\po}{\forall}                       
\newcommand{\f}{\rightarrow}                    
\newcommand{\C}{\mathbb{C}}                     
\newcommand{\de}{\partial}                      

\newcommand{\M}{\mathrm{M}}

\newcommand{\GL}{\mathop{\mathrm{GL}}}
\newcommand{\Aut}{\mathop{\mathrm{Aut}}}

\newcommand{\dd}{\mathrm{D}}
\newcommand{\Kj}{\mathrm{K}}
\newcommand{\N}{\mathrm{N}}

\newcommand{\ric}{\mathrm{Ric}}

\newcommand{\sgn}{\mathrm{sgn}}

\newcommand{\gen}{\mathop{\mathrm{span}}}
\newcommand{\ro}{\rho}
\newcommand{\K}{K\"{a}hler}
\newcommand{\rk}{r}
\newcommand{\gs}{\mathfrak{s}}
\newcommand{\set}[2]{ \left\{\,#1\,;\,#2\,\right\} }
\newcommand{\dirsum}{\sideset{\ }{^{\oplus}}\sum}
\newcommand{\ep}{\epsilon}
\newcommand{\lmb}{\lambda}
\newcommand{\w}[1]{\widetilde{#1}}

\newtheorem{theor}{Theorem}[section]
\newtheorem{prop}[theor]{Proposition}
\newtheorem{defin}[theor]{Definition}
\newtheorem{lem}[theor]{Lemma}
\newtheorem{cor}[theor]{Corollary}
\newtheorem{ex}[theor]{Example}
\newtheorem{remark}[theor]{Remark}
\newtheorem{conj}[theor]{Conjecture}

\usepackage{hyperref}
\hypersetup{colorlinks=true,linkcolor=black}

\title{K\"ahler immersions of K\"ahler manifolds into complex space forms}
\author{Andrea Loi\and Michela Zedda}

\begin{document}

\frontmatter
\maketitle

\chapter*{Preface}
The study of \K\ immersions of a given real analytic \K\ manifold into a finite or infinite dimensional complex space form originates from the pioneering work of Eugenio Calabi \cite{Cal}. With a stroke of genius Calabi defines a powerful tool, a special (local) potential called {\em diastasis function}, which allows him to obtain necessary and sufficient conditions for a neighbourhood of a point to be locally \K\ immersed into a finite or infinite dimensional complex space form. As application of its criterion, he also provides a classification of (finite dimensional) complex space forms admitting a K\"ahler immersion into another. 
Although, a complete classification of K\"ahler manifolds admitting a K\"ahler immersion into complex space forms is not known, not even when the \K\ manifolds involved are of great interest, e.g. when they are \K\--Einstein or homogeneous spaces. In fact, the diastasis function is not always explicitely given and Calabi's criterion, although theoretically impeccable, most of the time is of difficult application. Nevertheless, throughout the last 60 years many mathematicians have worked on the subject and many interesting results have been obtained. 

The aim of this book is to describe  Calabi's original work, to provide a detailed account  of what is known today on the subject and to point out some open problems.

Each chapter begins with a brief summary of the topics discussed and ends with a list of exercises which help the reader to test his understanding.

Apart from the topics discussed in Section \ref{hbdcpn} of Chapter \ref{simmes}, which could be skipped without compromising the understanding of the rest of the book, the requirements to read this book are a basic knowledge of complex and \K\  geometry (treated, e.g. in  Moroianu's book \cite{Mor}).\\

The authors are grateful to Claudio Arezzo and Fabio Zuddas for a careful reading of the text and for valuable comments that have improved the book's exposure.
\tableofcontents

\mainmatter

\chapter{The diastasis function}

In this chapter we describe the {\em diastasis function}, a basic tool introduced by E. Calabi in  \cite{Cal} which is fundamental to study K\"ahler immersions of K\"ahler manifolds into complex space forms. 

In Section \ref{diastasis} we define the {diastasis function} and summarize its basic properties, while in Section \ref{csf} we describe the diastasis functions of complex space forms, which represent the basic examples of K\"ahler manifolds. Finally, in Section \ref{ihs} we give the formal definition of what a {\em K\"ahler immersion} is and prove that the indefinite Hilbert space constitutes a universal K\"ahler manifold, in the sense that it is a space in which every real analytic K\"ahler manifold can be locally K\"ahler immersed.

\section{Calabi's diastasis function}\label{diastasis}

Let $M$ be an $n$-dimensional complex manifold endowed with a real analytic K\"ahler metric $g$. Recall that the K\"ahler metric $g$ is real analytic if fixed a local coordinate system $z=(z_1,\dots,z_n)$ on a neighbourhood $U$ of any point $p\in M$, it can be described on $U$ by a real analytic K\"ahler potential $\Phi\!:U\f \R$. In that case the potential $\Phi(z)$ can be analytically continued to an open neighbourhood $W\subset U\times U$ of the diagonal. Denote this extension by $\Phi(z,\bar w)$. 
\begin{defin}
The \emph{diastasis function} $\dd(z,w)$ on $W$ is defined by:
\begin{equation}\label{diastdefinition}
\dd(z,w)=\Phi\left(z,\bar z\right)+\Phi\left(w, \bar w\right)-\Phi\left(z,\bar w\right)-\Phi\left(w,\bar z\right).
\end{equation}
\end{defin}
The following proposition describes the basic properties of $\dd(z,w)$.
\begin{prop}[E. Calabi, \cite{Cal}]\label{diastmetric}
The diastasis function $\dd(z,w)$ given by (\ref{diastdefinition}) satisfies the following properties:
\begin{enumerate}
\item[(i)]  it is uniquely determined by the K\"ahler metric $g$ and it does not depend on the choice of the local coordinate system;
\item[(ii)] it is real valued in its domain of (real) analyticity;
\item[(iii)] it is symmetric in $z$ and $w$ and $\dd(z,z)=0$;
\item[(iv)] once fixed one of its two entries, it is a K\"ahler potential for $g$.
\end{enumerate}
\end{prop}
\begin{proof}$\textrm{}$
\begin{enumerate}
\item[(i)] By the $\partial\bar\partial$--Lemma a K\"ahler potential is defined up to the addition of the real part of a holomorphic function, namely, given two K\"ahler potentials $\Phi$ and $\Phi'$ on $U\subset M$, then $\Phi'=\Phi+f+\bar f$ for some holomorphic function $f$. Conclusions follow again by \eqref{diastdefinition}.
\item[(ii)] Since $\Phi(z,\bar z)=\Phi(z)$ is real valued, then from $\Phi(z,\bar z)=\overline{\Phi(z,\bar z)}$ and  by uniqueness of the extension it follows $\Phi(z,\bar w)=\overline{\Phi( w,\bar z)}$.
\item[(iii)] It follows directly from \eqref{diastdefinition}.
\item[(iv)] Fix $w$ (the case of $z$ fixed is totally similar). Then:
$$
\frac{\partial^2}{\partial z_j\partial \bar z_k}\dd(z,w)=\frac{\partial^2}{\partial z_j\partial \bar z_k}\Phi\left(z,\bar z\right)=\frac{\partial^2}{\partial z_j\partial \bar z_k}\Phi\left(z\right).
$$
\end{enumerate}
\end{proof}
The last property justifies the following second definition.
\begin{defin}
If $w=(w_1,\dots,w_n)$ are local coordinates for a fixed point $p\in M$, the {\em diastasis function centered at $p$} is given by:
$$
\dd_p(z)=\dd(z,w).
$$
In particular, if $p$ is the origin of the coordinate system chosen, we write $\dd_0(z)$.
\end{defin}


The importance of the diastasis function for our purposes is expressed by the following:
\begin{prop}[E. Calabi, \cite{Cal}]\label{induceddiast}
 Let $(M, g)$ and $(S,G)$ be K\"ahler manifolds and assume $G$ to be real analytic. Denote by $\omega$ and $\Omega$ the K\"ahler forms associated to $g$ and $G$ respectively. If there exists a holomorphic map $f\!:(M,g)\f(S,G)$ such that $f^*\Omega=\omega$, then the metric $g$ is real analytic. Further, denoted by $\dd^M_p\!:U\f \R$ and $\dd^S_{f(p)}\!:V\f\R$ the diastasis functions of $(M,g)$ and $(S,G)$ around $p$ and $f(p)$ respectively, we have $\dd_{f(p)}^S\circ f=\dd^M_p$ on $f^{-1}(V)\cap U$.
\end{prop}
\begin{proof}
Observe first that the metric $g$ on $M$ is real analytic being the pull--back through a holomorphic map of the real analytic metric $G$. In order to prove the second part, fix a coordinate system $\{z\}$ around $p\in M$. From $f^*G|_{V\cap f(U)}=g|_{f^{-1}(V)\cap U}$, if $\Phi^S$ and $\Phi^M$ are K\"ahler potential for $G$ and $g$ around $f(p)$ and $p$ respectively, we get:
$$
\frac{\partial^2\Phi^S(f(z),\overline{f(z)})}{\partial z_j\partial \bar z_k}=\frac{\partial^2\Phi^M(z,\bar z)}{\partial z_j\partial \bar z_k},
$$
i.e. $\Phi^S(f(z),\overline{f(z)})=\Phi^M(z,\bar z)+h+\bar h$ and conlcusion follows by \eqref{diastdefinition}.
\end{proof}
Observe that the pull-back of any other K\"ahler potential is still a K\"ahler potential, but the fact underlined in the previous proposition that holomorphic maps pull-back the diastasis function in the diastasis function is a fundamental ingredient to prove Calabi's criteria in the next chapter.\\

Recall that given any K\"ahler manifold $(M,g)$ and a K\"ahler potential $\Phi$  around a point $p\in M$, there always exists a coordinate system $\{z_j\}$ around $p$, which satisfies:
$$
\frac{\partial^2\Phi}{\partial z_j\partial \bar z_k}(p)=g_{j\bar k}(p)=\delta_{jk},
$$
$$
\frac{\partial^3\Phi}{\partial z_l\partial z_j\partial \bar z_k}(p)=\frac{\partial}{\partial z_l}\left(g_{j\bar k}\right)(p)=0;\quad \frac{\partial^3\Phi}{\partial \bar z_l\partial z_j\partial \bar z_k}(p)=\frac{\partial}{\partial\bar  z_l}\left(g_{j\bar k}\right)(p)=0.
$$
If we assume the K\"ahler metric $g$ to be also real analytic then in such coordinates the diastasis satisfies:
\begin{equation}\label{bochnercoordinates}
\dd_p(z)=\sum_{\alpha=1}^n|z_\alpha|^2+\psi_{2,2},
\end{equation}
where $\psi_{2,2}$ is a power series with degree $\geq 2$ in both the variables $z$ and $\bar z$. These coordinates, uniquely defined up to a unitary transformation (cfr. \cite{bochner,Cal}), are called the \emph{normal} or \emph{Bochner's coordinates} around the point $p$ (see \cite{bochner,Cal,hulin,hulinlambda,ruan,tian4} for more details and further results about Bochner's coordinates).

The following proposition shows how the diastasis function is related to the geodesic distance explaining the name \emph{diastasis}, from the Greek {\small $\delta\iota\acute{\alpha}\sigma\tau\alpha\sigma\iota\varsigma$}, that means {\em distance}.
\begin{prop}[E. Calabi, \cite{Cal}]
If $\ro(p,q)$ is the geodesic distance between $p$ and $q$, then
\begin{equation}
\dd(p,q)=(\ro(p,q))^2+O((\ro(p,q))^4).\nonumber
\end{equation}
\end{prop}
\begin{proof}
Fix $p\in M$ and let $\{z\}$ be Bochner coordinates around it. Then, since $\dd_p(z)$ is a K\"ahler potential for $g$ around $p$, its power expansion around $p$ in the variables $z$ and $\bar z$ reads:
$$
\dd(p,q)=\dd_p(z)=||z||^2+\psi_{2,2}(z,\bar z),
$$
where $\psi_{2,2}$ is a power series with no terms of degree less than 2 in either the variables $z$ and $\bar z$.
On the other hand, since at the origin one has $g_{j\bar k}=\delta_{jk}$, the geodesic distance satisfies:
$$
(\ro(p,q))^2=||z||^2+O((||z||^2)^2),
$$
and conclusion follows.
\end{proof}

We conclude this section giving a very useful characterization of the diastasis, easily deducible by the definition, in terms of its power expansion. 
In order to semplify the notation, let us first fix the following multi-index convention that we are going to use through all this book.
We arrange every $n$-tuple of nonnegative integers as the sequence $m_j=(m_{j,1},\dots,m_{j,n})$ with not decreasing order, that is $m_0=(0,\dots,0)$ and if $|m_j|=\sum_{\alpha=1}^n m_{j,\alpha}$, we have $|m_j|\leq |m_{j+1}|$ for all positive integer $j$. Further $z^{m_j}$ denotes  the monomial in $n$ variables $\prod_{\alpha=1}^n z_\alpha^{m_{j,\alpha}}$. For example, if $n=2$ we can consider the ordering $m_0=(0,0)$, $m_1=(1,0)$, $m_2=(0,1)$, $m_3=(1,1)$, $m_4=(2,0)$, etc. and we would have $z^{m_0}=1$, $z^{m_1}=z_1$, $z^{m_2}=z_2$, $z^{m_3}=z_1 z_2$, $z^{m_4}=z_1^2$, etc. Notice that the order is not uniquely determined by these rules, since we are allowed to switch terms of equal module $|m_j|$ (i.e. in the $2$ dimensional case we may also take $m_1=(0,1)$, $m_2=(1,0)$, etc.).

\begin{theor}[E. Calabi, \cite{Cal}]\label{chardiast}
Among all the K\"ahler potentials the diastasis $\dd_p(z)$ is characterized by the fact that in every coordinate system $(z_1,\dots,z_n)$ centered at $p$, the $\infty\times\infty$ matrix of coefficients $(a_{jk})$ in its power expansion around the origin
\begin{equation}\label{powexdiastc}
\dd_p(z)=\sum_{j,k=0}^{\infty}a_{jk}z^{m_j}\bar z^{m_k},
\end{equation}
satisfies $a_{j0}=a_{0j}=0$ for every nonnegative integer $j$.
\end{theor}
\begin{proof}
Let $(z_1,\dots,z_n)$ be a coordinate system centered at $p\in M$ and assume that:
$$
 \Phi(z,\bar z)=\sum_{j,k=0}^{\infty}a_{jk}z^{m_j}\bar z^{m_k},
$$
is a K\"ahler potential satisfying $a_{j0}=a_{0j}=0$ for every nonnegative integer $j$. Then,
$$
 \Phi(z,0)=0=\Phi(0,\bar z),
$$
and by \eqref{diastdefinition} we have:
$$
\dd_0(z,\bar z)= \Phi(z,\bar z).
$$
Conversely, let:
$$
\Phi(z,\bar z)=\sum_{j,k=0}^{\infty}a_{jk}z^{m_j}\bar z^{m_k},
$$
be the power expansion around the origin of a K\"ahler potential. Then by \eqref{diastdefinition}:
$$
\dd_p(z)=\dd(z,0)=\sum_{j,k=0}^{\infty}a_{jk}z^{m_j}\bar z^{m_k}+a_{00}-\sum_{j=0}^{\infty}a_{j0}z^{m_j}-\sum_{k=0}^{\infty}a_{0k}\bar z^{m_k},
$$
and conclusion follows.
\end{proof}

\section{Complex space forms}\label{csf}
We describe here the diastasis of complex space forms. Recall that a complex space form is a finite or infinite dimensional K\"ahler manifold of constant holomorphic sectional curvature, that if we assume to be complete and simply connected, then up to homotheties it can be of the following three types, according to the sign of the holomorphic sectional curvature.

\noindent{\bf 1. Complex Euclidean space.}
The complex Euclidean space $\C^N$ of complex dimension $N\leq \infty$, endowed with the flat metric denoted by $g_0$.  Here $\C^\infty$ denotes the Hilbert space $l ^2({\C})$ consisting of sequences $w_j\in\C$, $j=1,2,\dots$,  such that $\sum_{j=1}^{+\infty}|w_j|^2<+\infty$. 
The diastasis, that we will denote from now on by $\dd^0$, is equal to the square of the geodesic distance, i.e. it is given by:
$$
\dd^0(p,q)=||p-q||^2.
$$
Obviously, $\dd^0$ is positive except for $p=q$.
The canonical coordinates $(z_1,\dots,z_n)$ of $\C^N$ are Bochner coordinates around the origin and the globally defined diastasis $\dd^0_0\!:\C^N\f \R$ centered at the origin reads:
\begin{equation}\label{diastc}
\dd^0_0(z)=\sum_{j=1}^N|z_j|^2.
\end{equation}


\noindent{\bf 2. Complex projective space.}
The complex projective space $\CP_b^N=(\CP^N, g_b)$, namely the complex projective space $\CP^N$ of complex dimension $N\leq \infty$, with the Fubini-Study metric $g_{b}$ of holomorphic sectional curvature $4b$ for $b>0$. Let $[Z_0,\dots,Z_N]$ be homogeneous coordinates,
$p=[1,0,\dots,0]$ and $U_0=\{Z_0\neq 0\}$. The affine coordinates $z_1,\dots, z_N$ on $U_0$  defined by $z_j=Z_j/(\sqrt{b}Z_0)$ are Bochner coordinates centered at $p$. The diastasis on $U_0$ centered at the origin and defined on $U_0$ reads:
\begin{equation}\label{diastcp}
\dd^b_0(z)=\frac{1}{b}\log\left(1+b\sum_{j=1}^N|z_j|^2\right), \quad \textrm{for}\ b>0.
\end{equation}
Since \eqref{diastcp} in homogeneous coordinates reads:
$$
\dd^b_0(Z)=\frac{1}{b}\log\sum_{j=0}^N\frac{|Z_j|^2}{|Z_0|^2},
$$
by \eqref{diastdefinition} we have:
\begin{equation}\label{diastcphom}
\dd^b(Z,W)=\frac{1}{b}\log\frac{\sum_{j,k=0}^N{|Z_j|^2}{|W_k|^2}}{\left|\sum_{j=0}^N{Z_j\bar W_j}\right|^2}.
\end{equation}
Observe that $\dd^b$ is positive everywhere.
Further, $g_{b}$ is Einstein with Einstein constant $\lambda=2b(N+1)$.\\

\noindent{\bf 3. Complex hyperbolic space.}
The complex hyperbolic space $\CH_b^N$ of complex dimension $N\leq \infty$, namely the unit ball $B\subset\C^N$ given by:
$$
B=\left\{(z_1,\dots,z_N)\in\C^N,\; \sum_{j=1}^N|z_j|^2<-\frac1b\right\},
$$
endowed with the hyperbolic metric $g_{b}$ of constant holomorphic sectional curvature $4b$, for $b<0$. Fixed a coordinate system around a point $p\in B$, the hyperbolic metric is described by the (globally defined) diastasis $\dd^b_0$ centered at the origin which reads as:
\begin{equation}\label{diastch}
\dd^b_0(z)=\frac{1}{b}\log\left(1+b\sum_{j=1}^N|z_j|^2\right), \quad \textrm{for}\ b<0.
\end{equation}
If we introduce homogeneous coordinates $(Z_0,\dots, Z_N)$, defined by $z_j=Z_j/(\sqrt{-b}Z_0)$, similarly to the case of the complex projective space, we obtain:
$$
\dd^b_0(Z)=\frac{1}{b}\log\frac{|Z_0|^2-\sum_{j=1}^N|Z_j|^2}{|Z_0|^2},
$$
and thus:
$$
\dd^b(Z,W)=\frac{1}{b}\log\frac{\left(|Z_0|^2-\sum_{j=1}^N|Z_j|^2\right)\left(|W_0|^2-\sum_{k=1}^N|W_k|^2\right)}{\left|Z_0\bar W_0-\sum_{j=1}^NZ_j\bar W_j\right|^2}.
$$
In this case $g_{b}$ is Einstein with Einstein constant $\lambda=2b(N+1)$. \\

\vskip 0.3cm

\noindent
{\bf Notation.}
\noindent
In the sequel we denote $g_1$ by $g_{FS}$ and $g_{-1}$ by $g_{hyp}$. Furthermore, in order to simplify the notation we write $\C^N$, $\CP^N$ and $\CH^N$ instead of $(\C^N,g_0)$, $(\CP^N, g_{FS})$ and $(\CH^N,g_{hyp})$. Finally, according with the notation in \cite{Cal}, we will write ${\rm F}(N,b)$ to refer to a complex space form of curvature $4b$ and dimension $N$. Observe that for the case $b=0$ the notation is justified since the diastasis $\dd^0$ can be seen as the limit for $b$ approaching $0$ of $\dd^b$. Moreover, notice also that with these notations one has
$\CP^N_b=F(N, b)$, $b>0$ and $\CH^N_b=F(N, b)$, for  $b<0$. 

\section{The indefinite Hilbert space}\label{ihs}
Consider the indefinite Hilbert space $E$ of sequences
$$
(x_1,x_{-1},x_2,x_{-2},\dots, x_j,x_{-j},\dots),\quad \sum_{j\in\z^*}|x_j|^2<\infty,
$$
endowed with the indefinite Hermitian metric defined by the {\em diastasis}:
$$
\dd^E_0(x)=\sum_{j\in\z^*}(\sgn j)|x_j|^2.
$$
\begin{defin}\label{lki}
We say that a complex manifold $(M,g)$ admits a local K\"ahler immersion into $E$ if given any point $p\in M$ there exists a neighbourhood $U$ of $p$ and a map $f\!:U\rightarrow E$ such that:
\begin{enumerate}
\item $f$ is holomorphic;
\item $f$ is isometric, namely $\dd^M_p(z)=\dd^E_{f(p)}(f(z))$;
\item there exists $0<R<+\infty$ such that $\sum_{j=1}^\infty|f_j(z)|^2<R$.
\end{enumerate}
\end{defin}
The last condition is justified by the following:
\begin{lem}[E. Calabi, \cite{Cal}]\label{convergence}
If a sequence $f_j(z)$ of holomorphic functions defined on a common domain satisfies $\sum_{j=1}^\infty |f_j(z)|^2<R$, for some $0<R<+\infty$,
then the function $f(z,\bar z)=\sum_{j=1}^\infty |f_j(z)|^2$ is real analytic as a function in the variables $z$ and $\bar z$.
\end{lem}
In particular observe that the weaker hypothesis $\sum_{j=1}^\infty|f_j(z)|^2<+\infty$ does not even imply $f$ to be continuous.

As we are about to prove, the indefinite Hilbert space $E$ constitutes a \emph{universal} K\"ahler manifold, in the sense that it is a space in which every real analytic K\"ahler manifold can be locally K\"ahler immersed.

More precisely we have the following:
\begin{theor}[E. Calabi, {\cite[pages 6-9]{Cal}}]\label{indefcal}
A complex manifold $M$ endowed with a metric $g$ admits a local K\"ahler immersion into the indefinite Hilbert space $E$ if and only if $g$ is a real analytic K\"ahler metric.
\end{theor}
\begin{proof}
Let $p\in M$ and let $\dd_p(z)$ be the distasis function of $g$ in a neighbourhood $U$ of $p$. If $f\!:U\rightarrow E$, $f(z)=(\dots, f_{-j}(z),\dots, f_{-1}(z),f_1(z),\dots, f_j(z),\dots)$ is a holomorphic and isometric immersion then by Prop. \ref{induceddiast}:
\begin{equation}\label{ddE}
\dd_p(z)=\sum_{j\in\z^*}(\sgn j)|f_j(z)-f_j(p)|^2,
\end{equation}
and by Lemma \ref{convergence} $\dd_p(z)$ is real analytic on $U$. Thus, since by ($iv$) of Prop. \ref{diastmetric} $\dd_p(z)$ is a K\"ahler potential around $p$ for the induced metric $g$ on $M$, $g$ is a real analytic K\"ahler metric.

Conversely, assume the metric $g$ to be real analytic. Then for each $p\in M$, there exists a neigborhood $U$ such that $\dd_p(z)$ is real analytic and admits the power expansion:
\begin{equation}\label{ddconv}
\dd_p(z,\bar z)=\sum_{j,k=1}^\infty a_{jk}\,z^{m_j}\bar z^{m_k}.
\end{equation}
Observe that being $\dd_p(z)$ real valued, $(a_{jk})$ is a $\infty\times\infty$ Hermitian matrix. We need to construct a sequence of functions $f_j$ which satisfies \eqref{ddE} and converges in norm in a sufficiently small neighbourhood of $p$. Let $r=(r_1,\dots, r_n)$ be an $n$-tuple of arbitrary positive numbers to be fixed later. Define:
$$
f_j(z):=\frac12\left(a_{j j}r^{m_j} +\frac1{r^{m_j}}\right)z^{m_j}+\sum_{k=j+1}^\infty a_{j k}r^{m_j}z^{m_k}
$$
$$
f_{-j}(z):=\frac12\left(a_{j j}r^{m_j} -\frac1{r^{m_j}}\right)z^{m_j}+\sum_{k=j+1}^\infty a_{j k}r^{m_j}z^{m_k}
$$
Then:
$$
|f_j(z)|^2-|f_{-j}(z)|^2=\sum_{i,k=j}^\infty a_{ik}z^{m_i}\bar z^{m_k}
$$
and from:
$$
\sum_{j\in\mathds{Z}^*}(\sgn j)|f_j(z)|^2=\sum_{j=1}^\infty\left(|f_j(z)|^2-|f_{-j}(z)|^2\right)=\sum_{j,k=1}^\infty a_{jk}z^{m_j}\bar z^{m_k},
$$
\eqref{ddE} follows. It remains to prove that the sequence $(f_j)_{j\in \mathds{Z}^*}$ converges in norm to a real analytic function, i.e. due to Lemma \ref{convergence}, that it satisfies the third condition of Def. \ref{lki} above.  From the definition of $f_{j}$ and $f_{-j}$, we get:
\begin{equation}
\begin{split}
|f_{\pm j}(z)|^2&=\left|\frac12\left(a_{j j}r^{m_j} \pm\frac1{r^{m_j}}\right)z^{m_j}+\sum_{k=j+1}^\infty a_{j k}r^{m_j}z^{m_k}\right|^2\\
&\leq \left|\frac1{2r^{m_j}}z^{m_j}+\sum_{k=1}^\infty a_{j k}r^{m_j}z^{m_k}\right|^2\\
&\leq \frac{|z|^{2m_j}}{2r^{2m_j}}+\left|\sum_{k=1}^\infty a_{j k}r^{m_j}z^{m_k}\right|^2
\end{split}\nonumber
\end{equation}
which implies:
$$
\sum_{j\in\mathds{Z}^*}|f_{ j}(z)|^2\leq\sum_{j=1}^{\infty}\frac{|z|^{2m_j}}{r^{2m_j}}+2\sum_{j,k=1}^\infty \left|a_{j k}r^{m_j}z^{m_k}\right|^2.
$$
From the convergence of the RHS of \eqref{ddconv} to a real analytic function on $U$, there exists positive constant $H$ such that for any polycylinder $|z_\alpha|\leq\rho_\alpha$, $\alpha=1,\dots, n$ contained in $U$, we have $|a_{jk}|\leq H/(\rho^{m_j}\rho^{m_k})$, where $\rho=(\rho_1,\dots, \rho_n)$.
Choose an $n$-tuple $\rho'$ such that $0<\rho'<\rho$. Then for $|z_\alpha|\leq \rho'_\alpha$ one has:
$$
\sum_{j\in\mathds{Z}^*}|f_{ j}(z)|^2\leq\sum_{j=1}^{\infty}\frac{\rho'^{2m_j}}{r^{2m_j}}+2 H^2\sum_{j=1}^\infty \frac{r^{2m_j}}{\rho^{2m_j}}\sum_{k=1}^\infty\frac{\rho'^{2m_k}}{\rho^{2m_k}}.
$$
Fixing $r_j=\sqrt{\rho_j\rho'_j}$ we get:
$$
\sum_{j\in\mathds{Z}^*}|f_{ j}(z)|^2\leq\sum_{j=1}^{\infty}\frac{\rho'^{m_j}}{\rho^{m_j}}+2 H^2\sum_{j=1}^\infty \frac{\rho'^{m_j}}{\rho^{m_j}}\sum_{k=1}^\infty\frac{\rho'^{2m_k}}{\rho^{2m_k}},
$$
and thus:
$$
\sum_{j\in\mathds{Z}^*}|f_{ j}(z)|^2\leq\frac1{\prod_{j=1}^n(1-\rho'_j/\rho_j)}\left(1+\frac{2H^2}{\prod_{j=1}^n(1-\rho'^{2}/\rho_j^2)}\right)=R<\infty
$$
as wished.
\end{proof}

We end this chapter with the following result about Bochner's coordinates and \K\ immersions.

\begin{theor}[E. Calabi, \cite{Cal}]\label{bochnergraph}
Let $(M,g)$ be an $n$-dimensional real analytic submanifold of an $N$-dimensional K\"ahler manifold $(S,G)$ with real analytic metric. Then $g$ is a real analytic K\"ahler metric and 
if $z=(z_1,\dots, z_n)$ are Bochner's coordinates on $M$ with respect to a point $p\in M$, then there exist Bochner's coordinates on $S$ such that the immersion $f\!: M\f S$ is given in a neighbourhood of $p$ by a graph:
\begin{equation}
(z_1,\dots,z_n)\mapsto (z_1,\dots,z_n,f_1(z),\dots, f_{N-n}(z)),
\end{equation}
where for all $j=1,\dots, N-n$, $f_j$ is a holomorphic function with no terms of degree less than $2$.
\end{theor}
\begin{proof}
Observe first that since $S$ is a real analytic K\"ahler manifold, by Theorem \ref{indefcal} it admits a local K\"ahler immersion into the indefinite Hilbert space $E$ and $M$, being a real analytic submanifold of $S$, amits such local immersion as well. Thus, the metric $g$ must be a real analytic K\"ahler metric. Denote by $f\!:M\rightarrow S$ the immersion of $M$ in $S$ and let $z$ be Bochner's coordinates around a point $p\in M$. Up to performe a unitary transformation, one can choose Bochner's coordinates on $S$ centered at $f(p)$ and such that $f(z)=(f_1(z),\dots, f_N(z))$ with:
$$
f_j(z)=z_j+\sum_{k=n+1}^\infty a_{jk} z^{m_k},\quad \textrm{for}\ j=1,\dots, n;
$$
$$
f_j(z)=\sum_{k=n+1}^\infty a_{jk} z^{m_k},\quad \textrm{for}\ j=n+1,\dots, N.
$$
Since both the diastasis around $p$ and $f(p)$ satisfy \eqref{bochnercoordinates} in normal coordinates, by Prop. \ref{induceddiast} we have:
$$
\dd_p(z)=\sum_{\alpha=1}^n|z_\alpha|^2+\psi_{2,2}(z)=\dd^S_{f(p)}(f(z))=\sum_{\alpha=1}^N|f_\alpha(z)|^2+\psi_{2,2}(f(z)),
$$
and in particular from:
$$
|z_j+\sum_{k=n+1}^\infty a_{jk} z^{m_k}|^2=|z_j|^2+|\sum_{k=n+1}^\infty a_{jk} z^{m_k}|^2+\sum_{k=n+1}^\infty a_{jk} z^{m_k}\bar z_j+\sum_{k=n+1}^\infty a_{jk} \bar z^{m_k}z_j
$$
we get that $a_{jk}$ must vanish for any $j\leq n$.
\end{proof}

\section{Exercises}
\begin{ExerciseList}
\Exercise Consider the {\em Springer domain} defined by:
$$D=\left\{(z_0,\dots,z_{n-1})\in\C^n\ | \ \sum_{j=1}^{n-1}|z_j|^2<e^{-|z_0|^2}\right\},$$
with the K\"ahler metric $g$ described by the globally defined K\"ahler potential: 
$$
\Phi:=-\log\left(e^{-|z_0|^2}-\sum_{j=1}^{n-1}|z_j|^2\right).
$$
Prove that $\Phi$ is the diastasis function centered at the origin for $(D, g)$.
\Exercise 
Consider a circular bounded domain $\Omega$ of $\mathds{C}^3$ endowed with the metric $g_B$ described in a neighbourhood of the origin by the K\"ahler potential:
$$
\Phi_B=-3\log(1-|z_1|^2-2|z_2|^2-|z_3|^2+|z_1|^2|z_3|^2+|z_2|^4-z_1z_3\bar z_2^2- z_2^2\bar z_1\bar z_3).
$$
Prove that $\Phi_B$ is the diastasis function centered at the origin for $(\Omega,g_B)$.
\Exercise
 Consider on $\mathds{C}$ the metric $g$ whose associate K\"ahler form $\omega$ is given by:
$\omega=\left(4\cos(z-\bar z)-1\right)dz\wedge d\bar z$. Write the diastasis $D(z,w)$ associated to $\omega$.
\Exercise
Let $M$ be a complex manifold.
We say that a real analytic \K\ metric $g$
on $M$ is {\em projective-like} if for all points $p\in M$ the function $e^{-D_p}$ is globally defined on $M$ and $e^{-D_p(q)}=1$ implies $p=q$.
Prove that:
\begin{enumerate}
\item[$(a)\ $] the \K\ metric of a complex space form is projective-like;
\item[$(b)\ $] the \K\ metric of a \K\ manifold which admits an injective  \K\ immersion into a complex space form is projective-like.
\end{enumerate}
Finally, give an example of real analytic \K\ metric which is not projective-like.
\Exercise Prove that if a K\"ahler manifold $(M,g)$ admits a K\"ahler immersion $f$ into $\mathds{C}{\rm P}^N_b$ then for any $p\in M$, the diastasis $\dd_p(q)$ is real analytic, single valued and nonnegative over $M\setminus f^{-1}(H)$ where $H$ is a hyperplane in $\mathds{C}{\rm P}^N$.
\Exercise Prove that a K\"ahler manifold $(M,g)$ admits a K\"ahler immersion into $\mathds{C}{\rm H}^N_b$ then for any $p\in M$, the diastasis $\dd_p(q)$ is real analytic, single valued and nonnegative for any $q\in M$.
\end{ExerciseList}

\chapter{Calabi's criterion}

This chapter summarizes the work of E. Calabi \cite{Cal} about the existence of a K\"ahler immersion of a complex manifold into a finite or infinite dimensional complex space form. In particular, Calabi provides an algebraic criterion to find out whether a complex manifold admits or not such an immersion. 
Sections \ref{flatcriterion} and \ref{nonflatcriterion} are devoted to illustrate Calabi's criterion for K\"ahler immersions into the complex Euclidean space and nonflat complex space forms respectively.
In Section \ref{csfia} we discuss the existence of a K\"ahler immersion of a complex space form into another, which Calabi himself in \cite{Cal} completely classified as direct application of his criterion.

\section{K\"ahler immersions into the complex Euclidean space}\label{flatcriterion}
We describe now Calabi's criterion for K\"ahler immersion into the complex flat space $\mathds{C}^N$ (see Section \ref{csf}).
\begin{defin}\label{lkiflat}
We say that a complex manifold $(M,g)$ admits a local K\"ahler immersion into $\mathds{C}^N$ if given any point $p\in M$ there exists a neighbourhood $U$ of $p$ and a map $f\!:U\rightarrow \mathds{C}^N$ such that:
\begin{enumerate}
\item $f$ is holomorphic;
\item $f$ is isometric, i.e. $\dd^M_p(z)=\sum_{j=1}^N|f_j(p)-f_j(z)|^2$;
\item there exists $0<R<+\infty$ such that $\sum_{j=1}^N|f_j(z)|^2<R$.
\end{enumerate}
Further, we say that the immersion is \emph{full} if the image $f(M)$ is not contained in any proper linear subspace of $\mathds{C}^N$.
\end{defin}
As already remarked, recall that if there exists a K\"ahler immersion of a complex manifold $(M,g)$ into $\mathds{C}^N$, then the metric $g$ is forced to be a real analytic K\"ahler metric, being the pull--back via a holomorphic map of a real analytic K\"ahler metric.
Thus, consider a real analytic K\"ahler manifold $(M,g)$, fix a coordinate system $z=(z_1,\dots,z_n)$ with origin at $p\in M$ and denote by $\dd_0(z)$ the diastasis of $g$ at $p$. Define the matrix $(a_{jk})$ to be the $\infty\times \infty$ matrix of coefficients given by (\ref{powexdiastc}).
\begin{defin}
A real analytic K\"ahler manifold $(M,g)$ is \emph{resolvable} of rank $N$ at $p\in M$ if $(a_{jk})$ is positive semidefinite of rank $N$.
\end{defin}

Calabi's criterion for local K\"ahler immersion into $\mathds{C}^N$ can be stated as follows (cfr. \cite[pages 9, 18]{Cal}):
\begin{theor}\label{localcrit}
Let $(M,g)$ be a real analytic K\"ahler manifold. There exists a neighbourhood $U$ of a point $p$ that admits a K\"ahler immersion into $\C^N$ if and only if  $(M,g)$ is resolvable of rank at most $N$ at $p\in M$. Furthermore if the rank is exactly $N$, the immersion is full.
\end{theor}
\begin{proof}
Assume that there exists a K\"ahler immersion $f\!:U\rightarrow \mathds{C}^N$. Fixed local coordinates $(z_1,\dots,z_n)$ centered at $p\in U$, up to translate we can assume $f(p)=0$. By Prop. \ref{induceddiast}, the induced diastasis:
$$
{\rm D}^M_p(z)=\sum_{h=1}^N|f_h(z)|^2
$$
is real analytic on $U$.
By expanding $f_j(z)=\sum_{k=0}^\infty a^j_kz^{m_k}$ we get:
$$
{\rm D}^M_p(z)=\sum_{j,k=0}^\infty \sum_{h=1}^Na^h_j\bar a^h_k z^{m_j}\bar z^{m_k}.
$$
It follows that $(a_{jk})$ is the product of the $N\times \infty$ matrix $a^h_j$ with its transpose conjugate, and thus it is positive semidefinite and of rank at most $N$. 

Assume now that \eqref{powexdiastc} converges in a domain $U$ and that $(a_{jk})$ is positive semidefinite. Then, we can decompose $(a_{jk})=\sum_{h=1}^N a_j^h\bar a_k^h$, where for each $h$, $a^h=(a_j^h)$ is an infinite nonzero vector. Then, we can define the map $f=(f_1,\dots, f_N)$ as a formal power series:
$$
f_h(z)=\sum_{j=1}^\infty a_j^h z^{m_j}.
$$
Since for any $j$, $k=1,2,\dots$, $|a_{jk}|$ is bounded on any maximal polycylindrical domain, from:
$$
|a_j^h|^2\leq \sum_{h=1}^N|a_j^h|^2=|a_{jj}|,
$$
we get that also $|a_j^h|$ is bounded in such domain.
Since $U$ is a convergence domain of the power series \eqref{powexdiastc}, it is a union of its maximal polycylinders. Thus, the $f_h$ are holomorphic function on $U$ and by construction the sum of the square of their absolute values $\sum_{h=1}^N|f_h(z)|^2$ converges on $U$ to ${\rm D}^M_p(z)$.
\end{proof}

It follows directly by Th. \ref{localcrit} that all Stein manifolds with the induced metric are examples of resolvable manifolds of finite rank.

In order to state the global version of Calabi's criterion, we need two further results (cfr. \cite[pages 8, 11, 18]{Cal}):
\begin{theor}[Rigidity]\label{local rigidity}
If a neighbourhood $U$ of a point $p$ admits a full K\"ahler immersion into $\mathds{C}^N$, then $N$ is univocally determined by the metric and the immersion is unique up to rigid motions of $\mathds{C}^N$.
\end{theor}    
\begin{proof}
Let $(z_1,\dots, z_n)$ be a coordinate system on $U$ centered at $p$ and consider two full K\"ahler immersions:
$$
f\!:U\rightarrow \mathds{C}^N,\quad f(z)=(f_1(z),\dots, f_{N}(z)),
$$
$$f'\!:U\rightarrow \mathds{C}^{N'}, \quad f'(z)=(f'_1(z),\dots, f'_{N'}(z)).
$$
We can assume without loss of generality that $f(p)=f'(p)=0$.

Observe now that being $f$ holomorphic, for any $j=1,\dots, N$, $f_j(U)$ is not contained in a one dimensional real subspace of $\mathds{C}$. In fact, if it was so, we would have $f_j(z)=\overline{f_j(z)}$
and thus $f_j(z)$ would be a constant which is equal to zero since $f(0)=0$, contradicting the hypothesis of fullness. The same holds for $f'$.

Since $f^*(g_0)=f'^*(g_0)$, by Prop. \ref{induceddiast} we get:
\begin{equation}\label{auxdiast}
\dd(z,w)=\sum_{j=1}^N||f_j(z)-f_j(w)||^2=\sum_{j=1}^{N'}||f'_j(z)-f_j'(w)||^2.
\end{equation}
Consider $n+1$ points $p_0,p_1,\dots, p_n\in U$. Their images in $\mathds{C}^N$ are linearly dependent in a real sense if and only if the vectors $v_1=f(p_1)-f(p_0),\dots, v_n=f(p_n)-f(p_0)$ are, i.e. if and only if:
$$
\sum_{j=1}^n\alpha_jv_j=0
$$ 
for real constants $\alpha_j$ not all vanishing. Taking the norm, this is equivalent to require that:
$$
\sum_{j,k=1}^n\alpha_j\alpha_k\langle v_j, v_k\rangle=0,
$$
for not all vanishing $\alpha_j$, $\alpha_k$.
From:
$$
\langle v_j,v_j\rangle=|| f(p_j)-f(p_0)||^2=D(p_j,p_0),
$$
we get:
$$
\langle v_j, v_k\rangle=\frac12\left(D(p_0,p_j)+D(p_0,p_k)-D(p_j,p_k)\right),
$$
which means that we can write the condition of being linearly dependent in terms of the diastasis. In view of \eqref{auxdiast}, this means that the maximum number of linearly independent points in the images of $U$ through $f$ and $f'$ does depend on the metric on $U$ alone and thus, the fullness condition implies $N=N'$.

From \eqref{auxdiast} the two maps $f$ and $f'$ preserves distances and thus there exists a rigid motion $T$ of $\mathds{C}^N$ such that $f'(U)=Tf(U)$. Furthermore $T$ is unique, since  $f(U)$ and $f'(U)$ span linearly $\mathds{C}^N$ in the real sense.
It remains to show that $T$ is unitary. Since $f(0)=f'(0)=0$, the transformation $T$ can be written:
$$
f'_j(z)=\sum_{k=1}^Na_{jk}f_k(z)+\sum_{k=1}^Nb_{jk}\bar f_k(z),\quad j=1,\dots, N,
$$
i.e.:
$$
f'_j(z)-\sum_{k=1}^Na_{jk}f_k(z)=\sum_{k=1}^Nb_{jk}\bar f_k(z),\quad j=1,\dots, N,
$$
which implies both sides are constant and thus vanish. Then, $T$ can be written:
$$
f'_j(z)=\sum_{k=1}^Na_{jk}f_k(z),\quad j=1,\dots, N,
$$
which is a complex linear transformation preserving distance and thus a unitary transformation of $\mathds{C}^N$. 
\end{proof}
\begin{theor}[Global character of resolvability]\label{gcr}
If a real analytic connected  K\"ahler manifold $(M,g)$ is resolvable of rank $N$ at a point $p\in M$, then it also is at any other point.
\end{theor}
\begin{proof}
We will prove that the set of resolvable points in $M$ is open and closed.

The set of resolvable points of rank $N$ is open since a point $p$ is resolvable of rank $N$ if and only if there exists a neighbourhood $ U\ni p$ which admits a K\"ahler immersion $f$ into $\mathds{C}^N$. Since the points in $f(U)$ spans $\mathds{C}^N$  (see the proof of the previous theorem), it follows that any other point in $U$ is resolvable of rank exactly $N$.

In order to prove it is also closed, let $p$ be a limit point of the set of resolvable points of rank $N$ in $M$. By Theorem \ref{indefcal} there exists a neighbourhood $V$ of $p$ admitting a K\"ahler immersion into the indefinite Hilbert space $E$. Since $p$ is a limit point of the set of resolvable points, there exist also $p'\in V$ and a neighbourhood $V'\subset V$ of $p'$ such that $V'$ admits a K\"ahler immersion into $\mathds{C}^N$. Let $z$ be a coordinate system defined on $V$ with origin at $p'$ and denote by $f'=(f'_j)_{j=1,\dots ,N}$ the K\"ahler immersion $f'\!:V'\rightarrow\mathds{C}^N$ and by $f=(f_j)_{j\in \mathds{Z}^*}$ the K\"ahler immersion $f\!:V\rightarrow E$. Assume also that $f'(0)=0=f(0)$. Observe that the diastasis of $E$ restricted to the subspace $E'$ spanned by $f(V')$ (and thus by $f(V)$) is positive semidefinite, in the sense that for any $q\in E'$, $\dd^E(f(p'),q)\geq 0$. In fact, we can write the vector $v=q-f(p')$ as a linear combination of vectors $v_j=f(p_j)-f(p')$ with $p_j\in V'$, let us say $v=\sum_{j=1}^k\alpha_jv_j$. Then, we consider $v'+f'(p')=q'\in \mathds{C}^N$ where $v'=\sum_{j=1}^k\alpha_jv'_j$, for $v'_j=f'(p_j)-f'(p')$, and by Prop. \ref{induceddiast}, we have:
\begin{equation}
\begin{split}
\dd^E_0(q)=&\frac12\sum_{i,j=1}^k\alpha_j\alpha_k\left(\dd^E(f(p'),f(p_j))+\dd^E(f(p'),f(p_k))-\dd^E(f(p_j),f(p_k))\right)\\
=&\frac12\sum_{i,j=1}^k\alpha_j\alpha_k\left(\dd^{\mathds{C}^N}(f'(p'),f'(p_j))+\dd^{\mathds{C}^N}(f'(p'),f'(p_k))-\dd^{\mathds{C}^N}(f'(p_j),f'(p_k))\right)\\
=&\frac12\sum_{i,j=1}^k\alpha_j\alpha_k\sum_{\sigma=1}^N\left(|f'_\sigma(p_j)|^2+|f'_\sigma(p_k)|^2-|f_\sigma'(p_j)-f_\sigma'(p_k)|^2\right)\\
=&\sum_{\sigma=1}^N\left|\sum_{j=1}^k\alpha_jf_\sigma'(p_j)\right|^2\geq 0,
\end{split}\nonumber
\end{equation}
as wished. Denote now by $E_0$ the subspace of $E'$ defined by:
$$
E_0=\left\{y\in E'|\, \sum_{j\in \mathds{Z}^*}(\sgn j)|y_j|^2=0\right\},
$$
(where $y$ are coordinates in $E$) and by $E_+$ the orthogonal complement of $E_0$ with respect to the Hermitian form $\sum_{j\in \mathds{Z}^*}|y_j|^2$ (if necessary $E'$, $E_0$ and $E_+$ should be replaced by their completion with respect to that metric). On $E_+$ the diastasis is positive definite, and further the orthogonal projection $T$ of $E'$ onto $E_+$ has the effect of preserving the diastasis of all pairs of points. Thus, the map defined by $T\circ f\!:V\rightarrow E_+$ is an isometric immersion of $V$ into a unitary space. Finally, since $V$ span the same linear space as $V'$, we have that $E_+$ must be of dimension $N$. Thus, every points in $V$, and in particular $p$, must be resolvable of rank $N$ concluding the proof.
\end{proof}
Last theorem states that if a local K\"ahler immersion around a point $p\in M$ exists, then the same is true for any other point. Due to this result, we can say that a manifold is resolvable without specifying the point. 

The following result states that if $M$ is chosen to be simply connected, then it is possible to extend the local immersion to the whole manifold (cfr. \cite[pages 12-13]{Cal}):
\begin{theor}[Calabi's criterion for simply connected manifolds]\label{globalcriterion}
A simply connected complex manifold $(M,g)$ admits a K\"ahler immersion into $\C^N$ if and only if the metric is real analytic and $M$ is resolvable of rank at most $N$. Furthermore, the immersion is full if and only if the rank is exactly $N$.
\end{theor}
\begin{proof}
The conditions are necessary in view of Theorem \ref{localcrit}. Thus, assume that $M$ is a real analytic K\"ahler manifold resolvable of rank $N$. By Theorem \ref{localcrit} for any point $p\in M$ there exists a neighbourhood $U\ni p$ admitting a K\"ahler immersion into $\mathds{C}^N$ in such a way that the image of $U$ spans linearly $\mathds{C}^N$. Let $\{U_j\}$ be an open covering of $M$ such that each $U_j$ admits a K\"ahler immersion into $\mathds{C}^N$. Let $p_0$ be the origin in $U_0$ and let $f\!:U_0\rightarrow \mathds{C}^N$, $f=(f_j)$, be a K\"ahler immersion. For any point $p\in M$, consider a path connecting $p_0$ and $p$ and denote by $\pi_0$, $\pi_1,\dots, \pi_k$ the overlapping open segments obtained as intersection between the path and the $U_j$'s.\\

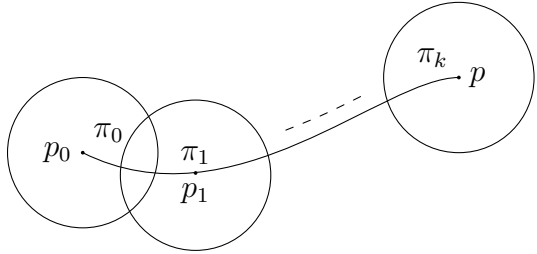
\begin{tikzpicture}[line cap=round,line join=round,x=1.0cm,y=1.0cm]
\draw (-3,0) .. controls (-1,-1) and (1,1) .. (2,1);
\draw [fill=black] (-3,0)circle (0.5pt) node[anchor= east] {$p_0$};
\draw [fill=black] (2,1)circle (0.5pt) node[anchor= west] {$p$};
\draw (-3,0)circle (1cm);
\draw (-3,0) node[anchor=south west] {$\pi_0$};

\draw (2,1)circle (1cm);
\draw (2,1) node[anchor=south east] {$\pi_k$};

\draw [fill=black] (-1.5,-0.27)circle (0.5pt) node[anchor= north] {$p_1$};
\draw (-1.5,-0.3)circle (1cm);
\draw (-1.5,-0.3) node[anchor=south ] {$\pi_1$};

\draw[dash pattern=on 3pt off 4pt](-0.3,0.3).. controls (0.2,0.5).. (0.8,0.8);
\end{tikzpicture}

At each overlap $\pi_j\cap \pi_{j+1}$, there exists a unique unitary motion which transforms the immersion functions of $\pi_j$ into those of $\pi_{j+1}$. If we apply the transformation to all the neighbourhood of $\pi_j$, we get the same K\"ahler immersion for both $\pi_j$ and $\pi_{j+1}$. By induction, we can extend the K\"ahler immersion around $p_0$ to a K\"ahler immersion around the whole path between $p_0$ and $p$. Since $M$ is arcwise connected, the K\"ahler immersion can be extended to the whole manifold and since it is simply connected, the extension does not depend on the path chosen.
\end{proof}

\begin{cor}
If a simply connected K\"ahler manifold $(M,g)$ is resolvable of any rank, then its diastasis $\dd(p,q)$ can be extended to all pairs of points and its everywhere nonnegative.
\end{cor}
\begin{proof}
Since the diastasis of the complex Euclidean space can be extended to all pairs of points and its everywhere nonnegative, the statement is an immediate consequence of the previous theorems and of Prop. \ref{induceddiast}.
\end{proof}

\begin{theor}\label{critcn}
A complex manifold $M$ endowed with a \K\ metric $g$ admits a K\"ahler immersion into $\mathds{C}^N$ if and only if the following conditions are fulfilled:
\begin{itemize}
\item[($i$)] $g$ is  real analytic K\"ahler metric,
\item[($ii$)]  $(M, g)$ is resolvable of rank at most $N$,
\item[($iii$)] for each point $p\in M$ the analytic extension of the diastasis $\dd_p$ is single valued.
\end{itemize}
Further, the immersion is also injective if and only if for any point $p\in M$:
\begin{itemize}
\item[($iv$)] $\dd_p(q)=0$ only for $q=p$.
\end{itemize}
\end{theor}
\begin{proof}
By Theorem \ref{globalcriterion}, conditions ($i$) and ($ii$) are necessary and sufficient for the universal covering $\pi\!:\tilde M\rightarrow M$ to admit a K\"ahler immersion $f\!:\tilde M\rightarrow \mathds{C}^N$. The necessity of condition ($iii$) follows directly from Prop. \ref{induceddiast}. In order to prove it is also sufficient for $f$ to descend to the quotient, fix $p\in M$ and consider the analytic extension $\tilde \dd_p$ of $ \dd_p$ to the whole $\tilde M$. If $\dd_p$ is single valued on $M$, $\dd_p\circ \pi=\tilde \dd_p$ implies that $\tilde \dd_p(q)=0$ for any point $q\in \tilde M$ that belongs to the same fibre of $p$. Hence, the Poincar\'e group of $M$ acting on $\tilde M$ leaves the image of $\tilde M$ in $\mathds{C}^N$ pointwise fixed and the map $f$ descends to a globally defined K\"ahler map $M\rightarrow \mathds{C}^N$.   
Finally, condition ($iv$) is equivalent for the immersion to be injective since by Prop. \ref{induceddiast}  $f(p)=f(q)$ only for $p=q$.
\end{proof}

\section{K\"ahler immersions into nonflat complex space forms}\label{nonflatcriterion}
Let ${\rm F}(N,b)$ be an $N$-dimensional complex space form of holomorphic sectional curvature $4b$ and denote by $\dd^b$ its diastasis function, described in Section \ref{csf}. 
The following definition generalizes Def. \ref{lkiflat} to the case when $b\neq 0$:
\begin{defin}\label{lkiflat2}
We say that a complex manifold $(M,g)$ admits a local K\"ahler immersion into ${\rm F}(N,b)$ if given any point $p\in M$ there exists a neighbourhood $U$ of $p$ and a map $f\!:U\rightarrow {\rm F}(N,b)$ such that:
\begin{enumerate}
\item $f$ is holomorphic;
\item $f$ is isometric, i.e., due to Prop. \ref{induceddiast}, $\dd^M_p(z)=\dd^b_{f(p)}(f(z))$;
\item there exists $0<R<+\infty$ such that $\sum_{j=1}^N|f_j(z)|^2<R$.
\end{enumerate}
Further, we say that the immersion is \emph{full} if the image $f(M)$ is not contained in any proper totally geodesic submanifold of ${\rm F}(N,b)$.
\end{defin}

We introduce here a {\em generalized stereographic projection} performed from a complex space form of nonzero curvature to the complex euclidean space. Let $p$ be a point in ${\rm F}(N,b)$ and set normal coordinates $z$ in a neighbourhood $U$ centered at $p$. The generalized stereographic projection is the map:
$$
\pi\!:U\rightarrow\mathds{C}^N,\quad \pi(z)=\left(\pi_1(z),\dots,\pi_N(z)\right)
$$
which satisfies $\dd^b_0(z)= \dd^0_0(\pi(z))$:

\begin{center}
\begin{tikzpicture}[line cap=round,line join=round,x=1.0cm,y=1.0cm]
\draw[->] (0.5,0)--(0.5,-1);
\draw[->] (0.8,0.25)--(2,0.25);
\draw[->] (0.9,-1)--(2.1,0);

\draw (0.3,-0.5) node {$\pi$};
\draw (1.3,0.6) node {$\dd^b_0$};
\draw (1.7,-0.8) node {$\dd^0_0$};

\draw (-0.35,0) node[anchor=south] {${\rm F}(N,b)\supset U$};
\draw (0.6,-1) node[anchor=north] {$\mathds{C}^N$};
\draw (2,0)node[anchor=south west] {$\mathds{R}$};
\end{tikzpicture}
\end{center}
i.e. such that:
$$
\sum_{j=1}^N|\pi_j(z)|^2=\frac1b\left(e^{b\dd^b_0(z)}-1\right).
$$

Consider now a real analytic K\"ahler manifold $(M,g)$ and fix a coordinate system $(z_1,\dots,z_n)$ with origin at $p\in M$. Recall that as for the case of flat ambient space, chosing a real analytic K\"ahler manifold is not restrictive since if there exists a K\"ahler immersion of a complex manifold $(M,g)$ into ${\rm F}(N,b)$, then the metric $g$ is forced to be a real analytic K\"ahler metric, being the pull--back via a holomorphic map of the real analytic K\"ahler metric $g_b$. Denote by $\dd_0(z)$ the diastasis of $g$ at $p$ and consider the power expansion around the origin of the function $(e^{b\dd_0(z)}-1)/b$:
\begin{equation}
\frac{e^{b\dd_0(z)}-1}{b}=\sum_{j,k=0}^\infty s_{jk}\,z^{m_j}\bar z^{m_k}.\nonumber
\end{equation}
\begin{defin}\label{bresolvabledefin}
A real analytic K\"ahler manifold $(M,g)$ is \emph{$b$-resolvable} of rank $N$ at $p\in M$ if the matrix $(s_{jk})$ is semipositive definite of rank $N$.
\end{defin}
In particular, $(M,g)$ is $1$-resolvable of rank $N$ at $p$ if  the matrix of coefficients $(b_{jk})$ given by:
\begin{equation}\label{powexdiastcp}
e^{\dd_0(z)}-1=\sum_{j,k=0}^\infty b_{jk}\,z^{m_j}\bar z^{m_k},
\end{equation}
is positive semidefinite of rank $N$. Similarly $(M,g)$ is $-1$-resolvable of rank $N$ at $p$ if the matrix of coefficients $(c_{jk})$ given by
\begin{equation}\label{powexdiastch}
1-e^{-\dd_0(z)}=\sum_{j,k=0}^\infty c_{jk}\,z^{m_j}\bar z^{m_k}.
\end{equation}
is positive semidefinite of rank $N$.
\begin{remark}\label{bresolvproj}\rm 
Observe that a K\"ahler manifold $(M,g)$ is $b$-resolvable of rank $N$ at $p\in M$ if and only if the diastasis:
$$
\dd'_0(\pi(z))=\frac1b\left(e^{b\dd_0(z)}-1\right),
$$ 
obtained from the diastasis $\dd_0(z)$ of $g$ after a stereographic projection $\pi$ with respect to $p_0$, is resolvable of rank $N$ at $p$.
\end{remark}


Calabi's criterion for local K\"ahler immersion can be stated as follows (cfr. \cite[pages 9, 18]{Cal}):
\begin{theor}\label{localcritb}
Let $(M,g)$ be a real analytic K\"ahler manifold. There exists a neighbourhood $V$ of a point $p$ that admits a K\"ahler immersion into  ${\rm F}(N,b)$ if and only if  $(M,g)$ is $b$-resolvable of rank at most $N$ at $p\in M$. Furthermore if the rank is exactly $N$, the immersion is full.
\end{theor}
\begin{proof}
Consider local coordinates $z$ around $p$ and denote by $\dd_0$ the diastasis of $M$ centered at $p$. By Prop. \ref{induceddiast}, there exists a K\"ahler immersion $f\!:V\rightarrow {\rm F}(N,b)$ of a neighbourhood $V$ of $p$ into ${\rm F}(N,b)$ if and only if $\dd_0(z)=\frac1b\log\left(1+b\sum_{j=1}^N|f_j(z)|^2\right)$. Taking a stereographic projection with respect to $p$, we get that this is equivalent to have $\dd'_0(z)=\sum_{j=1}^N|f_j(z)|^2$, which is in turn equivalent by Theorem \ref{localcrit} to $M$ being resolvable of rank $N$ at $p$. The following diagram summarizes this setting:
\begin{center}
\begin{tikzpicture}[line cap=round,line join=round,x=1.0cm,y=1.0cm]
\draw[->] (-2.9,0.25)--(-1.7,0.25);
\draw[->] (-1,0)--(-1,-0.7);
\draw[->] (1.3,0.25)--(2.5,0.25);
\draw[->] (-0.4,-1)--(2.55,0.05);

\draw (-2.2,0.6) node {$f$};
\draw (-3.1,0.05) node[anchor=south] {$V$};

\draw[->] (-2.9,0.6).. controls (-1.8,2) and (1.5,2).. (2.6,0.6);
\draw[->] (-2.9,-0.1).. controls (-1.8,-2) and (1.5,-2).. (2.7,-0.1);

\draw (-0.8,-0.3) node {$\pi$};
\draw (1.8,0.6) node {$\dd^b_0$};
\draw (1,-0.8) node {$\dd^0_0$};
\draw (0,-1.9) node {$\dd_0'=\dd^0_0\circ \pi\circ f$};


\draw (-0.2,2) node {$\dd_0=\dd^b_0\circ f$};

\draw (-0.1,0) node[anchor=south] {$f(V)\subset {\rm F}(N,b)$};
\draw (-0.9,-0.7) node[anchor=north] {$\mathds{C}^N$};
\draw (2.5,0)node[anchor=south west] {$\mathds{R}$};
\end{tikzpicture}
\end{center}

Conclusion follows since by Remark \ref{bresolvproj}, $(M,g)$ is $b$-resolvable at $p$ if and only if its projected diastasis $\dd'_0$ is resolvable at $p$.
\end{proof}
In particular, a neighbourhood $V\ni p$ of $(M,g)$ admits a K\"ahler immersion into $\CP^N$ (resp. $\CH^N$), if and only if $M$ is $1$-resolvable (resp. $-1$-resolvable) of rank at most $N$ at $p$.
Hermitian symmetric spaces of compact type are examples of $1$-resolvable manifolds of finite rank. This follows from Th. \ref{localcritb} and the existence of a K\"ahler immersion of such spaces into the finite dimensional complex projective space, well--known since the work of Borel and Weil (see \cite{loidiast} or \cite{tak} for a proof).\\

In order to state the global version of Calabi's criterion, we need two further results analogous to Theorem \ref{local rigidity} and Theorem \ref{gcr} respectively (cfr. \cite[page 18]{Cal}):
\begin{theor}[Rigidity]\label{local rigidityb}
If a neighbourhood $V$ of a point $p$ admits a full K\"ahler immersion into ${\rm F}(N,b)$, then $N$ is univocally determined by the metric and the immersion is unique up to rigid motions of ${\rm F}(N,b)$.
\end{theor}    
\begin{proof}
Let $V$ be a neighbourhood of $p$ admitting two full K\"ahler immersions $f\!:V\rightarrow {\rm F}(N,b)$ and $f'\!:V\rightarrow {\rm F}(N',b)$. Let $\pi\!:f(V)\rightarrow \mathds{C}^N$ and $\pi'\!:f'(V)\rightarrow \mathds{C}^{N'}$ be stereographic projections with respect to $f(p)$ and $f'(p)$ respectively. Since: 
$$
\dd'_0(z)=\sum_{j=1}^N|f_j(z)|^2=\sum_{j=1}^{N'}|f'_j(z)|^2,
$$
$\pi\circ f$ and $\pi'\circ f'$ are two K\"ahler immersions of $V$ into $\mathds{C}^N$ and $\mathds{C}^{N'}$ respectively, with the same metric induced by $\dd'_0(z)$, and thus the proof reduces to that of Theorem \ref{local rigidity}.
%
%
\end{proof}
\begin{theor}[Global character of $b$-resolvability]
If a real analytic connected K\"ahler manifold $(M,g)$ is  $b$-resolvable of rank $N$ at a point $p\in M$, then it also is at any other point.
\end{theor}
\begin{proof}
Similarly to the case of flat ambient space, we will prove that the set of $b$-resolvable points in $M$ is open and closed. It is open since by Theorem \ref{localcritb} the $b$-resolvability is equivalent to the existence of a local K\"ahler immersion. It is also closed, in fact let $p$ be one of its limit points and let $V$ be a small enough neighbourhood around $p$ such that $\dd(q,q')$ is real analytic and single valued for any $q$, $q'\in V$. Let $p'$ be a $b$-resolvable point in $V$. By Theorem \ref{localcritb} there exist a neighbourhood $V'$ of $p'$ and a K\"ahler immersion $f\!:V'\rightarrow {\rm F}(N,b)$. Define a second diastasis on $V$ by:
\begin{equation}\label{ddprojected}
\dd'(q,q'):=\frac{e^{b\dd(q,q')}-1}b.
\end{equation}
Observe that on $V'$, $\dd'(q,q')$ coincides with the stereographic projection of $\dd(q,q')$ and thus at $p'$ the metric induced by $\dd'(q,q')$ is resolvable of rank $N$. By Theorem \ref{gcr}, $V$ with the metric induced by $\dd'(q,q')$ is resolvable of rank $N$ at any of its points. By Theorem \ref{critcn}, since in addition $\dd'(q,q')$ is single valued on $V$, the immersion $f$ can be extended to a K\"ahler immersion of the whole $V$ with the metric induced by $\dd'(q,q')$ into $\mathds{C}^N$. Thus, $f$ maps isometrically $V$ into ${\rm F}(N,b)$. The proof is complete by observing that when $b<0$, from \eqref{ddprojected} one gets $\dd'(q,q')<-\frac1b$, and thus:
$$
\sum_{j=1}^N|f_j(q)-f_j(q')|^2<-\frac1b,
$$
implies that the image of $V$ is actually contained in $\mathds{C}{\rm H}_b^N$.
\end{proof}
The previous theorem states that if a local K\"ahler immersion into ${\rm F}(N,b)$ around a point $p\in M$ exists, then the same is true for any other point. Due to this result we can say that a manifold is $b$-resolvable without specifying the point. In particular, if $(M,g)$ is $1$-resolvable, we say also that $g$ is \emph{projectively induced}. 

In complete analogy with the case of flat ambient space, the theorems just proven imply the following global criteria  (cfr. \cite[thms. 11-12, pages 19-20]{Cal}):
\begin{theor}\label{globalcriterionb}
A simply connected complex manifold $(M,g)$ admits a K\"ahler immersion into ${\rm F}(N,b)$, if and only if the metric $g$ is $b$-resolvable of rank at most $N$. Furthermore, if the immersion is full the rank is exactly $N$.
\end{theor}

\begin{theor}\label{critfb}
A complex manifold $(M,g)$ admits a K\"ahler immersion into ${\rm F}(N,b)$, if and only if the following conditions are fulfilled:
\begin{itemize}
\item[(i)] the metric is a real analytic K\"ahler metric,
\item[(ii)] the K\"ahler manifold $(M,g)$ is $b$-resolvable of rank at most $N$,
\item[(iii)] for each point $p\in M$ the analytic extension of the diastasis $\dd_p$ over $M$ is single valued.
\end{itemize}
Further, the immersion is also injective if and only if for any $p\in M$:
\begin{itemize}
\item[(iv)] $\dd_p(q)=0$ only for $q=p$.
\end{itemize}
\end{theor}

\begin{remark}\label{compactfull}\rm
Observe that, if it does exist, a K\"ahler immersion $f\!:M\rightarrow \mathds C{\rm P}^\infty$ of a {\em compact} K\"ahler manifold into $\mathds{C}{\rm P}^\infty$ is forced not to be full.
 In fact, assume by contradiction that $f\!:M\rightarrow \mathds C{\rm P}^\infty$ is a full K\"ahler immersion. Then we can write $f(p) = [s_0:\dots:s_j:\dots]$, where each $s_j$ is a global holomorphic section of the holomorphic line bundle $L$ on $M$ obtained as the pull-back of the hyperplane bundle of $\CP^{\infty}$. Since the map is full, the $s_j$'s are linearly independent and so the space of global holomorphic sections of $L$ is infinite dimensional. This is in contrast   with the  well known fact that   this space is finite dimensional  due to the compactness of $M$.
  Notice also that being the pull-back of the integral Fubini-Study form of $\mathds{C}{\rm P}^N$ through a holomorphic map, the induced K\"ahler form on $M$ is forced to be integral and so we are in the realm of algebraic geometry. It is worth pointing out that if we start with  compact \K\ manifold $(M, \omega)$ with $\omega$ integral then the Kodaira embedding
 $k:M\rightarrow \CP^N$  is a holomorphic map into some finite dimensional complex projective space $\CP^{N}$, but  in general $k$ is not isometric ($k^*\omega_{FS}$ cohomologous to $\omega$ but in general not equal).
 
\end{remark}

We conclude this section with the following example of K\"ahler metric admitting a local but not global immersion into $\mathds{C}{\rm P}^\infty$.
\begin{ex}\label{cstar}\rm
Consider the K\"ahler metric $\tilde g$ on $\C^*$ whose fundamental form is $$\tilde \omega=\frac{i}{2}\frac{dz\wedge d\bar z}{|z|^2}.$$ 
Since $\C$ admits a K\"ahler immersion $f_0\!:\C\f\CP^\infty$ into $\CP^\infty$ (cfr. Eq. (\ref{cincp}) below) and it covers $\C^*$ through the map $\exp\!:\C\f\C^*$, given by $\exp(z)=e^{2\pi i z}$, then a neighbourhood of each point of $\C^*$ can be K\"ahler immersed into $\CP^\infty$. The immersion cannot be extended to a global one. In fact, since $\exp^*(\tilde g)=g_0$, such K\"ahler immersion $f$ composed with $\exp$, would be a K\"ahler immersion of $\C$ into $\CP^\infty$. By Calabi's rigidity Theorem \ref{local rigidityb}, it would then exist a rigid motion $T$ of $\C$ such that $T\circ f_0=f\circ \exp$, that is impossible since $f_0$ is injective and $\exp$ is not.
\end{ex}

\section[K\"ahler immersions of a complex space form into another]{K\"ahler immersions of a complex space form into another}\label{csfia}

As application of its criterion Calabi studies the existence of K\"ahler immersion of a complex space form into another. 
Following the notations of the previous chapter, we will denote by ${\rm F}(N,b)$ a complex space form of dimension $N$ and  holomorphic sectional curvature $4b$.
\begin{theor}[E. Calabi,  {\cite[pages 21-22]{Cal}}]\label{immspaceform}
A complex space form ${\rm F}(n,b)$ admits a global K\"ahler immersion into ${\rm F}(N,b')$ if an only if $b\leq b'$ and
\begin{itemize}
\item[either] $b\leq 0$ and $N=\infty$,
\item[or] $b'=kb$ for some positive integer $k$, and $N\geq {n+k \choose k}-1$.
\end{itemize}
\end{theor}
\begin{proof}
Assume first $b\neq 0$. By Theorem \ref{globalcriterionb} it is enough to check for what values of $b$, $g_{b}$ is $b'$-resolvable. Fix a point $p\in {\rm F}(n,b)$ and consider local coordinates $z$ centered at $p$. Then:
$$
\dd_0^b=\frac1b\log\left(1+b\sum_{j=1}^n|z_j|^2\right),
$$
and the $(j,k)$ entry in the matrix $(a_{jk})$ of its power expansion around the origin \eqref{powexdiastc} reads:
$$
a_{jk}=\delta_{jk}\frac{(|m_j|-1)!}{m_j!}(-b)^{|m_{j}|-1}.
$$
Thus, $(a_{jk})$ is a diagonal matrix with nonvanishing elements on the diagonal. It follows that its rank is infinite and 
it is positive semidefinite if and only if each of its entries is nonnegative, i.e. if and only if $b<0$. It follows that ${\rm F}(n,b)$ admits a K\"ahler immersion into $\mathds{C}^N$ iff ${\rm F}(n,b)=\mathds{C}{\rm H}^n_b$ and $N=\infty$.
In order to check the $b'$-resolvability for $b'\neq 0$, consider that:
$$
\frac{e^{b'\dd_0^b}-1}{b'}=\frac{\left(1+b\sum_{j=1}^n|z_j|^2\right)^{\frac{b'}b}-1}{b'}=\frac{1}{b'}\sum_{j=1}^\infty{b'/b\choose j}\left(b\sum_{k=1}^n|z_k|^2\right)^{j},
$$
and thus the matrix $(s_{jk})$ of the coefficients of its power expansion around the origin is a diagonal matrix with terms on the diagonal given by:
$$
s_{jj}=\begin{cases}\frac{\prod_{l=1}^{|m_j|-1}(b'-lb)}{m_j!}&\ \textrm{for}\ |m_j|>1;\\
1&\ \textrm{otherwise}.
\end{cases}
$$ 
The rank $N$ of $(s_{jk})$ is the number of nonvanishing entries. When $b'$ is a multiple of $b$, i.e. there exists a positive integer $k$ such that $b'=kb$, then $N$ is equal to ${n+k\choose k}-1$. Since in this case each entry is nonnegative, $(s_{jk})$ is positive semidefinite.
When $b'$ is not a multiple of $b$, then the rank is $\infty$ and $(s_{jk})$ is $b'$-resolvable iff $b'-lb$ is nonnegative for any $l=2,3,\dots$, i.e. iff $b<0$.

Finally, the case $b=0$ is trivial when $b'=0$. For $b'\neq 0$ we get:
$$
s_{jj}=\frac{(b')^{|m_j|-1}}{m_j!}
$$ 
and thus $(s_{jk})$ is positive semidefinite if and only if $b'>0$ and the rank is infinite, i.e. the only nontrivial K\"ahler immersion $\mathds{C}^n$ admits is into $\mathds{C}{\rm P}^\infty_b$. 
\end{proof}
\begin{remark}\label{constantmulti}\rm
It is interesting to notice that a K\"ahler manifold $(M,\omega)$ is $b$-resolvable for $b>0$ (resp. $b<0$) if and only if $(M,b\,\omega)$ is $1$-resolvable (resp. $-1$-resolvable). To see this, notice that if we denote by $\varphi$ the immersion $\varphi\!:M\rightarrow \mathds{C}{\rm P}^N_b$, by Prop. \ref{induceddiast} we have:
$$
\dd^M_p(z)=\frac{1}{b}\log\left(1+b\sum_{j=1}^N|\varphi_j(z)|^2\right),
$$
thus the map $\sqrt b\,\varphi$ satisfies:
$$
(\sqrt b\,\varphi)^*\dd_{0}^b(z)=\log\left(1+b\sum_{j=1}^N| \varphi_j(z)|^2\right)=b\,\dd_p^M(z).
$$
Totally similar arguments apply to the $b<0$ case.
Finally, notice that the multiplication of the metric $g$ by $c$ is harmless  when one studies \K\ immersions into the infinite dimensional complex Euclidean space  $l ^2({\C})$ equipped with the flat metric $g_0$. In fact, if $f:M\rightarrow l ^2({\C})$ satisfies $f^*(g_0)=g$ then $(\sqrt{c}f)^*(g_0)=cg$.
\end{remark}
In the sight of the previous remark, Theorem \ref{immspaceform} can be stated in terms of K\"ahler immersions of $(\mathds{C}{\rm H}^n,c\,g_{hyp})$, $\mathds{C}^n$ and $(\mathds{C}{\rm P}^n,c\,g_{FS})$ into $\mathds{C}{\rm H}^{N\leq \infty}$,  $\mathds{C}^{N\leq \infty}$ or $\mathds{C}{\rm P}^{N\leq \infty}$ as follows. 

\begin{theor}\label{calabic}$\ $
\begin{enumerate}
\item For any $c>0$, $(\mathds{C}{\rm H}^n,c\,g_{hyp})$ admits a full K\"ahler immersion into $l^2(\mathds{C})$ and into $\mathds{C}{\rm P}^{\infty}$. Further, $(\mathds{C}{\rm H}^n,c\,g_{hyp})$ admits a K\"ahler immersion into $\mathds{C}{\rm H}^{N\leq \infty}$ if and only if $c\leq1$ and $N=\infty$.
\item  The flat space $\mathds{C}^n$ does not admit a K\"ahler immersion into $\mathds{C}{\rm H}^{N\leq \infty}$ for any value of $N$, but it does, full, into $\mathds{C}{\rm P}^{\infty}$. 
\item  For no value of $c>0$, $(\mathds{C}{\rm P}^n,c\,g_{FS})$ admits a K\"ahler immersion into $\mathds{C}{\rm H}^{N\leq \infty}$ nor $\mathds{C}^{N\leq \infty}$. Further, $(\mathds{C}{\rm P}^n,c\,g_{FS})$ admits a full K\"ahler immersion into $\mathds{C}{\rm P}^{N}$, $N<\infty$, if and only if $c$ is a positive integer and $N={n+c\choose c}-1$.
\end{enumerate}
\end{theor}

We conclude this chapter with the following theorems which show that if a  a K\"ahler manifold $(M,g)$ admits a K\"ahler immersion into $l^2(\mathds{C})$ (resp. $\CH^\infty$) then it also does into $\mathds{C}{\rm P}^\infty$ (resp. $l^2(\C)$). These facts has been firstly pointed out by S. Bochner in \cite{bochner}. 
\begin{theor}\label{bochnerth}
If a K\"ahler manifold $(M,g)$ admits a K\"ahler immersion into the infinite dimensional flat space $l^2(\C)$ then it also does into $\CP^\infty$.
\end{theor}
\begin{proof}
Fix a local coordinate system $(z_1,\dots,z_n)$ on a neighbourhood $U$ of $p\in M$. By Theorem \ref{chardiast} for some holomorphic functions $f_1,\dots,f_j,\dots$, the diastasis function for $g$ reads
$$\dd^M_0(z)=\sum_{j=1}^\infty|f_j|^2.$$
Let $\dd^M_0(z)=\log\psi$ with $\psi=e^{\dd^M_0(z)}$. Then for some suitable functions $h_j$, $j=1,2,\dots$ we get
$$\psi=1+\sum_{j=1}^\infty|h_j|^2,$$
and the conclusion follows.
\end{proof}
\begin{theor}\label{bochnerthalt}
If a K\"ahler manifold $(M,g)$ admits a K\"ahler immersion into the infinite dimensional hyperbolic space $\CH^\infty$ then it also does into $l^2(\C)$.
\end{theor}
\begin{proof}
Consider a local coordinate system $(z_1,\dots,z_n)$ on $M$ in a neighbourhood of $p\in M$ and let $\dd^M_0(z)$ be the diastasis function for $g$ at $p$. By Theorem \ref{chardiast}, there exists $f_1,\dots,f_j,\dots$ holomorphic functions such that
$$\dd^M_0(z)=-\log\left(1-\sum_{j=1}^\infty|f_j|^2\right).$$
Hence
$$\dd^M_0(z)=\sum_{j=1}^\infty|h_j|^2,$$
for some suitable holomorphic functions $h_j$, $j=1,2,\dots$, and we are done.
\end{proof}
\section{Exercises}
\begin{ExerciseList}
\Exercise Prove that 
\begin{equation}\label{chinell}
f:\CH^n\hookrightarrow l^2({\C}):z\mapsto \left(\dots, \sqrt{\frac{(|m_j|-1)!}{m_j!}}z^{m_j}, \dots \right),
\end{equation}
is a full K\"ahler immersion of $\CH^n$ into $l^2({\C})$.
\Exercise Prove that 
\begin{equation}\label{chincp}
f:\CH^n\hookrightarrow \CP^{\infty}:z\mapsto \left(\dots , \sqrt{\frac{|m_{j}|!}{m_{j}!}}z^{m_{j}}, \dots \right),
\end{equation}
is a full K\"ahler immersion of $\CH^n$ into $\CP^{\infty}$.
\Exercise Prove that 
\begin{equation}\label{cincp}
f:{\C}^n\hookrightarrow \CP^{\infty}:z\mapsto \left(\dots , \sqrt{\frac{1}{m_{j}!}}z^{m_{j}}, \dots \right),
\end{equation}
is a full K\"ahler immersion of $\C^n$ into $\CP^{\infty}$.

\Exercise
Let $k$ be a positive integer.
Construct a full \K\ immersion of $(\CP^1, kg_{FS})$ into $\CP^k$ (cfr. 3. of Theorem \ref{calabic}).

\Exercise Consider the {\em Springer domain} defined by:
$$D=\left\{(z_0,\dots,z_{n-1})\in\C^n\ | \ \sum_{j=1}^{n-1}|z_j|^2<e^{-|z_0|^2}\right\},$$
with the K\"ahler metric $g$ described by the globally defined K\"ahler potential: 
$$
\Phi:=-\log\left(e^{-|z_0|^2}-\sum_{j=1}^{n-1}|z_j|^2\right).
$$
Prove that $(D,g)$ admits a full K\"ahler immersion into $l^2(\mathds{C})$.
\Exercise For $\alpha>0$, consider:
$$D=\left\{(z_0,\dots,z_{n-1})\in\C^n\ | \ \sum_{j=1}^{n-1}|z_j|^2<\frac{\alpha}{|z_0|^2+\alpha}\right\},$$
endowed with the K\"ahler metric $g$ described by the globally defined K\"ahler potential: 
$$
\Phi:=-\log\left(\frac{\alpha}{|z_0|^2+\alpha}-\sum_{j=1}^{n-1}|z_j|^2\right).
$$
Prove that $(D,g)$ admits a full K\"ahler immersion into $\mathds{C}{\rm P}^\infty$.
\Exercise Let:
$$D=\left\{(z_0,\dots,z_{n-1})\in\C^n\ | \ \sum_{j=1}^{n-1}|z_j|^2<\frac{1}{\sqrt{|z_0|^2+1}}\right\},$$
with the K\"ahler metric $g$ described by the globally defined K\"ahler potential: 
$$
\Phi:=-\log\left(\frac{1}{\sqrt{|z_0|^2+1}}-\sum_{j=1}^{n-1}|z_j|^2\right).
$$
Prove that $(D,g)$ does not admit a K\"ahler immersion into any complex space form.
\Exercise Consider a circular bounded domain $\Omega$ of $\mathds{C}^3$ endowed with the metric $g_B$ described in a neighbourhood of the origin by the K\"ahler potential:
$$
\Phi_B=-3\log(1-|z_1|^2-2|z_2|^2-|z_3|^2+|z_1|^2|z_3|^2+|z_2|^4-z_1z_3\bar z_2^2- z_2^2\bar z_1\bar z_3).
$$
Prove that $(\Omega,g_B)$ does not admit a K\"ahler immersion into $l^2(\mathds{C})$.

\Exercise\label{constantell2c}
Theorem \ref{bochnerth} combined with Remark \ref{constantmulti}, implies that if a K\"ahler manifold $(M,g)$ admits a K\"ahler immersion into $l^2(\mathds{C})$, then $(M,cg)$ does into $\mathds{C}{\rm P}^{\infty}$, for any value of $c>0$.
Shows that the converse is true, namely that if K\"ahler manifold $(M,cg)$ admits a local K\"ahler immersion into $\CP^\infty$ for all $c>0$ then $(M,g)$ does into $l^2(\C)$.
\Exercise Prove that if a K\"ahler metric $g$ is projectively induced  the same is true for $kg$, for any positive integer $k$.
\Exercise\label{bochnerHp} Let $(M,g)$ be a K\"ahler manifold and let $f\!:M\rightarrow \CP^N$, $N\leq \infty$, be a K\"ahler immersion. Prove that the Bochner coordinates around a point $p\in M$ can be extended over $M\setminus f^{-1}(H_p)$, where $H_p$ is the hyperplane at infinity with respect to $f(p)$.
\end{ExerciseList}

\chapter{Homogeneous K\"ahler manifolds} \label{simmes}
In this chapter we survey what is known about the existence of K\"ahler immersions of homogeneous K\"ahler manifolds into complex space forms.
Recall that a homogeneous K\"ahler manifold  is a K\"ahler manifold on which the group of holomorphic isometries
$\Aut (M)\cap \isom (M, g)$ acts transitively on $M$ (here $\Aut (M)$ denotes the group of biholomorphisms of $M$). 

In the first two sections we summarize the results of A. J. Di Scala, H. Ishi and A. Loi \cite{ishi} about K\"ahler immersion of homogeneous K\"ahler manifolds into complex Euclidean and hyperbolic spaces. Section \ref{hbdcpn} is devoted to proving that the only homogeneous bounded domains which are projectively induced for all positive multiples of their metrics are given by the product of complex hyperbolic spaces. This result, combined with the solution of J. Dorfmeister and  K. Nakajima \cite{DN88} of the fundamental conjecture on homogeneous K\"ahler manifolds (Theorem \ref{FC}), will be applied in Section \ref{hkmcch} to classify homogeneous K\"ahler manifolds admitting a K\"ahler immersion into $\CH^N$ or $\C^N$, $N\leq \infty$ (Theorem \ref{Flat-Case}).

In the last three sections we consider K\"ahler immersions of homogeneous K\"ahler manifolds into $\mathds{C}{\rm P}^N$, $N\leq \infty$. The general case is discussed in Section \ref{hkmcp}, while in sections \ref{bergmansection} and \ref{symmcp} we detail the case of K\"aher immersions of bounded symmetric domains  into $\CP^\infty$.

\section[A result about \K\ immersions of h.b.d. into $\CP^{\infty}$]{A result about \K\ immersions of homogeneous bounded domains into $\CP^{\infty}$}\label{hbdcpn}
We have already noticed in Section \ref{csfia} that the complex Euclidean space $(\C^n, \lambda g_0)$ and the complex  hyperbolic space $(\CH^n, \lambda g_{hyp})$ both admit a \K\ immersion into $\CP^{\infty}$, for all $\lambda >0$.
In the following theorem we prove that this  fact characterizes these two spaces among all homogeneous bounded domains.
Recall that a {\em homogeneous bounded domain} $(\Omega, g)$  is a  bounded domain  (i.e. a connected open set)  $\Omega\subset\mathds{C}^n$ such that $(\Omega, g)$ is a homogeneous K\"ahler manifold. Recall also that we say that a K\"ahler manifold is {\em projectively induced} when it is $1$-resolvable in the sense of Definition \ref{bresolvabledefin}, i.e. when it does admit a local K\"ahler immersion into $\CP^\infty$.

This  theorem will be one of the key ingredients for the study of \K\ immersions of homogeneous K\"ahler manifolds into finite or infinite dimensional complex space forms.

\begin{theor}[A. J. Di Scala, H. Ishi, A. Loi \cite{ishi}]\label{thmsmall}
Let $(\Omega, g)$ be an $n$-dimensional homogeneous bounded domain.
The metric $\lambda g$ is  projectively induced for all
$\lambda >0$
if and only if:
\begin{equation} \label{eqn:direct_prod}
(\Omega, g) =\left(\CH^{n_1}\times\cdots \times \CH^{n_\rk}, \lambda_1g_{hyp}\oplus\cdots\oplus{\lambda_\rk}g_{hyp}\right),
\end{equation}
where $n_1+\cdots +n_r=n$,  $\lambda_j$, $j=1,\dots , \rk$
are positive real numbers.
\end{theor}
\begin{proof}
First we find a global potential for the homogeneous K\"ahler metric $g$
on the domain $\Omega$ following Dorfmeister \cite{D85}.
By \cite[Theorem 2 (c)]{D85},
there exists a split solvable Lie subgroup
$S \subset \mathrm{Aut}(\Omega, g)$
acting simply transitively on the domain $\Omega$.
Taking a reference point $z_0 \in \Omega$,
we have a diffeomorphism
$S \owns s \overset{\sim}{\mapsto} s \cdot z_0 \in \Omega$,
and
by the differentiation, we get the linear isomorphism
$\gs := \mathrm{Lie}(S) \owns X
\overset{\sim}{\mapsto} X \cdot z_0 \in T_{z_0}\Omega \equiv \C^n$.
Then the evaluation of the K\"ahler form $\omega$ on $T_{z_o}\Omega$
is given by
$\omega(X\cdot z_o, Y \cdot z_0) = \beta([X,Y])\,\,
(X, Y \in \gs)$
with a certain linear form $\beta \in \gs^*$.
Let $j : \gs \to \gs$ be the linear map defined in such a way that
$(jX) \cdot z_0 = \sqrt{-1} (X \cdot z_0)$ for $X \in \gs$.
We have
$\Re g(X \cdot z_0,\,Y \cdot z_0) = \beta([jX, Y])$
for $X, Y \in \gs$,
and the right-hand side defines a positive inner product on $\gs$.
Let $\mathfrak{a}$ be the orthogonal complement of $[\gs, \gs]$ in $\gs$
with respect to the inner product.
Then $\mathfrak{a}$ is a commutative Cartan subalgebra of $\gs$.
Define $\gamma  \in \mathfrak{a}^*$ by
$\gamma(C) := -4 \beta (jC)\,\,\,(C \in \mathfrak{a})$,
where we extended $\gamma$ to $\gs = \mathfrak{a} \oplus [\gs, \gs]$ by the zero-extension.
Keeping the diffeomorphism between $S$ and $\Omega$ in mind,
we define a positive smooth function $\Psi$ on $\Omega$ by:
$$ \Psi((\exp X) \cdot z_0) = e^{-\gamma(X)} \,\,\, (X \in \gs). $$
From the argument in \cite[pages 302--304]{D85},
we see that:
\begin{equation} \label{eqn:globalpotential}
\omega = \frac{i}{2}\partial\bar\partial\log \Psi.
\end{equation}
It is known that
there exists a unique kernel function $\tilde{\Psi} : \Omega \times \Omega \to \C$
such that (1) $\tilde{\Psi}(z,z) = \Psi(z)$ for $z \in \Omega$
and (2) $\tilde{\Psi}(z,w)$ is holomorphic in $z$ and anti-holomorphic in $w$
(cf. \cite[Prop. 4.6]{I99}).
Let us observe that the metric $g$ is projectively induced
if and only if
$\tilde{\Psi}$ is a reproducing kernel of a Hilbert space
of holomorphic functions on $\Omega$.
Indeed,
if $f : \Omega \to \C P^N\,\,(N \le \infty)$ is a K\"ahler immersion
with
$f(z) = [\psi_0(z) : \psi_1(z) : \cdots ]\,\,(z \in \Omega)$
its homogeneous coordinate expression,
then we have
$\omega = \frac{i}{2} \partial\bar\partial\log \sum_{j=0}^{N}|\psi_j|^2$.
Comparing (\ref{eqn:globalpotential}) with it,
we see that
there exists a holomorphic function $\phi$ on $\Omega$ for which
$\Psi = |e^{\phi}|^2 \sum_{j=0}^N |\psi_j|^2$.
By analytic continuation,
we obtain
$\tilde{\Psi}(z,w) = e^{\phi(z)} \overline{e^{\phi(w)}} \sum_{j=0}^N \psi_j(z) \overline{\psi_j(w)}$
for $z, w \in \Omega$.
For any $z_1, \dots, z_m \in \Omega$ and $c_1, \dots, c_m \in \C$,
we have
\begin{align*}
\sum_{p,q=1}^m c_p \bar{c}_q\tilde{\Psi}(z_p,z_q)
&= \sum_{p,q=1}^m c_p \bar{c}_q e^{\phi(z_p)} \overline{e^{\phi(z_q)}} \sum_{j=0}^N \psi_j(z_p) \overline{\psi_j(z_q)} \\
&= \sum_{j=0}^N |\sum_{p=1}^m c_p e^{\phi(z_p)}\psi_j(z_p)|^2 \ge 0.
\end{align*}
Thus the matrix $(\tilde{\Psi}(z_p,z_q))_{p,q} \in \mathrm{Mat}(m,\C)$ is always
a positive Hermitian matrix.
Therefore $\tilde{\Psi}$ is a reproducing kernel of a Hilbert space
(see \cite[p. 344]{Ar50}).

On the other hand,
if $\tilde{\Psi}$ is a reproducing kernel of a Hilbert space
$\mathcal{H} \subset \mathcal{O}(\Omega)$,
then by taking an orthonormal basis $\{\psi_j\}_{j=0}^N$ of $\mathcal{H}$,
we have a K\"ahler immersion $f : M \owns z \mapsto [\psi_0(z) : \psi_1(z) : \cdots] \in \C P^N$
because we have $\Psi(z) = \tilde{\Psi}(z,z) = \sum_{j=0}^N |\psi_j(z)|^2$.
Note that there exists no point $a \in \Omega$ such that $\psi_j(a) = 0$ for all $1 \le j \le N$
since $\Psi(z) = \sum_{j=0}^N |\psi_j(z)|^2$ is always positive.

The condition for $\tilde{\Psi}$ to be a reproducing kernel is described in \cite{I99}.
In order to apply the results,
we need a fine description of the Lie algebra $\gs$ with $j$
due to Piatetskii-Shapiro \cite{PS69}.
Indeed, it is shown in \cite[Ch. 2]{PS69} that
the correspondence between the homogeneous bounded domain $\Omega$ and the structure of $(\gs, j)$
is one-to-one up to natural equivalence.
For a linear form $\alpha$ on the Cartan algebra $\mathfrak{a}$,
we denote by $\gs_{\alpha}$ the root subspace
$\set{X \in \gs}{[C,X] = \alpha(C)X \,\, (\forall C \in \mathfrak{a})}$
of $\gs$.
The number $r := \dim \mathfrak{a}$
is nothing but the rank of $\Omega$.
Thanks to \cite[Ch. 2, Sec. 3]{PS69},
there exists a basis
$\{\alpha_1, \dots, \alpha_r\}$
of $\mathfrak{a}^*$
such that
$\gs = \gs(0) \oplus \gs(1/2) \oplus \gs(1)$ with:
\begin{align*}
\gs(0) &= \mathfrak{a} \oplus
          \dirsum_{1 \le k < l \le r} \gs_{(\alpha_l - \alpha_k)/2}, \quad
\gs(1/2) = \dirsum_{1 \le k \le r} \gs_{\alpha_k /2}, \\
\gs(1) &= \dirsum_{1 \le k \le r} \gs_{\alpha_k}
        \oplus
        \dirsum_{1 \le k < l \le r} \gs_{(\alpha_l + \alpha_k)/2}.
\end{align*}
If $\{A_1, \dots, A_r\}$ is the basis of $\mathfrak{a}$ dual to $\{\alpha_1, \dots, \alpha_r\}$,
then $\gs_{\alpha_k} = \R jA_k$.
Thus $\gs_{\alpha_k}\,\,(k=1, \dots, r)$ is always one dimensional,
whereas other root spaces $\gs_{\alpha_k/2}$ and $\gs_{(\alpha_l \pm \alpha_k)/2}$ may be $\{0\}$.
Since $\{\alpha_1, \dots, \alpha_r\}$ is a basis of $\mathfrak{a}^*$,
the linear form $\gamma \in \mathfrak{a}^*$ defined above can be  written as
$\gamma = \sum_{k=1}^r \gamma_k \alpha_k$
with unique $\gamma_1, \dots, \gamma_r \in \R$.
Since $j A_k \in \gs_{\alpha_k}$,
we have:
\begin{align*}
\gamma_k = \gamma(A_k) = -4 \beta (jA_k) = -4 \beta([A_k, jA_k]) = 4 \beta([jA_k, A_k])
\end{align*}
and the last term equals $4 g (A_k \cdot z_0, A_k \cdot z_0)$.
Thus we get $\gamma_k >0$.

For $\ep = (\ep_1, \dots, \ep_r) \in \{0,1\}^r$,
put
$q_k(\ep) := \sum_{l>k} \ep_l \dim \gs_{(\alpha_l- \alpha_k)/2}
\,\,\,(k=1, \dots, r)$.
Define:
$$
\mathfrak{X}(\ep) := \set{(\sigma_1, \dots, \sigma_r) \in \C^r}
{\begin{aligned}
\sigma_k &> q_k(\ep) /2 \quad (\ep_k = 1)\\
\sigma_k &= q_k(\ep) /2 \quad (\ep_k = 0)
\end{aligned}},
$$
and $\mathfrak{X} := \bigsqcup_{\ep \in \{0,1\}^r} \mathfrak{X}(\ep)$.
By \cite[Theorem 4.8]{I99},
$\tilde{\Psi}$ is a reproducing kernel if and only if
$\underline{\gamma} := (\gamma_1, \dots, \gamma_r)$
belongs to $\mathfrak{X}$.
We denote by $W(g)$ the set of $\lambda >0$ for which $\lambda g$ is projectively induced.
Since
the metric $\lambda g$ corresponds to the parameter $\lambda \underline{\gamma}$,
we see that $\lambda g$ is projectively induced if and only if $\lambda \underline{\gamma} \in \mathfrak{X}$.
Namely we obtain:
$$
W(g) = \set{\lambda >0}{\lambda \underline{\gamma} \in \mathfrak{X}},
$$
and the right-hand side is considered in \cite{I10}.
Put $q_k = \sum_{l>k} \dim \gs_{(\alpha_l - \alpha_k)/2}$ for $k=1, \dots, r$.
Then \cite[Theorem 15]{I10} tells us that:
$$
W(g) \cup \{0\} \subset \set{\frac{q_k}{2\gamma_k}}{k=1, \dots, r} \cup (c_0, +\infty),
$$
where $c_0 := \max \set{\frac{q_k}{2\gamma_k}}{k=1, \dots, r}$.

Now assume that $\lambda g$ is projectively induced for all $\lambda >0$.
Then we have $c_0 = 0$, so that $\dim \gs_{(\alpha_l-\alpha_k)/2} = 0$ for all $1 \le k < l \le r$.
In this case, we see that $\gs$ is a direct sum of ideals
$\gs_k := j \gs_{\alpha_k} \oplus \gs_{\alpha_k/2} \oplus \gs_{\alpha_k}
\,\,\,(k=1, \dots, r),$
which correspond to the hyperbolic spaces $\CH^{n_k}$ with $n_k = 1 + (\dim_{\alpha_k/2})/2$
(\cite[pages 52--53]{PS69}).
Therefore the Lie algebra $\gs$ corresponds to the direct product
$\CH^{n_1} \times \cdots \times \CH^{n_\rk}$,
which is biholomorphic to $\Omega$ because the homogeneous domain
$\Omega$ also corresponds to $\gs$.
Hence (\ref{eqn:direct_prod}) holds and this concludes the proof of the theorem.
\end{proof}

\section[\K\ immersions of h.K.m. into $\C^{N\leq\infty}$ and $\CH^{N\leq\infty}$]{K\"ahler immersions of homogeneous K\"ahler manifolds into  $\C^{N\leq\infty}$ and $\CH^{N\leq\infty}$}\label{hkmcch}

In this section we classify homogeneous K\"ahler manifolds  
which admit a \K\ immersion into  $\C^N$ or $\CH^N$, $N\leq\infty$ (Theorems \ref{Flat-Case} and  \ref{Negative-case} respectively).
For this purpose, we need the following two lemmata (Lemma \ref{parallel} and Lemma \ref{splitting}) and the classification 
of all  the homogeneous K\"ahler manifolds (Theorem \ref{FC}) due to 
J. Dorfmeister and  K. Nakajima \cite{DN88}.

Recall that complete connected totally geodesic submanifolds of $\mathbb{R}^n$ are affine subspaces $p + \mathbb{W}$,
where $p \in \mathbb{R}^n$ and $\mathbb{W} \subset \mathbb{R}^n$ is a vector subspace.
The reader is referred to  \cite{AD03} for the proof of the following result.  

\begin{lem} \label{parallel} Let $G$ be a connected  Lie subgroup of
isometries of the Euclidean space $\mathbb{R}^n$. Let $G.p = p + \mathbb{V}$ and $G.q = q + \mathbb{W}$ be two totally geodesic $G$-orbits.
Then $\mathbb{V} = \mathbb{W}$, i.e. $G.p$ and $G.q$ are parallel affine subspaces of $\mathbb{R}^n$.
\end{lem}

Notice that if two K\"ahler manifolds $(M_1,g_1)$ and $(M_2,g_2)$ admit K\"ahler immersions, say $f_1$ and $f_2$, into $\C^{N_1}$ and $\C^{N_2}$ respectivley, $N_j\leq\infty, j=1, 2$, then the K\"ahler manifold $(M_1\times M_2,g_1\oplus g_2)$ admits a K\"ahler immersion into $\C^N$, $N=N_1+N_2$ obtained by mapping $(z_1,z_2)\in M_1\times M_2$ to $(c_1f_1(z_1),c_2f_2(z_2))\in \C^N$, for suitable constants $c_1$ and $c_2$. The converse is also true:

\begin{lem}\label{splitting}
A K\"ahler immersion $f\!: M_1\times M_2\f \C^N$, $N\leq\infty$, from the product $M_1\times M_2$ of two K\"ahler manifolds is a product, i.e. up to unitary transformation of $\C^N$ $f(p,q)=(f_1(p),f_2(q))$, where $f_1\!: M_1\f\C^{N_1}$ and $f_2\!: M_2\f\C^{N_2}$, $N=N_1+N_2$, are K\"ahler immersions.
\end{lem}
\begin{proof}
Let $\alpha(X,Y)$ be the second fundamental form of the K\"ahler map $f$. In order to show that $f$ is a product it is enough to prove that $\alpha(TM_1,TM_2)\equiv0$, see \cite{moore} and \cite[Lemma 2.5]{discala}. The Gauss equation implies the following equation for the holomorphic bisectional curvature of $M_1\times M_2$, see \cite[Prop. 9.2, pp. 176]{kono}:
\begin{equation}
-<R_{X,JX}JY,Y>=2||\alpha(X,Y)||^2,\nonumber
\end{equation}
where $R$ is the curvature tensor of $M_1\times M_2$. Thus, if $X\in TM_1$ and $Y\in TM_2$, we get that $\alpha(X,Y)=0$.
\end{proof}

\begin{theor}[J. Dorfmeister, K. Nakajima, \cite{DN88}]\label{FC}
A homogeneous K\"ahler manifold $(M, g)$  is the total space of a holomorphic
fiber bundle over a homogeneous bounded domain $(\Omega, g)$ in which the fiber ${\mathcal F} ={\mathcal E} \times {\mathcal C}$
is (with the
induced K\"ahler metric) the  \K\ product of a flat homogeneous K\"ahler manifold ${\mathcal E}$ and a
compact simply-connected homogeneous K\"ahler manifold ${\mathcal C}$.
\end{theor}
A flat homogeneous K\"ahler manifold is the \K\ product of the quotients of the complex Euclidean spaces  with the flat metric. Examples of such manifold are in the compact case the flat  complex tori (see  Example \ref{cstar} for a noncompact and non simply-connected example).

\begin{theor}[A. J. Di Scala, H. Ishi, A. Loi, \cite{ishi}]\label{Flat-Case}
Let  $(M, g)$ be an $n$-dimensional homogeneous K\"ahler manifold.
\begin{itemize}
\item [(a)]
If $(M, g)$  can be \K\ immersed into $\C ^N$, $N<\infty$,  then
$(M, g)=\C^n$;
\item [(b)]
if $(M, g)$  can be \K\ immersed into $l^2(\C)$,  then
$(M,  g)$  equals:
$$\C ^k\times (\CH^{n_1},\lambda_1g_{hyp}) \times\cdots \times (\CH^{n_\rk},\lambda_\rk g_{hyp}),$$
where $k+n_1+\cdots +n_\rk=n$, $\lambda_j>0$, $j=1,\dots , \rk$.
\end{itemize}
Moreover, in case (a) (resp. case (b)) the immersion is given, up to a unitary transformation of $\C ^N$ (resp. $l^2 (\C))$,  by the linear inclusion $\C^n\hookrightarrow \C^N$
(resp. by
$(f_0, f_1,\dots ,f_\rk)$,
where $f_0$ is the  linear inclusion
$\C ^k\hookrightarrow l^2(\C)$
and   each $f_j:\CH^{n_j}\rightarrow l^2({\C})$
is
$\sqrt{\lambda_j}$ times
the map \eqref{chinell}).
\end{theor}
\begin{proof}
Assume that there exists a K\"ahler immersion $f: M \rightarrow \C^N$. By Theorem \ref{FC}
and by the fact that a homogeneous bounded domain is contractible
we get that  $M = \mathbb{C}^k\times \Omega$ as a complex manifold since,
by the maximum principle, the fiber ${\mathcal F}$ cannot contain a compact manifold. Let $M = G/K$ be the homogeneous realization  of $M$ (so the metric
$g$ is $G$-invariant).
It follows  again by Theorem \ref{FC} that there exists $L \subset G$ such that the $L$-orbits are the fibers of the fibration $\pi: M = G/K \rightarrow \Omega = G/L $. Let $F_p, F_q$ be the fibers over $p,q \in \Omega$. We claim that $f(F_p)$ and $f(F_q)$ are parallel affine subspaces of $\C^N$.
Indeed, by Calabi's rigidity  $f(F_p)$ and $f(F_q)$ are affine subspaces of $\C^N$ since both $F_p$ and $F_q$ are flat \K\ manifolds  of $\mathbb{C}^n$.
Moreover, Calabi rigidity theorem implies the existence of a  morphism of groups $\rho: G \rightarrow Iso_{\mathbb{C}}(\C^N) = \mathrm{U}(\C^N) \ltimes \C^N$ such that
$f(g\cdot x) = \rho(g)f(x)$
for all $g \in G, x \in M$.
Let $W_{p,q}$ be the affine subspace generated by $f(F_p)$ and $f(F_q)$. Since both $f(F_p)$ and $f(F_q)$ are $\rho(L)$-invariant it follows that $W_{p,q}$ is also $\rho(L)$-invariant. Indeed, for any $g \in L$ the isometry $\rho(g)$ is an affine map and so must preserve the affine space generated by $f(F_p)$ and $f(F_q)$. Observe that $W_{p,q}$ is a finite dimensional complex Euclidean space,  $\rho(L)$ acts on $W_{p,q}$
and $f(F_p)$ and $f(F_q)$ are  two complex totally geodesic orbits in $W_{p,q}$. Then,  by  Lemma \ref{parallel}, we get that $f(F_p)$ and $f(F_q)$ are parallel affine subspaces of $W_{p,q}$ and hence  of  $\C^N$. Since $p,q \in \Omega$ are two arbitrary points it follows that $f(M)$ is a K\"ahler product. Thus $M = \mathbb{C}^{k}\times \Omega$ is a K\"ahler product of  homogeneous K\"ahler manifolds. Using again the fact that $M$ can be \K\ immersed into $\C^N$ it follows that the homogeneous bounded domain $\Omega$ can be \K\ immersed into $\C^N$.
If one denotes by $\varphi$ this immersion  and by
$g_{\Omega}$ the homogeneous
\K\ metric of $\Omega$,
it follows that  the map $\sqrt{\lambda}\varphi$ is a \K\ immersion of $(\Omega, \lambda g_{\Omega})$ into $\C^N$ (cfr. Remark \ref{constantmulti}). Therefore, by Lemma 
\ref{bochnerth}, $\lambda g_{\Omega}$ is projectively  induced for all $\lambda >0$ and
Theorem \ref{thmsmall} yields:
$$(M, g) = \C ^k\times (\CH^{n_1},\lambda_1g_{hyp}) \times\cdots \times (\CH^{n_\rk},\lambda_\rk g_{hyp}),$$
where $k+n_1+\cdots +n_\rk=n$ and $\lambda_j$, $j=1,\dots , \rk$
are positive real numbers.
If the dimension $N$ of the ambient space $\C^N$ is finite then
$M =\C^n$ since there cannot exist a \K\ immersion of
$(\C H^{n_j}, \lambda_j g_{hyp})$
into $\C^N$, $N<\infty$ (see Theorem \ref{calabic}) and this proves (a).
The last part of Theorem \ref{Flat-Case}
is a consequence of Calabi's Rigidity Theorem \ref{local rigidity} together with Lemma \ref{splitting}.
\end{proof}

\begin{theor}[A. J. Di Scala, H. Ishi, A. Loi, \cite{ishi}]\label{Negative-case}
Let  $(M, g)$  be an $n$-dimensional  homogeneous K\"ahler manifold.
Then if $(M, g)$  can be \K\ immersed into $\CH^N$, $N\leq\infty$,  then
$(M, g)=\CH^n$
and the immersion is given,
up to  a unitary transformation of $\CH^N$,
by the linear inclusion $\CH^n\hookrightarrow \CH^N$.
\end{theor}

%

\begin{proof}
If  a homogeneous K\"ahler manifold $(M, g)$  can be
\K\ immersed into $\C H^{N}$,
$N\leq\infty$,
then, by Lemma \ref{bochnerthalt}  it can also be
K\"ahler immersed into $l^2(\mathbb{C})$.
By Theorem \ref{Flat-Case},
$(M, g)$ is then   a \K\ product of complex space forms, namely
$$(M, g) = \C ^k\times (\CH^{n_1},\lambda_1g_{hyp}) \times\cdots \times (\CH^{n_\rk},\lambda_\rk g_{hyp}),$$
Then  the conclusion follows from  the fact that $\C^k$ cannot be \K\ immersed into
$\CH^N$ for all $N\leq\infty$ (see Theorem \ref{calabic}) and from \cite[Theorem 2.11]{alek}
which  shows that there are not \K\ immersions
from a product $M_1 \times M_2$ of \K\ manifolds into $\CH^{N}$, $N\leq\infty$.
\end{proof}

\section[\K\ immersions of h.K.m. into $\CP^{N\leq\infty}$]{K\"ahler immersions of homogeneous K\"ahler manifolds into $\CP^{N\leq\infty}$}\label{hkmcp}
As we have already pointed out in Remark \ref{compactfull} a necessary condition for a 
K\"ahler metric $g$ on a complex  manifold $M$  to be projectively induced 
is that its associated \K\ form $\omega$ is integral
i.e.  it   represents the first Chern class $c_{1}(L)$ in $H^2(M, \z)$
of a holomorphic line bundle $L\rightarrow M$. Indeed $L$ can be taken as the pull-back of the hyperplane line  bundle on $\C P^N$
whose first Chern class can represented  by  $\omega_{FS}$.
Observe also  that if $\omega$ is an exact form (e.g. when  $M$ is contractible) then $\omega$ is obviously
integral  since its  second cohomology class vanishes.
Other (less obvious) conditions are expressed by the following theorem and its corollary.
\vskip 0.3cm

\begin{theor}[A. J. Di Scala, H. Ishi, A. Loi \cite{ishi}]\label{necessary}
Assume that a homogeneous K\"ahler manifold $(M, g)$ admits a K\"ahler immersion $f : M \rightarrow \CP^{N}$, $N\leq\infty$. Then
$M$ is simply-connected and $f$ is injective.
\end{theor}
\begin{proof}
Theorem \ref{FC} and the fact that a homogeneous bounded domain is contractible imply that $M$ is a {\em complex}
product $\Omega \times {\mathcal F}$, where  ${\mathcal F}={\mathcal E}\times {\mathcal C}$ is a \K\ product  of a flat  \K\ manifold ${\mathcal E}$ \K\ embedded into $(M, g)$ and a simply-connected homogeneous K\"ahler manifold ${\mathcal C}$.
We claim that ${\mathcal E}$ is simply-connected and hence   $M=\Omega\times {\mathcal E}\times {\mathcal C}$ is simply-connected.
In order to prove our claim notice that
${\mathcal E}$ is the \K\ product  $\C^k\times T_1\times\cdots\times T_s$, where $T_j$ are non simply-connected flat \K\ manifolds.
So one needs to show that each $T_j$  reduces to a point.
If, by a contradiction,  the dimension of one of this space, say $T_{j_0}$  is not zero, then by composing
the \K\  immersion of  $T_{j_0}$ in $(M, g)$
with the immersion  $f:M\rightarrow \CP^N$
we would get  a \K\ immersion of
$T_{j_0}$ into $\CP^N$ in contrast with Exercise \ref{simplyflat}.
In order to prove that $f$ is injective we first observe that, by Calabi's Rigidity Theorem \ref{local rigidityb},  $f(M)$ is still a homogeneous K\"ahler manifold.
Then, by  the first part of the theorem,  $f(M)\subset\CP^{N}$ is simply-connected. Moreover,  since $M$ is complete and  $f: M \rightarrow f(M)$ is a local isometry,  it is a covering map (see,  e.g.,  \cite[Lemma 3.3, p. 150]{DC92}) and  hence injective.
\end{proof}
\begin{remark}\rm
When the dimension of the  ambient space is finite, i.e. $\CP^N$, $N<\infty$,
$M$ is forced to be  compact and a proof of Theorem \ref{necessary} is
well-known  by the work of M. Takeuchi \cite{tak}.
In this case he also provides a complete classification of all compact
homogeneous K\"ahler manifolds which can be \K\ immersed into $\CP^N$
by making use of the representation theory of semisimple Lie groups and Dynkin diagrams.
\end{remark}

\begin{cor}\label{corolnecessary} Let $(M, g)$ be a complete and  locally homogeneous K\"ahler manifold. Assume that  $f: (M, g)\rightarrow \CP^N$, $N\leq\infty$,  is a \K\ immersion.
Then $(M, g)$ is a homogeneous K\"ahler manifold.
\end{cor}
\begin{proof}
Let $\pi:\tilde M\rightarrow M$ be the universal  covering map.
Then $(\tilde M, \tilde g)$ is a homogeneous K\"ahler manifold and, by Theorem \ref{necessary},
$f\circ\pi :\tilde M\rightarrow \CP^n$ is injective.
Therefore $\pi$ is injective,
and since it is a covering map,
it defines a holomorphic isometry between $(\tilde M, \tilde g)$
and $(M, g)$.
\end{proof}

The following somehow surprising theorem shows that, 
once that the necessary conditions expressed above are satisified then,  
up to homotheties, any homogeneous K\"ahler manifold  is projectively induced.

\begin{theor}[A. Loi, R. Mossa \cite{loimossahom}]\label{loimossaimm}
Let $(M, g)$ be a simply-connected homogeneous K\"ahler manifold with associated \K\ form $\omega$ integral. Then there exists a positive real number $\lambda$
such that $(M, \lambda g)$ is projectively induced.
\end{theor}
\begin{proof}
Since  $\omega$ is integral there exists a holomorphic line bundle $L$ over $M$ such that $c_1(L)=[\omega]$. 
Let $h$ be an Hermitian metric on $L$ such that $\ric(h)=\omega$, where
$\ric(h)$
is the $2$-form on $M$
defined  by the equation: 
\begin{equation}\label{ricci}
\ric (h)=-\frac{i}{2}
\partial\bar\partial\log h(\sigma(x), \sigma (x)),
\end{equation} 
for a trivializing holomorphic section
$\sigma :U\subset M\rightarrow L\setminus \{0\}$
of $L$.

Choose $\lambda >0$ sufficiently large in such a way that 
$\lmb \omega$ is integral and the Hilbert space of global holomorphic sections of $L^{\lmb}=\otimes^\lmb L$ given by:
\begin{equation}\label{hilbertspaceh}
\hilb_{\lmb,h}=\left\{ s\in\ol(L^{\lmb}) \ |\   \int_M  h^{\lmb}(s,s)\, \frac{\omega^n}{n!}<\infty\right\},
\end{equation}
is non-empty. The existence of such a $\lambda$ is due to Rosenberg--Vergne \cite{rv}.
Let $\{s_j\}_{j=0,\dots,N}$, $N \leq \infty$, be an orthonormal basis for $\hilb_{\lmb, h}$. Consider the function
\begin{equation}\label{epl}
\varepsilon_{\lmb}(x)=\sum_{j=0}^Nh^\lmb(s_j(x),s_j(x)).
\end{equation}
This definition depends only on the \K\ form $\omega$. Indeed since $M$ is simply-connected, there exists (up to isomorphism) a unique $L \f M$ such that $c_1(L)=[\omega]$, and 
 it is easy to see that the definition  does not depend on the orthonormal basis chosen or on the Hermitian metric $h$. 
 Assume now that the function $\varepsilon_{\lmb}$ is  strictly positive.  One can then  consider the map $f: M \f \CP^N$ defined by:
\begin{equation}\label{eqcoherent}
f(x)=\left[s_0(x), \dots, s_j(x), \dots\right].
\end{equation}
It is not hard to see (cfr. Exercise \ref{coherentstates}) that:
\begin{equation}\label{pullbackdie}
f^*\omega_{FS}=\lmb \omega + \frac{i}{2} \partial \bar\partial\log\varepsilon_\lmb,
\end{equation}
where $\omega_{FS}$ is the Fubini--Study form on $\CP^N$.

Let now $F$  be a holomorphic isometry and let $\w F$ be its lift to $L$ (which exists since $M$ is simply-connected). 
Notice now 
 that, if $\lbrace s_0, \dots,s_N\rbrace$, $N\leq\infty$, is an orthonormal basis for $\hilb_{\lmb, h}$, then 
 $\lbrace\w{ F}^{-1} \left( s_0\left(F\left(x\right) \right)\right),\dots , \w{ F}^{-1} \left( s_N\left(F\left(x\right) \right)\right)\rbrace$ is an orthonormal basis for $\hilb_{\lmb,\w{ F} ^*h}$. 
Therefore
\begin{equation*}
\begin{split}
\epsilon_\lmb(x)
&=\sum_{j=0}^N \w {F}^* h^\lmb \left( \w {F}^{-1}(s_j(F(x))),\w {F}^{-1}(s_j(F(x)))\right) \\
&=\sum_{j=0}^N h^\lmb\left( s_j(F(x)), s_j(F(x)) \right)=\varepsilon_{\lmb}\left(F(x)\right).
\end{split}
\end{equation*}
Since the group of holomorphic isometries acts transitively on $M$ it follows that $\varepsilon_{\lmb}$ is forced to be a  positive  constant.
Hence the map $f$ can be defined and it is a \K\ immersion. 
\end{proof}

\begin{remark}\rm
The integrality of  $\omega$  in this theorem cannot be dropped since there exists a simply-connected homogeneous K\"ahler manifold $(M, \omega)$ such that $\lambda \omega$ is not integral for any $\lambda \in \mathbb{R}^+$ (take,  for example, $(M, g) = (\mathbb{C}P^1, g_{FS}) \times (\mathbb{C}P^1, \sqrt{2}g_{FS})$).
Observe also that  there exist simply-connected (even contractible) homogeneous K\"ahler manifolds $(M, g)$ such that $\omega$ is an integral form but
$g$ is not projectively induced (see e.g. Theorem \ref{wallach}).
\end{remark}

\begin{remark}\label{balanced}\rm
The \K\ metric $g$  as in the previous theorem such that the function $\epsilon_{\lambda}$ is a positive constant for all $\lambda >0$ plays a prominent role in the theory of quantization of \K\ manifolds and also in algebraic geometry when $M$ is compact. A \K\  metric $\lambda g$
satisfying this property is  called a {\em balanced metric} and the pair $(L, h)$ is called a {\em regular quantization}  of the  the \K\ manifold
$(M ,\omega)$. The interested reader is referred to \cite{arezzoloi, cucculoi, donaldson,donaldson2,englisb, grecoloi, hartogs, balancedch, taubnut}
 for more details on these metrics.
\end{remark}

\section{Bergman metric and bounded symmetric domains}\label{bergmansection}
Let $D$ be a (non necessarily homogeneous) bounded domain of $\mathds{C}^n$ with coordinate system $z_1,\dots, z_n$ and consider the separable complex Hilbert space $L_{hol}^2(D)$ of square integrable holomorphic functions on $D$, i.e.:
$$L_{hol}^2(D)=\left\{f\in \ol(D), \int_D |f|^2d\mu<\infty\right\},$$
where $d\mu$ denotes the Lebesgue measure on $\R^{2n}=\C^n$. Pick an orthonormal basis $\{\varphi_j\}$ of $L_{hol}^2(D)$ with respect to the inner product given by:
\begin{displaymath}
(f,h)=\int_D f(\zeta)\overline{h(\zeta)}d\mu(\zeta),\quad f,\,h\,\in L_{hol}^2(D).
\end{displaymath}
The {\em Bergman kernel} of $L_{hol}^2(D)$ is the function:
\begin{displaymath}
\Kj(z,\zeta)=\sum_{j=0}^\infty\varphi_j(z)\overline{\varphi_j(\zeta)},
\end{displaymath}
which is holomorphic in $D\times \bar D$, or equivalently, holomorphic in $z$ and antiholomorphic in $\zeta$, also called {\em reproducing kernel} for its reproducing property:
\begin{equation}\label{reproducing}
f(z)=\int_D  \Kj(z,\zeta)f(\zeta)d\mu(\zeta), \quad f\in L_{hol}^2(D).
\end{equation}
The Bergman metric on $D$ is the K\"ahler metric associated to the K\"ahler form:
$$
\omega_B=\frac{i}{2}\de\bar\de\log\Kj(z,z).
$$
Since $\log\Kj(z,z)$ is a K\"ahler potential for $\omega_B$, from \eqref{diastdefinition} we get:
 $$
\dd(z,w)=\log\frac{\Kj(z,z)\Kj(w,w)}{|\Kj(z,w)|^2},
$$
further, since by the reproducing property \eqref{reproducing} we get:
$$
\frac{1}{\Kj(0,0)}=\int_\Omega\frac{1}{\Kj(\zeta,0)}\Kj(\zeta,0)d\mu(\zeta),
$$
from which follows $\Kj(0,0)=1/V(\Omega)$, the diastasis centered at the origin reads:
$$
\dd_0(z)=\log\frac{\Kj(z,z)}{V(\Omega)|\Kj(z,0)|^2}.
$$

The Bergman metric is projectively induced in a natural way. In fact, the full holomorphic map:
\begin{equation}\label{philoc}
\varphi :M\rightarrow {\mathds{C}{\rm P}}^{\infty},\quad 
x\mapsto[\varphi_{0}(x),\dots ,\varphi_{j}(x),\dots],
\end{equation}
satisfies $g_B= \varphi^*(g_{FS})$, for by \eqref{diastcphom}:
$$
\dd^{FS}(\varphi(z),\varphi(w))=\log\frac{\sum_{j,k=0}^\infty|\varphi_j(z)|^2| \varphi_k(w)|^2}{\left|\sum_{j=0}^\infty\varphi_j(z)\bar \varphi_k(w)\right|^2}=\log\frac{\Kj(z,z)\Kj(w,w)}{|\Kj(z,w)|^2}=\dd(z,w).
$$

Notice that   the group of automorphisms $\aut(D)$ of $D$, i.e. biholomorphisms $f\!:D\f D$, is contained in the group of isometries ${\rm Isom}(D,g_B)$, that is if $F\in \aut(D)$ then  $F^* g_B=g_B$. If $\aut(D)$ also acts transitively, i.e. $D$ is homogeneous, then $g_B$ is Einstein and $\ric_{g_B}=-2g_B$, i.e. the Einstein constant is $-2$ (cfr. \cite[p. 163]{kono}).
   Observe that the Bergman metric and the hyperbolic metric on $\CH^n$ (see 3. of Section \ref{csf}) are homothetic, more precisely one has $(n+1)g_{hyp}=g_B$. 
   
 On a homogeneous bounded domain there could be many  non-homothetic homogeneous metrics, the Bergman metric is one of them.
 It could happen that the only homogeneous metric on homogeneous bounded domain is a multiple of the Bergman metric.
 This happens for example for the {\em bounded symmetric domains} $(\Omega, cg_B)$ that are convex domains $\Omega\subset\C^n$ which are circular, i.e. $z\in \Omega,\ \theta\in\R\Rightarrow\ e^{i\theta}z\in\Omega$ (see \cite{kodomain} for details).
Every bounded symmetric domain is the product of irreducible factors, called Cartan domains.
From E. Cartan classification, Cartan domains can be divided into two categories, classical and exceptional ones (see \cite{kohyper} for details). Classical domains 
can be described in terms of complex matrices as follows ($m$ and $n$ are nonnegative integers, $n\geq m$):
\begin{equation}
\begin{split}
&\Omega_1[m, n]=\{Z\in M_{m,n}(\C),\  I_m-ZZ^*>0\} \qquad \qquad \qquad(\dim(\Omega_1)=nm),\\
&\Omega_2[n]=\{Z\in M_{n}(\C),\  Z=Z^T,\  I_n-ZZ^*>0\} \qquad \quad(\dim(\Omega_2)=\tfrac{n(n+1)}{2}),\\
&\Omega_3[n]=\{Z\in M_{n}(\C),\  Z=-Z^T,\  I_n-ZZ^*>0\} \qquad\; (\dim(\Omega_3)=\tfrac{n(n-1)}{2}),\\
&\Omega_4[n]=\{Z=(z_1,\dots,z_n)\in\C^n,\ \sum_{j=1}^n|z_j|^2\!<\!1,1+|\sum_{j=1}^nz_j^2|^2\! -\!2\sum_{j=1}^n|z_j|^2>0\}\\
& \qquad \qquad \qquad \qquad \quad \quad\ \  \qquad \qquad \qquad \qquad \qquad\qquad(\dim(\Omega_4)=n),\ n\neq 2,\nonumber
\end{split}
\end{equation}
where $I_m$ (resp. $I_n$) denotes the $m\times m$ (resp $n\times n$) identity matrix, and $A>0$ means that $A$ is positive definite.
In the last domain we are assuming $n\neq 2$ since $\Omega_4[2]$ is not irreducible (and hence it is not a Cartan domain). In fact, the biholomorphism:
$$f\!:\Omega_4[2]\f\CH^1\times\CH^1,\ (z_1,z_2)\mapsto(z_1+iz_2,z_1-iz_2),$$
satisfies:
$$f^*(2(g_{hyp}\oplus g_{hyp}))=g_B.$$
The reproducing kernels of classical Cartan domains are given by:
\begin{equation}
\Kj_{\Omega_1}(z,z)=\frac{1}{V(\Omega_1)}[\det(I_m-ZZ^*)]^{-(n+m)},\nonumber
\end{equation}
\begin{equation}
\Kj_{\Omega_2}(z,z)=\frac{1}{V(\Omega_2)}[\det(I_n-ZZ^*)]^{-(n+1)},\nonumber
\end{equation}
\begin{equation}
\Kj_{\Omega_3}(z,z)=\frac{1}{V(\Omega_3)}[\det(I_n-ZZ^*)]^{-(n-1)},\nonumber
\end{equation}
\begin{equation}\label{kernel4}
\Kj_{\Omega_4}(z,z)=\frac{1}{V(\Omega_4)}\left(1+|\sum_{j=1}^nz_j^2|^2 -2\sum_{j=1}^n|z_j|^2\right)^{-n},\\
\end{equation}
where $V(\Omega_j)$, $j=1,\dots, 4$,  is the total volume of $\Omega_j$ with respect to the Euclidean measure of the ambient complex Euclidean space (see \cite{symm} for details).\\
Notice that for some values of $m$ and $n$, up to multiply the metric by a positive constant, the domains coincide with the hyperbolic space $\CH^n$, more precisely we have:
\begin{equation}
(\Omega_1[1,n],g_B)=(\CH^n,(n+1)g_{hyp}),\nonumber
\end{equation}
\begin{equation}
(\Omega_2[1],g_B)=(\Omega_3[2],g_B)=(\Omega_4[1],g_B)=(\CH^1,2g_{hyp}), \nonumber
\end{equation}
\begin{equation}
(\Omega_3[3],g_B)=(\CH^3,4g_{hyp}). \nonumber
\end{equation}
In general, $(\Omega,g_B)=(\CH^n,cg_{hyp})$, for some $c>0$, if and only if the rank of $\Omega$ is equal to $1$.
There are two kinds of exceptional domains $\Omega_5[16]$ of dimension $16$ and $\Omega_6[27]$ of dimension $27$, corresponding to the dual of $E\ III$ and $E\ VII$, that can be described in terms of $3\times 3$ matrices with entries in the $8$-dimensional algebra of complex octonions $\oo_\C$. We refer the reader to \cite{roos} for a more complete description of these domains.

\begin{remark}\label{immersbsd}\rm
It is interesting to observe that any irreducible bounded symmetric domain of rank greater or equal than $2$, can be exhausted by totally geodesic submanifolds isomorphic to $\Omega_4[3]$ and that 
every bounded symmetric domain different  from
$$
\left(\CH^{n_1}\times\cdots\times\CH^{n_s}, c_1\,g_{hyp}\oplus\cdots\oplus c_s\, g_{hyp}\right),
$$
for $c_1,\dots, c_s$ positive constants, admits $\Omega_4[3]$ as a K\"ahler submanifold
(cfr. \cite{tsai} for the proofs of these assertions).
\end{remark}

\section{\K\ immersions of bounded symmetric domains into $\CP^{\infty}$}\label{symmcp}
Being  a bounded symmetric domain a particular case of  homogeneous bounded domain and so of homogeneous \K\ manifolds, we already know about the existence of \K\ immersions
into finite or infinite dimensional complex space form. 
In Theorem \ref{wallach} 
we describe for what values of $c>0$ a bounded symmetric domain
can be \K\ immersed into $\CP^{\infty}$.
 We start with the definition of the  {\em Wallach set}  of an irreducible bounded symmetric domain  $(\Omega,cg_B)$ of genus $\gamma$ and Bergman kernel $\Kj$, referring  the reader to  \cite{arazy}, \cite{faraut} and \cite{upmeier} for more details and results. This set, denoted by $W(\Omega)$, consists of all $\eta\in\C$ such that there exists a Hilbert space ${\mathcal H}_\eta$ whose reproducing kernel is   $\Kj ^{\frac{\eta}{\gamma}}$. This is equivalent to the requirement that $\Kj ^{\frac{\eta}{\gamma}}$ is positive definite, i.e.  for all $n$-tuples of  points $x_1,\dots,x_n$ belonging to $\Omega$ the $n\times n$ matrix $(\Kj(x_{\alpha},x_{\beta})^{\frac{\eta}{\gamma}})$, is positive  {\em semidefinite}. It turns out (cfr. \cite[Cor. 4.4, p. 27]{arazy} and references therein)
that $W(\Omega)$ consists only of real numbers and depends on two of the domain's invariants, $a$ and  $r$. More precisely we have:
\begin{equation}\label{wallachset}
W(\Omega)=\left\{0,\,\frac{a}{2},\,2\frac{a}{2},\,\dots,\,(r-1)\frac{a}{2}\right\}\cup \left((r-1)\frac{a}{2},\,\infty\right).
\end{equation}
The set $W_d=\left\{0,\,\frac{a}{2},\,2\frac{a}{2},\,\dots,\,(r-1)\frac{a}{2}\right\}$ and the interval $W_c= \left((r-1)\frac{a}{2},\,\infty\right)$
are called respectively  the {\em discrete} and {\em continuous} part   of the Wallach set of the domain 
$\Omega$. The reader is referred to \cite[Prop. 3]{loimossaber} for an analogous description of the Wallach set  of bounded homogeneous domains.
\begin{remark}\label{rchimm}\rm
If $\Omega$ has rank $r=1$, namely $\Omega$ is the complex hyperbolic space $\CH^d$,
then $g_B=(d+1)g_{hyp}$.
In this case (and only in this case) $W_d=\{0\}$ and $W_c=(0, \infty)$.
If $d=1$, the Hilbert space ${\mathcal H}$ associated to the kernel:
$$\Kj=\frac{1}{(1-|z|^2)^\alpha}, \qquad \alpha>0,$$
is the space:
\begin{equation}
{\mathcal H}=\left\{f\in\ol(\CH^1), f(z)=\sum_{j=0}^\infty a_jz^j\,|\ \sum_{j=0}^\infty\frac{\Gamma(\alpha)\Gamma(j+1)}{\Gamma(j+\alpha)}|a_j|^2<\infty\right\},\nonumber
\end{equation}
endowed with the scalar product:
\begin{equation}
<g,h>=\sum_{j=0}^\infty\frac{\Gamma(\alpha)\Gamma(j+1)}{\Gamma(j+\alpha)}b_j\bar c_j,\nonumber
\end{equation}
where $g(z)=\sum_{j=0}^\infty b_j z^j$, $h(z)=\sum_{j=0}^\infty c_j z^j$ and $\Gamma$ is the Gamma function.

If $\alpha>1$, ${\mathcal H}$ is the weighted Bergman space of $\Omega$, namely the Hilbert space of analytic functions $f\in\ol(\CH^1)$ such that:
\begin{equation}
\int_{\CH^1}|f(z)|^2d\mu_\alpha(z)<\infty,\nonumber
\end{equation}
where $\mu_\alpha(z)$ is the Lebesgue measure of $\C$.
\end{remark}

The following proposition provides the expression of the diastasis function for $(\Omega,g_B)$ (see also \cite{loidiast}) and proves a very useful property of the matrix of coefficients $(b_{jk})$ given by  (\ref{powexdiastcp}).
\begin{prop}\label{diastdom}
Let $\Omega$ be a bounded symmetric domain.
Then the diastasis for its Bergman metric $g_B$ around the origin  is:
\begin{equation}\label{diastberg}
\dd^\Omega_0(z)=\log(V(\Omega)\Kj(z,z)),
\end{equation}
where $V(\Omega)$ denotes 
the total volume of $\Omega$ with respect to the Euclidean measure of the ambient complex Euclidean space.
Moreover the matrix $(b_{jk})$ given by  (\ref{powexdiastcp}) for $\dd_0^\Omega$ satisfies $b_{jk}=0$ whenever $|m_j|\neq|m_k|$.
\end{prop}
\begin{proof}
The K\"ahler potential $\dd^\Omega_0(z)$ is centered at the origin, in fact by the reproducing property of the  kernel we have:
\begin{equation}
\frac{1}{ \Kj(0,0)}=\int_{\Omega} \frac{1}{\Kj(\zeta,0)}\Kj(\zeta,0)d\mu,\nonumber
\end{equation}
hence $\Kj(0,0)=1/V(\Omega)$, and substituting in (\ref{diastberg}) we obtain $\dd^\Omega_0(0)=0$. By the circularity of $\Omega$ (i.e. $z\in \Omega,\ \theta\in\R$
 imply $e^{i\theta}z\in\Omega$),
rotations around the origin are automorphisms and hence isometries, that leave $\dd^\Omega_0$ invariant.  Thus we have $\dd^\Omega_0(z)=\dd^\Omega_0(e^{i\theta}z)$ for any $0\leq\theta\leq 2\pi$, that is, each time we have a monomial $z^{m_j}\bar z^{m_k}$ in $\dd^\Omega_0(z)$, we must have $$z^{m_j}\bar z^{m_k}=e^{i|m_j|\theta}z^{m_j} e^{-i|m_k|\theta}\bar z^{m_k}=z^{m_j}\bar z^{m_k} e^{(|m_j|-|m_k|)i\theta},$$ implying $|m_j|=|m_k|$.
This means that  every  monomial in the expansion of  $\dd^\Omega_0(z)$ has holomorphic and antiholomorphic part with the same degree. 
Hence, by Theorem \ref{chardiast}, $\dd^\Omega_0(z)$ is the diastasis
for $g_B$. By the chain rule the same property holds true for   $e^{\dd^\Omega_0(z)}-1$
and the second part of the proposition follows immediately. 
\end{proof}

The following theorem interesting on its own sake will be an important tool in the next chapter.

\begin{theor}[A. Loi, M. Zedda, \cite{articwall}]\label{wallach}
Let $\Omega$ be an irreducible bounded symmetric domain endowed with its Bergman metric $g_B$. Then $(\Omega,cg_B)$ admits a equivariant K\"ahler immersion into $\mathds{C}P^\infty$ if and only if $c\gamma$ belongs to $W(\Omega)\setminus\{0\}$, where $\gamma$ denotes the genus of $\Omega$.
\end{theor}
\begin{proof}
Let $f\!:(\Omega,cg_B)\rightarrow\CP^\infty$ be a K\"ahler immersion, we want to show that 
$c\gamma$ belongs to  $W(\Omega)$, i.e. 
 $\Kj^c$ is positive definite.
Since $\Omega$ is contractible it is not hard to see that there exists a sequence
 $f_j, j=0, 1\dots$ of holomorphic functions defined on $\Omega$, not vanishing simultaneously, such that the immersion 
 $f$ is given by $f(z)=[\dots ,f_j(z), \dots ],\ j=0,1 \dots$, where $[\dots ,f_j(z), \dots ]$ denotes the equivalence
 class in $l ^2(\C)$ (two sequences are equivalent if and only if they differ by the multiplication by a nonzero complex number). Let $x_1,\dots,x_n\in \Omega$. Without loss of generality (up to unitary transformation of ${\C}P^{\infty}$)
we can assume that  $f(0)=e_1$, where $e_1$ is the first vector of the canonical basis of $l ^2({\C})$, and $f(x_j)\notin H_0$, $\po\ j=1,\dots, n$.
Therefore, by Theorem \ref{induceddiast} and  Proposition \ref{diastdom},
we have:
$$c\,\dd^\Omega_0(z)=\log [V(\Omega)^c\,\Kj^c(z, z)]=\log\left(1+\sum_{j=1}^{\infty}\frac{|f_j(z)|^2}{|f_0(z)|^2}\right),\ \ z \in  \Omega\setminus f^{-1}(H_0),$$
that is:
\begin{equation}
V(\Omega)^c\,\Kj^c(x_{\alpha},x_{\beta})=1+\sum_{j=1}^\infty g_j(x_{\alpha})\overline{g_j(x_{\beta})},
\quad g_j=\frac{f_j}{f_0}.\nonumber
\end{equation}
Thus for every $(v_1, \dots v_n)\in \C^n$ one has:
$$\sum_{\alpha, \beta =1}^nv_{\alpha}\Kj ^c (x_{\alpha}, x_{\beta})\bar v_{\beta}=
\frac{1}{V(\Omega)^c}\sum_{k=0}^{\infty}|v_1g_k(x_1)+\cdots +v_ng_k(x_n)|^2\geq 0, g_0=1,$$
and hence the matrix  $(\Kj^c(x_{\alpha}, x_{\beta}))$ is positive semidefinite.

Conversely, 
assume that $c\gamma\in W(\Omega)$.
Then, by the very definition of Wallach set,
there exists a Hilbert space ${\mathcal H}_{c\gamma}$ whose reproducing kernel
is $\Kj^c=\sum_{j=0}^{\infty}|f_j|^2$, where $f_j$ is an orthonormal basis of 
${\mathcal H}_{c\gamma}$.
Then  the holomorphic map 
$f :\Omega\rightarrow l^2 (\C)\subset \CP^{\infty}$  constructed by using this orthonormal basis 
satisfies $f^*(g_{FS})=cg_B$. In order to prove that this map is equivariant  
write $\Omega =G/K$ where $G$ is the simple Lie group acting holomorphically and isometrically 
on $\Omega$ and $K$ is its isotropy group. For each $h\in G$ the map $f\circ h:(\Omega, cg_B)\f\CP^{\infty}$ is a full \K\ immersion and therefore by Calabi's rigidity (Theorem \ref{local rigidityb}) there exists
a unitary transformation $U_h$ of $\CP^{\infty}$ such that $f\circ h=U_h\circ f$ and we are done.
\end{proof}

\begin{remark}\rm
 In \cite{arazy} it is proven that  if $\eta$ belongs to $W(\Omega)\setminus \{0\}$ then  $G$ admits a representation in the Hilbert space ${\mathcal H}_{\eta}$. This is in accordance with our result. Indeed
if $c\gamma$ belongs to $W(\Omega)\setminus \{0\}$  then
the correspondence $h\mapsto U_h, h\in G$ defined in the last part of the  proof of Theorem \ref{wallach} is a representation of $G$.
\end{remark}
\begin{remark}\rm
Notice that Theorem \ref{thmsmall} for bounded symmetric domains  follows directly  by Theorem \ref{wallach} 
and Remark \ref{rchimm}. 
\end{remark}

\section{Exercises}
\begin{ExerciseList}
\Exercise 
Prove that the Bergman metric $g_B$ on $\Omega_4[3]$  is not resolvable.

({\em Hint: compute the first $9\times 9$ entries of the matrix of coefficients in the power expansion \eqref{powexdiastc} for the  diastasis function  given by (\ref{kernel4}) and (\ref{diastberg}) and show it is not positive semidefinite}).

\Exercise
Use Remark \ref{immersbsd}  and the previous exercise to  prove that up to biholomorphism, the only irreducible bounded symmetric domain
of complex dimension $n$ admitting a K\"ahler immersion into $l^2(\mathds{C})$ is  $\mathds{C}{\rm H}^n$.

\Exercise 
Let $G_{k, n}$ be the complex Grassmannian of $k$-planes in $\C^n$.
Let $M\subset \C^{n\times k}$ denote the open subset of matrices of rank $k$ and $\pi :M\rightarrow G_{k, n}$
the canonical projection which turns out to be  a holomorphic principal bundle with structure group $\GL (r, \C)$.
Let $Z$ be a holomorphic section of $\pi$ over an open subset $U\subset G_{k, n}$ and define a closed form $\omega_G$
of type $(1, 1)$ on $U$ by:
$$\omega_G =\frac{i}{2}\partial\bar\partial\log\det \left(\bar{Z}^tZ\right).$$
Show that:
\begin{itemize}
\item [$(a)\ $]
$\omega_G$ is a well-defined \K\ form on  $G_{k, n}$;
\item [$(b)\ $]
$(G_{k, n}, g)$ is a homogeneous \K\ manifold, where $g$ is the \K\ metric whose associated \K\  form is $\omega$;
\item [$(c)\ $]
the Pl\"{u}cker embedding, namely the map
$$p:G_{k, n}\rightarrow\p \left(\Lambda^k(\C^n)\right)\cong\CP^{\tiny{{n\choose k}-1}}, \gen \{v_1, \dots , v_k\}\ \mapsto v_1\wedge\cdots\wedge v_k,$$
is a \K\ immersion from $\left(G_{k, n}, g\right)$ into $\left(\CP^{\tiny{{n\choose k}-1}}, g_{FS}\right)$.
\end{itemize}
\Exercise 
Prove that the Segre embedding, namely the  map $\sigma\!:\CP^n\times\CP^m\f\CP^{(n+1)(m+1)-1}$ defined by:
$$\sigma([Z_0,\dots,Z_n],[W_0,\dots,W_m])\mapsto [Z_0W_0,\dots,Z_jW_k,\dots,Z_nW_m],$$
is a K\"ahler immersion.

\Exercise
Let $g_1$ (resp. $g_2$) be a projectively induced \K\ metric on a complex manifold $M_1$ (resp. $M_2$).
Prove that the \K\ metric $g_1\oplus g_2$ on $M_1\times M_2$ is projectively induced.

({\em Hint: generalize the previous exercise}).

 \Exercise\label{simplyflat} Prove that a not simply-connected flat \K\ manifold does not admit a global K\"ahler immersion into 
 $\mathds{C}{\rm P}^{N\leq\infty}$.
 
  ({\em Hint: cfr. Example \ref{cstar}}).

\Exercise Prove that  the hyperbolic metric $g$ on a compact  Riemannian surface $\Sigma_g$ of genus $g\geq 2$ 
is not projectively induced.

({\em Hint: use the fact  that  the universal covering map $\pi: \CH^1\rightarrow \Sigma_g$, satisfies $\pi^*\omega=\omega_{hyp}$}).

\Exercise\label{coherentstates}
Prove (\ref{pullbackdie}).
\end{ExerciseList}

\chapter{K\"ahler--Einstein manifolds}
A K\"ahler manifold $(M,g)$ is Einstein when there exists $\lambda\in \mathds{R}$ such that $\rho=\lambda\omega$, where $\omega$ is the K\"ahler form associated to $g$ and $\rho$ is its Ricci form. 
The constant $\lambda$ is called the {\em Einstein constant} and it turns out that  $\lambda={s}/{2n}$, where $s$ is the scalar curvature of the metric $g$ and $n$ the complex dimension of $M$ (as a general reference for this chapter see e.g. \cite{tian5}).
If $\omega =\frac{i}{2}
\sum _{j=1}^{n}g_{\alpha\bar{\beta}}
dz_{\alpha}\wedge d\bar{z}_{\bar{\beta}}$
is the local expression of $\omega$ on an open set $U$ with local coordinates $(z_1, \dots ,z_n)$ centered at some point $p$ then 
the Ricci form is the  $2$-form on $M$ of type $(1, 1)$ defined by
\begin{equation}\label{defricciform}
\rho =-i\partial\bar\partial\log\det g_{\alpha\bar\beta}.
\end{equation}
By the  $\partial\bar\partial$-Lemma  (and by shrinking $U$ if necessary) this is equivalent  to  require that 
\begin{equation}\label{monge}
\det(g_{\alpha\bar\beta})=e^{-\frac{\lambda}{2}\dd_0(z)+f+\bar f},
\end{equation}
for some holomorphic function $f$, where $\dd_p$ denotes Calabi's diastasis function centered at $p$.

In this chapter we study \K\ immersions of \K--Einstein manifolds into complex space forms.
We begin describing in the next section the work of M. Umehara \cite{umehara2} which completely classifies K\"ahler--Einstein manifolds admitting a K\"ahler immersion into the finite dimensional complex hyperbolic or flat space.
In Section \ref{finitecpn} we summarize
what is known about   K\"ahler immersions of K\"ahler--Einstein manifolds into the finite dimensional complex projective space.

\section[\K\ immersions of KE manifolds into $\CH^{N}$ or $\C^N$]{\K\ immersions of K\"ahler--Einstein manifolds into $\CH^{N}$ or $\C^N$}\label{umeharasec}
In this section we summarize the results of M. Umehara in \cite{umehara2} which determine the nature of K\"ahler--Einstein manifolds admitting a K\"ahler immersion into $\CH^N$ or $\C^N$, for $N$
finite. 
\begin{theor}[M. Umehara]\label{umehara}
Every K\"ahler--Einstein manifold K\"ahler immersed into $\C^N$ or $\CH^N$ is totally geodesic.
\end{theor}

In order to prove this theorem we need the following lemma, achieved by Umehara himself in \cite{umehara}. Let $M$ be a K\"ahler manifold and denote by $\Lambda(M)$ the associative algebra of $\R$-linear combinations of real analytic functions of the form $h\bar k+\bar h k$ for $h$, $k\in\ol(M)$.
The importance of $\Lambda(M)$ for our purpose relies on the fact that given a K\"ahler map $f\!:M\rightarrow \ell^2(\mathds{C})$, if $f$ is full then $|f|^2\notin \Lambda(M)$ (cfr. \cite[p. 534]{umehara}).
\begin{lem}\label{lemmaumehara}
Let $f_1,\dots, f_N$ be non-constant holomorphic functions on a complex manifold $M$ such that for all $j=1,\dots, N$, $f_j(p)=0$ at some $p\in M$. Then:
\begin{itemize}
\item[(1)] $e^{\sum_{j=1}^N|f_j|^2}\notin \Lambda(M)$,
\item[(2)] $\log(1-\sum_{j=1}^N |f_j|^2)\notin \Lambda(M),$
\item[(3)] $(1-\sum_{j=1}^N |f_j|^2)^{-a}\notin \Lambda(M), \quad (a>0)$.
\end{itemize}
\end{lem}
\begin{proof}
We prove first $(1)$ and $(2)$. Consider the power expansions (cfr. Exercises \ref{chinell} and \ref{cincp}):
$$
e^{\sum_{j=1}^N|f_j|^2}-1=\sum_{j=1}^\infty\frac{|f^{m_j}|^2}{m_j!},
$$
$$
-\log\left(1-\sum_{j=1}^N |f_j|^2\right)=\sum_{j=1}^\infty\frac{(|m_j|-1)!}{m_j!}|f^{m_j}|^2,
$$
which, since $f_j(p)=0$ for any $j=1,\dots, N$, converge in a  sufficiently small neighbourhood $U$ of $p$. We can then define two full K\"ahler maps $\varphi$, $\psi\!:U \rightarrow \ell^2(\C)$ by:
$$
\varphi_j:=\frac{f^{m_j}}{\sqrt{m_j!}},\quad \psi_j:=\sqrt{\frac{(|m_j|-1)!}{m_j!}}f^{m_j},
$$
from which follows:
\begin{equation}\label{explambda}
e^{\sum_{j=1}^N|f_j|^2}=1+|\varphi|^2\notin \Lambda(M),
\end{equation}
\begin{equation}\label{loglambda}
\log\left(1-\sum_{j=1}^N |f_j|^2\right)=-|\psi|^2\notin \Lambda(M).
\end{equation}
In order to prove $(3)$, write $(1-\sum_{j=1}^N |f_j|^2)^{-a}=e^{-a\log(1-\sum_{j=1}^N |f_j|^2)}$ and use \eqref{loglambda} to get:
$$
\left(1-\sum_{j=1}^N |f_j|^2\right)^{-a}=1+\sum_{|m_j|=1}^\infty\sum_{j=1}^\infty\frac{|(\sqrt{a})^{|m_j|}\psi^{m_j}|^2}{m_j!}.
$$
If we arrange to order the $\left(\sqrt{a}\right)^{|m_j|}\psi^{m_j}$ as $|m_j|$ increases, we get again a full map $\tilde \psi=\left(\tilde \psi_1,\dots, \tilde \psi_j,\dots\right)$ of $U$ into $\ell^2(\C)$ and conclusion follows.
\end{proof}

Let us prove first Umehara's result in the case when the ambient space is $\C^N$.
\begin{proof}[Proof of the first part of Theorem \ref{umehara}]
Let $(M,g)$ be an $n$-dimensional K\"ahler--Einstein manifold K\"ahler immersed into $\C^N$, $\omega$ the K\"ahler form associated to $g$ and $\rho$ its  Ricci form given by (\ref{defricciform}). Let $z=(z_1,\dots,z_n)$ be a local coordinate system on $U\subset M$ such that $0\in U$ and let $$\omega_{|_U}=\frac{i}{2}\de\bar\de\,\dd^M_0,$$ where $\dd_0^M$ is the diastasis for $g$ on $U$ centered at $0$. The Gauss' Equation
\begin{equation}\label{gauss}
\rho\leq 2b(n+1)\omega,
\end{equation}
where $b$ is the holomorphic sectional curvature of the ambient space (see for example \cite[p. 177]{kono}), gives for $b=0$ $\rho\leq 0$, where the equality holds if and only if $M$ is totally geodesic. Hence, if $M$ is not totally geodesic, $\rho$ is negative definite and the Einstein's Equation $\rho=\lambda\omega$ implies $\lambda<0$. Up to homothetic transformations of $\C^N$ we can suppose $\lambda=-1$.

Since $M$ admits a K\"ahler immersion into $\C^N$, by Proposition \ref{induceddiast} there exists $f_1,\dots ,f_N$ holomorphic functions such that:
\begin{equation}
\dd_0^M(z)=\sum_{j=1}^N|f_j(z)|^2.\nonumber
\end{equation}
Thus, by previous lemma we have $e^{\dd_0^M}\notin\Lambda(M)$. On the other hand, by  Equation (\ref{monge}) with $\lambda=-1$, the function $\log\det(g_{\alpha\bar\beta})$ is a K\"ahler potential for $g$, hence we have:
$$
\dd^M_0(z)=h+\bar h+\log\det(g_{\alpha\bar\beta}),
$$
for a holomorphic function $h$. Hence:
\begin{equation}
e^{\dd_0^M}=|e^h|^2\det(g_{\alpha\bar\beta}).\nonumber
\end{equation}
Since $\det(g_{\alpha\bar\beta})\in\Lambda(M)$, for it is a real valued function being the matrix $(g_{\alpha\bar\beta})$ Hermitian, we get the contradiction  $e^{\dd_0^M}\in \Lambda(M)$.
\end{proof}
Before proving the second part of Umehara's theorem we need the following lemma:
\begin{lem}[M. Umehara]
Let $M$ be a complex $n$-dimensional manifold and let $(z_1,\dots, z_n)$ be a local coordinate system on an open set $U\subset M$. If $f\in\Lambda(U)$ then:
$$
f^{n+1}\det\left(\frac{\de^2 \log f}{\de z_\alpha\de \bar z_\beta}\right)\in\Lambda(U).
$$
\end{lem}
\begin{proof}
Let us write $f_\alpha$ for $\de f/\de z_{\alpha}$, $f_{\bar\beta}$ for $\de f/\de \bar z_{\beta}$ and $f_{\alpha\bar\beta}$ for $\de^2 f/\de z_{\alpha}\de \bar z_{\beta}$. We have:
$$
\frac{\de^2 \log f}{\de z_\alpha\de \bar z_\beta}=\frac{f_{\alpha\bar\beta}}{f}-\frac{f_\alpha f_{\bar\beta}}{f^2},
$$
thus we get:
{\small\begin{equation}
\begin{split}
f^{n+1}\det\left(\frac{\de^2 \log f}{\de z_\alpha\de \bar z_\beta}\right)&=f\det\left(f_{\alpha\bar\beta}-\frac{f_\alpha f_{\bar\beta}}{f}\right)=f\det\left(
\begin{array}{cccc}
 f_{1\bar 1}-f_1 f_{\bar 1}/f &  \dots &   f_{1\bar n}-f_1 f_{\bar n}/f&0 \\
 \vdots &   &\vdots & \vdots \\
   f_{n\bar 1}-f_n f_{\bar 1}/f &  \dots &   f_{n\bar n}-f_n f_{\bar n}/f&0 \\
 f_{\bar 1}/f &  \dots &  f_{\bar n}/f&1 
\end{array}
\right)\\
&=f\det\left(
\begin{array}{cccc}
 f_{1\bar 1} &  \dots &   f_{1\bar n}&f_1 \\
 \vdots &   &\vdots & \vdots \\
   f_{n\bar 1}&  \dots &   f_{n\bar n}&f_n \\
 f_{\bar 1}/f &  \dots &  f_{\bar n}/f&1 
\end{array}
\right)=\det\left(
\begin{array}{cccc}
 f_{1\bar 1} &  \dots &   f_{1\bar n}&f_1 \\
 \vdots &   &\vdots & \vdots \\
   f_{n\bar 1}&  \dots &   f_{n\bar n}&f_n \\
 f_{\bar 1}&  \dots &  f_{\bar n}&f 
\end{array}
\right).
\end{split}\nonumber
\end{equation}}
Hence:
$$
f^{n+1}\det\left(\frac{\de^2 \log f}{\de z_\alpha\de \bar z_\beta}\right)\in\Lambda(U),
$$
for it is finitely generated by holomorphic and antiholomorphic functions on $U$ and it is real valued, because the matrix $\left(\de^2 \log f/\de z_\alpha\de \bar z_\beta\right)$ is Hermitian.
\end{proof}
We can now prove the second part of Theorem \ref{umehara}.
\begin{proof}[Proof of the second part of Theorem \ref{umehara}]
Let $(M,g)$ be an $n$-dimensional K\"ahler--Einstein manifold K\"ahler immersed into $\CH^N$. Comparing Gauss' Equation (\ref{gauss}) with $b<0$ and Einstein's Equation $\rho=\lambda\omega$, we get that the Einstein constant $\lambda$ is negative. Let $(z_1,\dots, z_n)$ be local coordinates on an open set $U\subset M$ centered at $p\in U$. On $U$ the Monge--Amp\`ere Equation (\ref{monge}) for $g$ reads:
$$
e^{-\frac{\lambda}{2}D^M_0(z)}=|e^h|^2\det(g_{\alpha\bar\beta}),
$$
for some holomorphic function $h$. By Proposition \ref{induceddiast}, for some holomorphic functions $\varphi_1,\dots,\varphi_N$ that can be chosen to be zero at the origin, we have on $U$:
\begin{equation}
\dd^M_0(z)=-\log(1-\sum_{j=1}^N|\varphi_j(z)|^2).\nonumber
\end{equation}
Setting $f=1-\sum_{j=1}^N|\varphi_j|^2$ we get:
\begin{equation}
\det(g_{\alpha\bar\beta})=(-1)^n\det\left(\frac{\de^2\log f}{\de z_\alpha\de \bar z_\beta}\right).\nonumber
\end{equation}
Thus:
\begin{equation}
f^{\frac{\lambda}{2}}=(-1)^n|e^h|^2\det\left(\frac{\de^2\log f}{\de z_\alpha\de \bar z_\beta}\right),\nonumber
\end{equation}
and hence:
\begin{equation}
f^{\frac{\lambda}{2}+n+1}=(-1)^n|e^h|^2f^{n+1}\det\left(\frac{\de^2\log f}{\de z_\alpha\de \bar z_\beta}\right).\nonumber
\end{equation}
By previous lemma we obtain:
\begin{equation}
f^{\frac{\lambda}{2}+n+1}=\left(1-\sum_{j=1}^N|\varphi_j(z)|^2\right)^{\frac{\lambda}{2}+n+1}\in\Lambda(U),\nonumber
\end{equation}
and by $(3)$ of Lemma \ref{lemmaumehara} we get $\frac{\lambda}{2}+n+1\geq 0$. On the other hand, Gauss' Equation (\ref{gauss}) implies $n+1+\frac{\lambda}{2}\leq 0$. Thus $\lambda=-2(n+1)$, and $M$ is totally geodesic.
\end{proof}

Regarding the existence of a K\"ahler immersion of a \K\ manifold $(M,g)$ into $\CH^\infty$ and $l^2(\C)$, Umehara's result cannot be extended to that cases, as one can see simply considering the K\"ahler immersion (\ref{chinell}) given by Calabi of $\CH^n$ into $l^2(\C)$. Nevertheless, we conjecture that this is the only exception:
\begin{conj}\label{conjchell}
If a K\"ahler--Einstein manifold $(M,g)$ admits a K\"ahler immersion into $\CH^\infty$ or $l^2(\C)$, then either $(M,g)$ is totally geodesic or $(M,g)=(\CH^{n_1}\times\cdots\times\CH^{n_r},c_1g_{hyp}\oplus\cdots\oplus c_r g_{hyp})$ for positive constants $c_1,\dots, c_r$ and some $r\in\nat$.
\end{conj}

\section[KE manifolds into $\CP^N$: the Einstein constant]{\K\ immersions of KE manifolds into $\CP^N$: the Einstein constant}

We summarize in this section the work of D. Hulin \cite{hulin,hulinlambda} that studies K\"ahler--Einstein manifolds K\"ahler immersed into $\CP^N$ in relation with the sign of the Einstein constant. By the Bonnet--Myers' Theorem it follows that if the Einstein constant of a complete  K\"ahler--Einstein manifold $M$ is positive then $M$ is compact. D. Hulin proves that  in the case when $M$ is projectively induced the converse is also true:
\begin{theor}[D. Hulin, \cite{hulinlambda}]\label{lambdapos}
Let $(M,g)$ be a compact (connected) K\"ahler--Einstein manifold K\"ahler immersed into $\CP^N$. Then the Einstein constant is strictly positive.
\end{theor}
\begin{proof}
Let  $p$ be a point in $M$, up to a unitary transformation of $\CP^N$  we can assume that  $\varphi (p)=p_0=[1, 0. \dots , 0]$.
Take Bochner's  coordinates  $(w_1,\dots ,w_n)$ in a neighbourhood 
$U$ of $p$ which we take small enough to be contractible. Since the \K\ metric $g$ is Einstein with Einstein constant $\lambda$,
the volume form of $(M, g)$ reads on 
$U$ as:
\begin{equation}\label{voleucl}
\frac{\omega^n}{n!}=\frac{i^n}{2^n}
e^{-\frac{\lambda}{2}D_p+f+\bar f}
dw_1\wedge d\bar w_1\wedge\dots\wedge 
dw_n\wedge d\bar w_n\, \, ,
\end{equation}
where $f$ is a holomorphic
function on $U$
and
$D_p=\varphi^{-1}(D_{p_0})$
is the diastasis on $p$
(cfr. Prop.
\ref{induceddiast}), where
$D_{p_0}$ is the diastasis of $\CP^N$
globally defined in $\CP^N\setminus H_0$, $H_0=\{Z_0\neq 0\}$ (cfr. (\ref{diastcp})).
 
We claim that $f+\bar f=0$.
Indeed, observe that:
$$\frac{\omega ^n}{n!}=\frac{i^n}{2^n}
\det \left(\frac{\partial ^2 D_p}
{\partial w_{\alpha} \partial \bar w_{\beta}}\right)
dw_1\wedge d\bar w_1\wedge\dots\wedge 
dw_n\wedge d\bar w_n .$$
By the very definition
of Bochner's coordinates it is easy
to check that 
the expansion of 
$\log\det (\frac{\partial ^2 D_p}
{\partial w_{\alpha} \partial \bar w_{\beta}})$
in the 
$(w,\bar w)$-coordinates
contains only mixed terms
(i.e. of the form
$w^j\bar w^k, j\neq 0, k\neq 0$).
On the other hand
by formula  (\ref{voleucl}):
$$-\frac{\lambda}{2} D_p + f + \bar f=
\log\det \left(\frac{\partial ^2 D_p}
{\partial w_{\alpha} \partial \bar w_{\beta}}\right).$$
Again by the definition of the Bochner's 
coordinates this forces
$f + \bar f$ to be zero; hence:
\begin{equation}\label{mongebochner}
\det \left(\frac{\partial ^2 D_p}
{\partial w_{\alpha} \partial \bar w_{\beta}}\right)=e^{-\frac{\lambda}{2}\dd_p(w)},
\end{equation}
proving our claim.
By Theorem \ref{bochnergraph} 
there exist affine coordinates
$(z_1, \dots , z_N)$
on $X=\CP^N\setminus H_0$, 
satisfying:
$$z_1|_{\varphi(U)}=w_1,\dots ,z_n|_{\varphi (U)}=w_n.$$
Hence, by 
formula (\ref{voleucl})
(with $f+\bar f=0$),
the
$n$-forms 
$\frac{\omega_{FS}^n}{n!}$
and
$e^{-\frac{\lambda}{2}D_{p_0}}
dz_1\wedge d\bar z_1\wedge\dots\wedge 
dz_n\wedge d\bar z_n$ 
globally defined on $X$
agree on the open set
$\varphi (U)$.
Since they are
real analytic  they
must agree
on the
connected open set
$\hat M=\varphi (M)\cap X$, i.e.:
\begin{equation}\label{eqforms}
\frac{\omega_{FS}^n}{n!}=\frac{i^n}{2^n}
e^{-\frac{\lambda}{2}
D_{p_0}}
dz_1\wedge d\bar z_1\wedge\dots\wedge 
dz_n\wedge d\bar z_n.
\end{equation}

Since $\frac{\omega_{FS}^n}{n!}$ is a volume form on $\hat M$ we deduce that
the restriction of the projection map:
$$\pi :X\cong{\C}^N
\rightarrow {\C}^n:
(z_1,\dots ,z_N)\mapsto 
(z_1,\dots ,z_n)$$
to $\hat M$ 
is open.
Since it
is also algebraic 
its image 
contains a Zariski open
subset of ${\C}^n$
(see \cite[Theorem 13.2]{bo}),
hence 
its euclidean volume,
$vol_{eucl}(\pi (\hat M))$,  
has to be infinite.
Suppose now
that the Einstein  constant 
of $g$ is non-positive.
By formula 
(\ref{eqforms})
and  by the fact that $D_{p_0}$ is non-negative,
we get 
$vol (\hat M, g)\geq
vol_{eucl}(\pi (\hat M))$
which is the desired contradiction,
being the volume of $M$
(and hence that of 
$\hat M$)
finite.
\end{proof}

Consider now the following construction. Let $(M,g)$ be an $n$-dimensional K\"ahler manifold which admits a K\"ahler immersion $F\!:M\rightarrow \CP^N$ into $\CP^N$ and consider the Pl\"ucker embedding:
$$
i\!:{\rm Gr}(n,\CP^{N})\rightarrow \mathds{P}(\wedge^{n+1}\mathds{C}^{N+1}),\quad {\rm span}(e_{j_1},\dots,e_{j_r})\mapsto [e_{j_1}\wedge\dots\wedge e_{j_r}].
$$
where ${\rm Gr}(n,\CP^{N})$ is the Grassmanian of $n$-dimensional projective spaces in $\CP^{N}$ and $(e_0,e_1,\dots, e_N)$ is a unitary frame of $\mathds{C}^{N+1}$.
The Gauss map $\gamma\!:M\rightarrow  \mathds{P}(\wedge^{n+1}\mathds{C}^{N+1})$ takes a point $p\in M$ to  the $n$-dimensional projective space in $\CP^N$ tangent to $M$ at $p$. Setting Bochner coordinates $z=(z_1,\dots, z_n)$ around $p\in M$, by Theorem \ref{bochnergraph} we can write $F(z)=[1,z_1,\dots, z_n,f_1,\dots, f_{N-n}]\in \CP^N$. The vectors:
$$
v_0(z)=e_0+\sum_{j=1}^nz_je_j+\sum_{j=1}^{N-n}f_j(z)e_{n+j},
$$
$$
v_k(z)=e_k+\sum_{j=1}^{N-n}\frac{\partial f_j}{\partial z_j}(z)e_{n+j},\quad 1\leq k\leq n,
$$
span a complex space $\C^{N+1}$ whose projection is the projective space tangent to $M$ at $[1,z_1,\dots, z_n,f_1,\dots, f_{N-n}]$, and thus they satisfy $\gamma(z)=v_0\wedge\dots\wedge v_n$. It follows that $\gamma(z)=[1,\nabla f,\varphi ]$ for $\nabla f=(\partial_j f_k)_{j=0,\dots,n; k=1,\dots, N-n}$ and for suitable $\varphi=(\varphi_\alpha)_{\alpha=1,\dots, s}$, $s={N+1 \choose n+1}-1-(n+1)(N-n)$.
One has (see S. Nishikawa \cite{nishikawa} and also M. Obata \cite{obata}
for the case of real setting and ambient space of constant curvature):
\begin{equation}\label{pluckerinduced}
\gamma^*(G_{FS})=(n+1)g-\frac12\ric_g,
\end{equation}
where we denote by $G_{FS}$ the Fubini--Study metric on $\mathds{P}(\wedge^{n+1}\mathds{C}^{N+1})$.

When $(M,g)$ is K\"ahler--Einstein with Einstein constant $\lambda$, from \eqref{pluckerinduced} we get $\gamma^*(G_{FS})=\left(n+1-\frac\lambda2 \right)g$, which by \eqref{induceddiast} and by the expression of the Fubini--Study's diastasis implies:
\begin{equation}\label{gammainduced}
\left(1+\sum_{j=1}^n|z_j|^2+\sum_{k=n+1}^N|f_k|^2\right)^{n+1-\frac\lambda2}=1+\sum_{j,k=1}^n|\partial_jf_k|^2+\sum_{\alpha=1}^s|\varphi_\alpha|^2,
\end{equation}
which we write shortly:
\begin{equation}\label{gammainducedshort}
\left(1+|z|^2+|f|^2\right)^{n+1-\frac\lambda2}=1+|\nabla f|^2+|\varphi|^2.
\end{equation}

In the sequel we will denote by $H$ a hyperplane of  $\CP^N$ and by $H_p$ the hyperplane at infinity relative to the point $F(p)$, for $p\in M$.

\begin{lem}\label{lambdairrationalimage}
Let $(M,g)$ be a K\"ahler--Einstein manifold with Einstein constant $\lambda< 0$ and let $F\!:M\rightarrow \CP^N$, $N< \infty$, be a full K\"ahler immersion. If $\lambda\notin \mathds Q$ then $F(M)\subset \CP^N\setminus H$.
\end{lem}
\begin{proof}
Assume by contradiction that there exist two points $p$, $q\in M$ such that $F(q)\in H_p$. Since the immersion is full and $H_p\cap H_q$ has codimension $1$ in $\CP^N$, we can further choose $x\in M$ such that $F(p)$, $F(q)\notin H_x$. By Exercise \ref{bochnerHp} we can set Bochner coordinates $(z)=(z_1,\dots, z_n)$ centered at $x$ in the whole $M\setminus F^{-1}(H_x)$. From \eqref{gammainducedshort}, duplicating the variables and evaluating at $\bar z=q$ (to simplify the notations we identify a point with its coordinates) we get:
\begin{equation}\label{holq}
\left(1+ z\bar q+ f(z)\overline{f(q)}\right)^{n+1-\frac\lambda2}=1+(\nabla f)(z)\overline{(\nabla f)(q)}+\varphi(z)\overline{\varphi(q)},
\end{equation}
where $F(z)=[1,z,f(z)]$ (see Theorem \ref{bochnergraph} or the discussion above). Observe that since $F(q)\in H_p$, from $F(p)=[1:p:f(p)]$, $F(q)=[1:q:f(q)]$ we get:
$$
1+ p\bar q+ f(p)\overline{f(q)}=0.
$$
Thus, the RHS of \eqref{holq} is a holomorphic function equal to:
$$
\left(z-p\right)^{n+1-\frac\lambda2}h(z)^{n+1-\frac\lambda2}
$$
for some suitable $h(z)$. Since the order of a zero of a holomorphic function must be rational, we get the desired contradiction $\lambda\in \mathds{Q}$.
\end{proof}
\begin{lem}\label{lambdairrationalbounded}
Let $(M,g)$ be a K\"ahler--Einstein manifold with Einstein constant $\lambda< 0$ and let $F\!:M\rightarrow \CP^N$, $N\leq \infty$, be a full K\"ahler immersion. If $\lambda\notin \mathds Q$ then $F(M)\subset \CP^N\setminus H$ is bounded.
\end{lem}
\begin{proof}
Let $p\in M$. Since by Lemma \ref{lambdairrationalimage} $F(M)\subset \C^N$, Bochner coordinates $(z)=(z_1,\dots, z_n)$ around $p$ extends to the whole $M$. Assume $F(M)$ is not bounded, i.e. any open neighbourhood $\mathcal U$ of $ \CP^N\setminus H$ is such that $\mathcal U\cap F(M)\neq \emptyset$. Consider a path $t\mapsto F(x_t)$ in $F(M)$ which diverges as $t$ increases. Since $F(x_t)=[1,x_t,f(x_t)]$ (where to simplify the notations we identify a point with its coordinates), this means that either $x_t$ or $f(x_t)$ diverges. If $x_t$ diverges, then since $M$ is complete the limit point $x_\infty$ belongs to $M$ and $F(x_\infty)$ would be a point of both $F(M)$ and $ \CP^N\setminus H$. If $f(x_t)$ diverges and $x_t$ does not, then $[1,x_t,f(x_t)]$ approaches $[0,0,b]$ for a suitable nonvanishing $(N-n)$-vector $b$ as $t$ increases. Thus, we can conclude by showing that there exists a neighbourhood of $[0,0,b]\in \CP^N$ which does not meet $F(M)$. 
Since by Lemma \ref{lambdairrationalimage} for any $p,q\in F(M)$, $p\notin H_q$, it is enough to show that for each $(\alpha,\beta,\gamma)\in \C^{N+1}$ close enough to the origin the function:
$$
\Phi\!:\C^{N+1}\times \C\rightarrow \C,\quad ((\alpha,\beta,\gamma),t)\mapsto\langle (1,tz,f(tz)),\overline{(\alpha,\beta,b+\beta)}\rangle,
$$
satsfies $\Phi((\alpha,\beta,\gamma),t_0)=0$ for some $t_0\in \C$.
In order to do so, observe that since the function $\Phi_0\!:\C\rightarrow \C$ defined by $\Phi_0(t):=\Phi((0,0,0),t)=f(tz)\bar b$ is a holomorphic function not vanishing everywhere, the image $\Phi_0(U)\in \C$ of an open neighbourhood $U\in \C$ of the origin is still an open neighbourhood of the origin. Let $c$ be a closed curve in $U$ such that its image is contained in $\Phi_0(U)$ and turns around the origin. In the compact set with $\Phi_0(c)$ as boundary both $|t_0z|$ and $|f(t_0z)|$ are bounded. For sufficiently small $(\alpha,\beta,\gamma)\in \C^{N+1}$, the image of $c$ through $\Phi_{(\alpha,\beta,\gamma)}:=\Phi((\alpha,\beta,\gamma),\cdot)$ is a closed curve contained in $\Phi_{(\alpha,\beta,\gamma)}(U)$ and still turning around the origin. Thus there exists a point $t_0\in \C$ such that
$\langle (1,t_0z,f(t_0z)),\overline{(\alpha,\beta,b+\beta)}\rangle=0$, and we are done.
\end{proof}

\begin{theor}[D. Hulin, \cite{hulin}]\label{hulinrational}
Let $(M,g)$ be a complete K\"ahler--Einstein manifold which admits a K\"ahler immersion $F\!: M\f \CP^N$ into $\CP^N$. Then the Einstein constant $\lambda$ is rational. Further, if the immersion is full and we write $\lambda=2p/q>0$, where $p/q$ is irreducible, then $p\leq n+1$ and if $p=n+1$ (resp. $p=n$), then $(M,g)=(\CP^n,qg_{FS})$ (resp.  $(M,g)=(Q_n,qg_{FS})$).
\end{theor}
\begin{proof}
Assume first $\lambda>0$. Since $M$ is complete, by Bonnet--Myers' theorem  $M$ is compact. Combining the fact that  $\frac{1}{\pi}\omega_{FS}$ is an integral \K\ form (since it  represents the first Chern class of the hyperplane bundle of $\CP^N$) and $g$ is projectively induced we deduce  that 
$\frac{1}{\pi}\omega$ is integral, where $\omega$ is the \K\ form associated to $g$. Moreover,  $\frac{1}{\pi}\rho$ is an integral form 
since it represents the first Chern class of the canonical bundle over $M$. Then the Einstein condition $\rho=\lambda \omega$ forces 
$\lambda$ to be rational.

Let now $\lambda<0$. Assume by contradiction that $\lambda\notin \mathds{Q}$. Then by Lemma \ref{lambdairrationalimage}, $F(M)\subset \C^N\subset \CP^N$ and by Lemma \ref{lambdairrationalbounded} $F(M)$ is bounded. Set Bochner coordinates $(z)=(z_1,\dots, z_n)$ around a point $p\in M$ and write $F(z)=[1,z,f]$ (see Theorem \ref{bochnergraph} and the discussion above for the notations). Consider the path $\mathds{R}^+\rightarrow \C^N$, $t\mapsto (t,0,\dots, 0)$, and observe that since $0\in F(M)$, for small values of $t>0$, $(t,0,\dots, 0)\in F(M)$. Set $T=\sup_{t} \{(x,0,\dots, 0)\in F(M)\ {\textrm{for all}}\ x<t\}$. Since $F(M)$ is bounded, we have that the image $[1:(t,0,\dots,0):f(t,0,\dots, 0)]$ is bounded and thus, $T<+\infty$ and $f(t,0,\dots, 0)$ is bounded for all $t<T$. By \eqref{gammainducedshort} we get:
$$
| f'|^2<\left(1+|t|^2+|f|^2\right)^{n+1-\frac\lambda2},
$$
i.e. also $f'$ is bounded and so is the length of the curve in $F(M)$ defined by:
$$
\gamma\!:[0,T)\rightarrow F(M),\quad t\rightarrow [1,(t,0,\dots,0),f((t,0,\dots,0))],
$$
contradicting the completness of $M$ and the existence of global Bochner coordinates given by Exercise \ref{bochnerHp}.
\end{proof}

\section[KE manifolds into $\CP^N$: codimension $1$ and $2$]{\K\ immersions of KE manifolds into $\CP^N$: codimension $1$ and $2$}\label{finitecpn}
The problem of classifying K\"ahler--Einstein manifolds admitting a K\"ahler immersion into the finite dimensional complex projective space $\CP^N$ has been partially solved by S. S. Chern \cite{ch} and K. Tsukada \cite{ts}, that determined all the projectively induced K\"ahler--Einstein manifolds in the case when the codimension is respectively $1$ or $2$ (see Theorem \ref{ct} below). We follow essentialy the proof of D. Hulin given in \cite{hulin}, which makes use of the diastasis function.


Let $(M,g)$ be a K\"ahler--Einstein $n$-dimensional manifold with Einstein constant $\lambda$ and let $F\!:M\rightarrow \CP^{n+2}$ be a K\"ahler immersion. Setting Bochner coordinates $z=(z_1,\dots, z_n)$ around a point $p\in M$, due to Theorem \ref{bochnergraph} we can write $F(z)=[1,z_1,\dots, z_n,f_1,f_2]$. Let us denote by $Q_j$ and $B_j$ ($j=1,2$) the homogeneous part of $f_j$ of degree $2$ and $3$ respectively. 
 From \eqref{gammainduced} setting $\ell=n+1-\frac\lambda2$, follows: 
\begin{equation}\label{quadratic}
||\nabla Q_1||^2+||\nabla Q_2||^2=\ell ||z||^2,\quad \langle \nabla Q_1\overline{,\nabla B_1}\rangle+\langle \nabla Q_2,\overline{\nabla B_2}\rangle=0;
\end{equation}
\begin{equation}\label{quadratic2}
\sum_{j=1}^2||\nabla B_j||^2=\left(\ell-1\right)\sum_{j=1}^2|Q_j|^2+\frac{\ell\left(\ell-1\right)}2\sum_{j=1}^n|z_j|^4-\frac12\sum_{j,k=1}^{2}|\nabla Q_j\wedge \nabla Q_k|^2,
\end{equation}
where $|\nabla Q_j\wedge \nabla Q_k|^2=||\nabla Q_j||^2||\nabla Q_k||^2-|\langle \nabla Q_j,\overline{\nabla Q_k}\rangle|^2$.

We begin with the following lemma.
\begin{lem}[D. Hulin, \cite{hulin}]\label{coordQ}
Let $(M,g)$ be an $n$-dimensional K\"ahler manifold admitting a K\"ahler immersion $F\!:M\rightarrow \CP^{n+2}$ into $\CP^{n+2}$ and let $Q_1$, $Q_2$ as above. One can choose a unitary frame $(\nu_1,\nu_2)$ of the normal space to $T _pM$ and a coordinate system $(z_1,\dots, z_n)$ around $p\in M$, such that:
$$
Q_1=\frac12\sum_{j=1}^n\alpha_jz_j^2,\quad Q_2=\frac12\sum_{j=1}^na_jz_j^2,
$$
with $\alpha_j\geq0$ and $a_j\in \C$, $j=1,\dots, n$.
\end{lem}
\begin{proof}
We proceed by induction on $n$. When $n=1$ there is nothing to prove. Assume $n>1$ and choose $(\nu_1,\nu_2)$ such that $Q_2$ has rank less than $n$. Observe that this choice is always possible. In fact, if $Q_2$ has rank $n$ then the polynomial $Q_1+tQ_2$ is not constant in $t$ and has at least one zero $t=t_0$.  The unitary transformation of the normal space to $T_pM$ given by:
$$
\left(\begin{array}{cc}\frac1{\sqrt{1+t_0^2}}&\frac{t_0}{\sqrt{1+t_0^2}}\\-\frac{t_0}{\sqrt{1+t_0^2}}&\frac1{\sqrt{1+t_0^2}}\end{array}\right),
$$
moves $Q_1$ into $Q_1'=\frac1{\sqrt{1+t_0^2}}\left(Q_1+t_0Q_2\right)$, whose rank is less than $n$ since $\det(Q_1+t_0Q_2)=0$.  
Up to a unitary transformation $T\in U(n)$ we have:
$$
Q_1=\frac12\sum_{j=1}^{n}\alpha_jz_j^2,
$$
and from \eqref{quadratic} we get:
$$
||\nabla Q_2||^2=\sum_{j=1}^{n}(\ell-\alpha_j^2)|z_j|^2.
$$
Write $Q_2=\sum_{j=1}^nl_j(z)^2$, where $l_j(z)$ are homogeneous polynomials of degree $1$ in $z_1,\dots, z_n$. By hypothesis there exists $\xi\in\ker Q_2$, $\xi\neq 0$, such that $ l_j(\xi)=0$, for all $j=1,\dots, n$, which implies that also $\ker (||\nabla Q_2||^2)$ is not trivial and thus one between $\alpha_j$'s must be equal to $\sqrt \ell$. Assume $\alpha_n$ is. Then we have:
$$
Q_1=Q_1'(z_1,\dots, z_{n-1})+\ell |z_n|^2,\quad Q_2=Q_2'(z_1,\dots, z_{n-1}),
$$
for $Q_1'(z_1,\dots, z_{n-1})=\frac12\sum_{j=1}^{n-1}\alpha_jz_j^2$ and $Q_2'(z_1,\dots, z_{n-1})$ a quadratic form in $z_1,\dots, z_{n-1}$. We can apply the inductive hypothesis to $Q_1'$ and $Q_2'$ and performe a change of coordinates which leaves $z_n$ invariant and modifies $z_1,\dots, z_{n}$ in such a way that:
$$
Q_1'=\frac12\sum_{j=1}^n\alpha_jz_j^2,\quad Q_2'=\frac12\sum_{j=1}^na_jz_j^2,
$$
and conclusion follows.
\end{proof}
\begin{theor}[S. S. Chern \cite{ch}, K. Tsukada \cite{ts}]\label{ct}
Let $(M,g)$ be an $n$-dimensional K\"ahler--Einstein manifold ($n\geq 2$). If $(M,g)$ admits a K\"ahler immersion into $\CP^{n+2}$, then $M$ is either totally geodesic or the quadric $Q_n$ in $\CP^{n+1}$ (which is totally geodesic in $\CP^{n+2}$), with homogeneous equation $Z_0^2+\dots+Z_{n+1}^2=0$ .
\end{theor}
\begin{proof}
Assume first that $Q_2=cQ_1$. Up to unitary transformation of $\CP^{n+2}$ we can assume $c=0$. Then, from \eqref{quadratic} we get $B_1=0$ and $||\nabla Q_1||^2=\ell ||z||^2$. Up to a unitary transformation $T\in U(n)$, we can then assume $Q_1=\frac{\sqrt{\ell}}{2}(z_1^2+\dots +z_n^2)$, and substituting into \eqref{quadratic2} we obtain:
$$
||\nabla B_1||^2=\frac{\ell\left(\ell-1\right)}{4}\left(|\sum_{j=1}^nz_j^2|^2+2\sum_{j=1}^n|z_j|^4\right).
$$
Comparing the right and left hand sides of the above identity as polynomials in the variable $z_1, \dots, z_n$, we see that when $\ell\neq0$, $1$, the right hand side contains ${n+1 \choose 2}$ different monomials while the left has at most $n$. This implies $\ell=0$ or $\ell=1$, i.e. $\lambda=n+1$ and by Theorem \ref{hulinrational} $M$ is totally geodetic or $\lambda=n$ and $M$ is the quadric.

Assume now that $Q_1$ and $Q_2$ are not proportional. We will prove that this case is not possible. By Lemma \ref{coordQ}, we can choose a unitary frame $(\nu_1,\nu_2)$ of the normal space to $T _pM$ and a coordinate system $(z_1,\dots, z_n)$ around $p$ such that  $Q_1=\frac12\sum_{j=1}^n\alpha_jz_j^2$, $\alpha_j\geq 0$, $j=1,\dots, n$, and $Q_2=\frac12\sum_{j=1}^na_jz_j^2$, $a_j\in \mathds{C}$. From \eqref{quadratic} we get:
$$
\sum_{j=1}^n(\alpha_j^2+|a_j|^2-\ell)|z_j|^2=0,\quad \alpha_j \overline{\partial_j B_1}+a_j \overline{\partial_j B_2}=0, \ j=1,\dots, n.
$$
In particular, from the linear system in $\partial^2_{j,k} B_1$ and $\partial^2_{j,k} B_2$, obtained deriving the $j^{\rm th}$ identity $\alpha_j \overline{\partial_j B_1}+a_j \overline{\partial_j B_2}=0$ with respect to $\bar z_k$ and the $k^{\rm th}$ with respect to $\bar z_j$, for each $j$, $k=1,\dots, n$ we get $\partial^2_{j,k} B_1=\partial^2_{j,k} B_2=0$ whenever $\alpha_ja_k-\alpha_ka_j\neq 0$.
Observe that for $j=k$, $ |\nabla Q_j\wedge \nabla Q_k|^2=0$. Further:
 $$
 |Q_1|^2=\frac14|\sum_{j=1}^n\alpha_jz_j^2|^2=\frac14\sum_{j,k=1}^n\alpha_j\alpha_kz_j^2\bar z_k^2,\quad  |Q_2|^2=\frac14|\sum_{j=1}^na_jz_j^2|^2=\frac14\sum_{j,k=1}^na_j\bar a_kz_j^2\bar z_k^2,
 $$
 $$
||\nabla Q_1||^2=\sum_{j=1}^n\alpha_j^2|z_j|^2,\quad ||\nabla Q_2||^2=\sum_{j=1}^n|a_j|^2|z_j|^2,
 $$
 and for $j\neq k$:
 $$
|\langle \nabla Q_j,\nabla Q_k\rangle|^2=\sum_{j,k=1}^n\alpha_j\alpha_ka_j\bar a_k |z_j|^2|z_k|^2.
 $$
 Thus:
 $$
\sum_{j,k=1}^{2} |\nabla Q_j\wedge \nabla Q_k|^2= 2|\nabla Q_1\wedge \nabla Q_2|^2=2\sum_{j,k=1}^n\left(\alpha_j^2|a_k|^2-\alpha_j\alpha_k a_j\bar a_k\right)|z_jz_k|^2,
 $$
 and we have:
\begin{equation}\label{finaleqhulin}
\begin{split}
\sum_{j=1}^2||\nabla B_j||^2=&\frac14\left(\ell-1\right)\sum_{j\neq k}z^2_j\bar z^2_k\left(\alpha_j\alpha_k+a_j\bar a_k\right)+\frac{3\ell\left(\ell-1\right)}4\sum_{j=1}^n|z_j|^4+\\
&-\frac12\sum_{j,k=1}^n\left(\alpha_j^2|a_k|^2-\alpha_j\alpha_k a_j\bar a_k\right)|z_jz_k|^2.
 \end{split}
 \end{equation}
 Since $Q_1$ and $Q_2$ are not proportional, there exist $j$, $k$ such that $\alpha_ja_k-\alpha_ka_j\neq 0$. Observe that up to a unitary transformation of the normal space to $T_pM$ we can assume that such $\alpha_j$, $\alpha_k$, $a_j$, $a_k$ are not zero. For these fixed $j$, $k$ and for any $l=1,\dots, n$, $B_1$ and $B_2$ do not contain monomials in $z_jz_kz_l$. Thus, $||\nabla B_1||^2+||\nabla B_2||^2$ does not contain terms in $|z_jz_k|^2$. Comparing the left and right sides of \eqref{finaleqhulin}, we then get $\alpha_j^2|a_k|^2-\alpha_j\alpha_k a_j\bar a_k=0$, which leads to the desired contradiction $\alpha_ja_k-\alpha_k a_j=0$. 
\end{proof}



In general, it is an open problem to classify projectively induced K\"ahler--Einstein manifolds. The only known examples of such manifolds are homogeneous and it is conjecturally true these are the only ones (see e.g. \cite{note,ch,tak, ts}):
\begin{conj}\label{conjhom}
If a complete K\"ahler--Einstein manifold admits a K\"ahler immersion into $\CP^N$, then it is homogeneous.
\end{conj}

\begin{remark}\rm
When the ambient space is $\mathds{C}{\rm P}^\infty$, Conjecture \ref{conjhom} does not hold. Indeed in  the next chapter we describe a family of noncompact, nonhomogeneous and projectively induced K\"ahler--Einstein metrics. 
\end{remark}

Since  a homogeneous \K\ manifold which admits a \K\ immersion into  a complex projective space is compact (see \cite[\S 2 p. 178]{tak}),  we can state the following weaker conjecture  (cfr. Ex. \ref{excomplnonE}): 
\begin{conj}\label{compactconj}
If a complete K\"ahler--Einstein manifold admits a local K\"ahler immersion into $\CP^N$, then it is compact.
\end{conj}

\section{Exercises}
\begin{ExerciseList}
\Exercise
Let $(M, g)$ be  a complex $n$-dimensional \K\ manifold which admits a \K\ immersion into the finite dimensional complex projective space $(\CP^N, g_{FS})$.
Assume that  the  diastasis $\dd_0$ around some point $p\in M$ is {\em rotation invariant} with respect to the Bochner's coordinates $(z_1, \dots, z_n)$ around $p$
(this means that $\dd_0$ depends only on $|z_1|^2,\dots  ,|z_n|^2$).
Prove that  there exists an open neighbourhood   $W$ of $p$ such that $\dd_0(z)$ can be written on $W$ as:
\begin{equation}
\dd_0(z)=\log\left(1+\sum_{j=1}^n|z_j|^2+\sum_{j=n+1}^N a_j|z^{m_{h_j}}|^2\right)\nonumber
\end{equation}
where $a_j>0$ and $h_j\neq h_k$ for $j\neq k$.
\Exercise
Let $(M, g)$ be  as in the previous exercise.
Show that its  Einstein constant  is a positive rational number less or equal to $2(n+1)$. 
Deduce that  if $M^n$ is complete then $M^n$ is compact and simply connected.

({\em Hint:
The upper bound for $\lambda$ follows by Theorem \ref{hulinrational}. For the lower bound, use the previous exercise to write $\dd_0(z)=\log P$, where $P=1+\sum_{j=1}^n|z_j|^2+\sum_{j=n+1}^N a_j|z^{m_{h_j}}|^2$. From $\det(g_{\alpha\bar \beta})=\frac{1}{P^{2n}}\det\left(PP_{\alpha \bar \beta}-P_{\alpha}P_{\bar \beta}\right)$ one gets a inequality involving the total degree of $\det\left(PP_{\alpha \bar \beta}-P_{\alpha}P_{\bar \beta}\right)$ as a polynomial in the variables $z_1,\dots,z_n,\bar z_1,\dots,\bar z_n$, which combined with Eq. (\ref{mongebochner}) implies $\lambda>0$. The last part follows by Bonnet--Myers' Theorem and by a result of  Kobayashi \cite{koricci} which asserts that a compact manifold with positive first Chern class  is simply-connected.})
%
%

\Exercise
Let $(M, g)$ be  a complex $n$-dimensional \K\ manifold which admits a \K\ immersion into the finite dimensional complex projective space $(\CP^N, g_{FS})$.
Assume that  the  diastasis $\dd_0$ around some point $p\in M$ is {\em radial} with respect to the Bochner's coordinates $z_1, \dots, z_n$ around $p$
(this means that $\dd_0$ depends only on $|z_1|^2+\cdots +|z_n|^2$).
Prove that  there exists an open neighbourhood   $W$ of $p$ such that $\dd_0(z)$ can be written on $W$ as
$$
D_0(z)=\log\left(1+\sum_{j=1}^n|z_j|^2+\sum_{k=2}^Na_k \left(\sum_{j=1}^n|z_j|^2\right)^{k}\right),
$$
where $a_k>0$ for each $k=2,\dots, N$.
\Exercise
Let $(M,g)$ be as in the previous exercise.  Prove that $M$ is an open subset of $\mathds{C}{\rm P}^{n}$.

({\em Hint: Write the Monge--Amp\'ere Eq. \eqref{mongebochner} in terms of the polynomial $P=1+\sum_{j=1}^n|z_j|^2+\sum_{k=2}^Na_k \left(\sum_{j=1}^n|z_j|^2\right)^{k}$.})
 
 \Exercise
 Give an example of complete \K\ manifold which can be \K\ immersed into the finite complex projective space $(\CP^N, g_{FS})$.
 \label{excomplnonE}

\Exercise
Show that a compact simply-connected  \K\--Einstein manifold with nonpositive Einstein constant cannot be locally  \K\ immersed
into any complex space form. Show with an example that the assumption of simply-connectedness cannot be dropped.
\end{ExerciseList}

\chapter{Hartogs type domains}\label{hartogstype}
Hartogs type domains are a class of domains of $\mathds{C}^{n+m}$ characterized by a K\"ahler metric described locally by a K\"ahler potential of the form $\Phi(z,w)=H(z)-\log\left(F(z)-|w|^2\right)$, for suitable functions $H$ and $F$. They have been studied under several points of view and represent a large class of examples in K\"ahler geometry (the reader finds precise references inside each section). 

The first section describes Cartan--Hartogs domains. Prop. \ref{lemmadiastM} discusses the existence of a K\"ahler immersion into the infinite dimensional complex projective space in terms of the Cartan domains they are based on, and Th. \ref{thwallach} proves they represent a counterexample for Conjecture \ref{conjhom} when the ambient space is infinite dimensional.
Section \ref{bhdomains} extends some of these results when the base domain is not symmetric but just a bounded homogeneous domain.

Finally, in Section \ref{rotinvhart} we discuss the existence of a K\"ahler immersion for a large class of Hartogs domains whose K\"ahler potentials are given locally by $-\log\left(F(|z_0|^2)-||z||^2\right)$ for suitable function $F$ (see Prop. \ref{Kmet}).

\section{Cartan--Hartogs domains}

Let $\Omega$ be an irreducible bounded symmetric domain of complex dimension $d$ and genus $\gamma$. For all positive real numbers $\mu$ consider the family of Cartan-Hartogs domains:
\begin{equation}
\M_{\Omega}(\mu)=\left\{(z,w)\in \Omega\times\C,\ |w|^2<\N_\Omega^\mu(z,z)\right\},
\end{equation}
where $\N_\Omega(z,z)$ is the  {\em generic norm} of $\Omega$, i.e.:
$$\N_{\Omega}(z, z)=(V(\Omega)\Kj(z, z))^{-\frac{1}{\gamma}},$$
with $V(\Omega)$ the total volume of $\Omega$ with respect to the Euclidean measure of the ambient complex Euclidean space and $\Kj(z, z)$ is its Bergman kernel.

The domain $\Omega$ is called  the {\em  base} of the Cartan--Hartogs domain 
$\M_{\Omega}(\mu)$ (one also  says that 
$\M_{\Omega}(\mu)$  is based on $\Omega$).
Consider on $\M_{\Omega}(\mu)$ the metric $g(\mu)$  whose globally
defined  K\"ahler potential around the origin is given by
\begin{equation}\label{diastM0}
\dd_0(z,w)=-\log(\N_{\Omega}^\mu(z,z)-|w|^2).
\end{equation}
Cartan--Hartogs domains has been considered by many authors (see e.g. \cite{fengtu, fengtubalanced, indefinite,articwall,balancedch,roos, compl, zedda,berezinCH,coeff,chrel}) under different points of view. Their importance relies on being examples of nonhomogeneous domains which for a particular value of the parameter $\mu$ are K\"ahler--Einstein. The following theorem summarizes these properties.
(see  \cite{roos} and \cite{compl} for a proof).
\begin{theor}[G. Roos, A. Wang, W. Yin, L. Zhang, W. Zhang, \cite{roos}]\label{roos}
Let $\mu_0=\gamma/(d+1)$. Then $\left(\M_{\Omega}(\mu_0),g(\mu_0)\right)$ is a complete K\"ahler--Einstein manifold which is homogeneous  if and only if the rank of  $\Omega$ equals $1$,
i.e.  $\Omega = \CH^d$.
\end{theor}
\begin{remark}\rm\label{rrchimm}
Observe that when $\Omega=\CH^d$, we have $\mu_0=1$, $\M_\Omega(1)=\CH^{d+1}$ and $g(1)=g_{hyp}$.
\end{remark}

The following proposition shows that the existence of a K\"ahler immersion of a Cartan--Hartogs domain into $\mathds{C}{\rm P}^{\infty}$ is completely determined by the base domain $(\Omega, g_B)$,
where $g_B$ is its Bergman metric.
\begin{prop}[A. Loi, M. Zedda, \cite{articwall}]\label{lemmadiastM}
The potential $\dd_0(z,w)$ given by (\ref{diastM0}) is the diastasis
around the origin  of the metric
$g(\mu)$.
Moreover,  $cg(\mu)$ is projectively induced if and only if  $(c+m)\frac{\mu}{\gamma}g_B$ is projectively induced for every integer $m\geq 0$.
\end{prop}
\begin{proof}
The power expansion around the origin of $\dd_0(z,w)$ can be written as:
\begin{equation}\label{powexpnot}
\dd_0(z,w)=\sum_{j,k=0}^\infty A_{jk} (zw)^{m_j}(\bar z\bar w)^{m_k},
\end{equation}
where $m_j$ are ordered $(d+1)$-uples of integer and:
$$
(zw)^{m_j}=z_1^{m_{j,1}}\cdots z_d^{m_{j,d}} w^{m_{j,d+1}}.
$$
In order to prove that   $\dd_0(z,w)$ is the diastasis for $g(\mu)$
we need to verify  that $A_{j0}=A_{0j}=0$ (see Theorem \ref{chardiast}).
This is straightforward. Indeed if we take derivatives with respect 
either to $z$ or $\bar z$  is the same as deriving the function $-\log(\N_\Omega^\mu(z,z))=\frac{\mu}{\gamma}\dd_0^\Omega(z)$ that is the diastasis of $(\Omega,\frac{\mu}{\gamma}g_B)$, thus we obtain $0$. If we take derivatives with respect either to $w$ or  $\bar w$  we obtain zero no matter how many times we derive with respect to $z$ or $\bar z$, since $\dd_0(z,w)$ is radial in $w$.

In order to prove the second part of the proposition
take the function:
\begin{equation}\label{funzione}
e^{c\dd_0(z,w)}-1=\frac{1}{(\N_\Omega^\mu(z, z)-|w|^2)^c}-1,
\end{equation}
and using the same notations as in (\ref{powexpnot}) write the power expansion around the origin as:
\begin{equation}\nonumber
e^{c\dd_0(z,w)}-1=\sum_{j,k=0}^\infty B_{jk} (zw)^{m_j}(\bar z\bar w)^{m_k}.
\end{equation}
By Calabi's criterion (Theorem \ref{localcrit}), $cg(\mu)$
is projectively induced  if and only if $B=(B_{jk})$ is positive semidefinite of infinite rank.
The generic entry of $B$ is given by:
\begin{equation}
B_{jk}=\frac{1}{m_j!\cdot m_k!}\frac{\de^{|m_j|+|m_k|}}{\de (zw)^{m_j}\de (\bar z\bar w)^{m_k}}\left(\frac{1}{(\N_\Omega^\mu(z, z)-|w|^2)^c}-1\right)\Bigg|_0,\nonumber
\end{equation}
where $m_j! =m_{j,1}!\cdots m_{j,d+1}!$ and  $\de(zw)^{m_j}=\de z_1^{m_{j,1}}\cdots\de z_d^{m_{j,d}} \de w^{m_{j,d+1}}$.
By Proposition \ref{diastdom} we have:
\begin{equation}\label{condi}
m_{j,1}+\cdots +m_{j,d}\neq m_{k,1}+\cdots +m_{k,d} \Rightarrow B_{jk}=0,
\end{equation}
and since (\ref{funzione}) is radial in $w$ we also have:
\begin{equation}\label{condii}
m_{j,d+1}\neq m_{k,d+1} \Rightarrow B_{jk}=0.
\end{equation}

Thus, $B$ is a $\infty\times\infty$ matrix of the form
\begin{equation}
B=\left(
\begin{array}{cccccc}
0&0&0&0&0&0\\
0&E_1&0&0&0&\dots\\
0&0&E_2&0&0&\dots\\
0&\vdots&0&E_3&0&\dots\\
0&&\vdots&0&\ddots&
\end{array}
\right),\nonumber
\end{equation}
where the generic block $E_i$ contains derivatives
$\de(zw)^{m_j}$$\de(\bar z\bar w)^{m_k}$ of   order $2i$, $i=1,2,\dots$ such that $|m_j|=|m_k|=i$.
We can further write:
\begin{equation}\label{matrixz}
E_i=\left(
\begin{array}{ccc}
F_{z(i)}(0)&0&0\\
0&F_{w(i)}(0)&0\\
0&0&F_{(z,w)(i)}(0)
\end{array}
\right),
\end{equation}
where $F_{z(i)}(0)$ (resp. $F_{w(i)}(0)$, $F_{(z,w)(i)}(0)$) contains derivatives $\de(zw)^{m_j}$
 $\de(\bar z\bar w)^{m_k}$ (of   order $2i$  with $|m_j|=|m_k|=i$) such that $m_{j, d+1}=m_{k, d+1}=0$ (resp. $m_{j, d+1}=m_{k, d+1}=i$, $m_{j, d+1}, m_{k, d+1}\neq 0, i$).
(Notice also  that we have $0$ in all the other entries because of (\ref{condi}) and (\ref{condii})).
Since the derivatives are evaluated at the origin, deriving (\ref{funzione}) with respect to
$\de(zw)^{m_j}$ $\de(\bar z\bar w)^{m_k}$ with $|m_j|=|m_k|=i$ and $m_{j, d+1}=m_{k, d+1}=0$
 is the same as deriving the function:
\begin{equation}\label{funzdominio}
\frac{1}{(\N_\Omega^\mu(z,z))^c}-1=e^{c\frac{\mu}{\gamma}\dd_0^\Omega(z)}-1.
\end{equation}
Thus,  by Calabi's criterion, all the blocks $F_{z(i)}(0)$ are positive semidefinite if and only if $c\frac{\mu}{\gamma} g_B$ is projectively induced.
Observe that  the blocks $F_{w(i)}(0)$ are semipositive definite without extras assumptions.  Indeed 
 if we consider derivatives $\de(zw)^{m_j}$$\de(\bar z\bar w)^{m_k}$ of (\ref{funzione})
with 
$|m_j|=|m_k|=i$ and $m_{j, d+1}=m_{k, d+1}=i$, since $\N_\Omega^\mu(z,z)$ evaluated in $0$ is equal to $1$, it is the same as deriving the function $1/(1-|w|^2)^c-1=\left(\sum_{j=0}^\infty |w|^{2j}\right)^c-1$
and the claim follows.
Finally,  consider  the block $F_{(z,w)(i)}(0)$. It can be written as:
\begin{equation}
F_{(z,w)(i)}(0)=\left(
\begin{array}{cccc}
H_{z(i-1),w(1)}(0)&0&0&0\\
0&H_{z(i-2),w(2)}(0)&0&0\\
\vdots&&\ddots&\\
0&0&0&H_{z(1),w(i-1)}(0)
\end{array}
\right),\nonumber
\end{equation}
where the generic block $H_{z(i-m),w(m)}(0)$, $1\leq m\leq i-1$, 
contains derivatives
$\de(zw)^{m_j}$\linebreak $\de(\bar z\bar w)^{m_k}$ of   order $2i$  such that $|m_j|=|m_k|=i$ and $m_{j, d+1}=m_{k, d+1}=m$ evaluated at zero
(as before, by (\ref{condi}) and (\ref{condii}) all entries outside these blocks are $0$).
Now it is not hard to verify   that  these blocks can be obtained by taking derivatives
$\de(zw)^{m_j}$$\de(\bar z\bar w)^{m_k}$ of order  $2(i-m)$ such that $|m_j|=|m_k|=2(i-m)$ and $m_{j, d+1}=m_{k, d+1}=0$  of  the function
\begin{equation}\label{formulareq1}
\frac{(m+c-1)!}{(c-1)!\; m!\ \N^{\mu(c+m)}_\Omega(z, z)}-1=
e^{(c+m)\frac{\mu}{\gamma}\dd_0^\Omega(z)}-1,
\end{equation}
and evaluating at $z=\bar z =0$.
Thus, again by Calabi's criterion,
$F_{(z,w)(i)}(0)$
is positive semidefinite
iff $(c+m)\frac{\mu}{\gamma}g_B$, $m\geq 1$,
is projectively induced
and this ends the proof of the proposition.
\end{proof}
\begin{remark}{\rm
Proposition \ref{lemmadiastM} can be also  proved for ``general'' Cartan-Hartogs domains with dimension $n=d+r$, namely
\begin{equation}
\M_{\Omega}(\mu)=\left\{(z,w)\in \Omega\times\C^r,\ ||w||^2<\N_\Omega^\mu(z,z)\right\},\nonumber
\end{equation}
where $||w||^2=|w_1|^2+\dots+|w_r|^2$. In that case Equation (\ref{formulareq1}) can be obtained using the following formula
\begin{align}\label{formula}
&\frac{1}{m_1!^2\cdots m_r!^2}\frac{\de^{2m}}{\de w_1^{m_1}\de \bar w_1^{m_1}\cdots\de
w_r^{m_r}\de \bar
w_r^{m_r}}\left(\frac{1}{f(z,\bar z)-||w||^2}\right)^c=\nonumber\\
=&\frac{1}{m_1!^2\cdots m_r!^2}\sum_{k_1=1}^{m_1+1}\cdots\sum_{k_r=1}^{m_r+1}\left[\frac{(\sum_{j=1}^r(k_j)+m+c-r-1)!}{(c-1)!}\right.\cdot\nonumber\\
&\cdot\left.\prod_{i=1}^r\left[{m_i\choose k_i-1}^2(m_i+1-k_i)!(w_i\bar
w_i)^{k_i-1}\right]\frac{1}{\left(f(z,\bar z)-||w||^2\right)^{\sum_{j=1}^r(k_j)+m+c-r}}\right].\nonumber
\end{align}
}
\end{remark}

\medskip

From Theorem \ref{roos}, Prop. \ref{lemmadiastM} and Theorem \ref{wallach} we get the following theorem, which gives a counterexample to Conjecture \ref{conjhom} in the case when the ambient space is infinite dimensional.

 \begin{theor}[A. Loi, M. Zedda, \cite{articwall}]\label{thwallach}
There exists a continuous family of homothetic, complete,   nonhomogeneous and  projectively induced K\"ahler-Einstein metrics on each Cartan--Hartogs domain based on an irreducible bounded symmetric domain of rank $r\neq 1$.
\end{theor}
\begin{proof}
Take  $\mu=\mu_0=\gamma/(d+1)$ in (\ref{diastM0})  and $\Omega\neq\CH^d$. By Theorem \ref{roos} $\left(\M_\Omega(\mu_0),cg(\mu_0)\right)$ is K\"ahler-Einstein, complete and nonhomogeneous for all positive real numbers $c$.
By Proposition \ref{lemmadiastM}
 $cg(\mu_0)$  
 is projectively induced
if and only if $\frac{c+m}{d+1}g_B$ is projectively induced, for all nonnegative integer $m$.
By Theorem \ref{wallach} this  happens if  $\frac{(c+m)}{d+1}\geq 
\frac{(r-1)a}{2\gamma}$.
Hence  $cg (\mu_0)$ with $c\geq \frac{(r-1)(d+1)a}{2\gamma}$ is the desired family of
projectively induced  K\"ahler-Einstein metrics.
\end{proof}


By applying the same argument with $0<c< \frac{a(d+1)}{2\gamma}$ (and $r\neq 1$) one also gets the following:
\begin{cor}\label{corw2}
There exists a continuous  family of nonhomogeneous, complete, 
 K\"ahler-Einstein metrics which  does not admit a local K\"ahler immersion into $\CP^N$ for any $N\leq \infty$.
\end{cor}

\begin{remark}{\rm
As direct consequence of Corollary \ref{corw2} together with Exercise \ref{constantell2c}, we get that a Cartan-Hartogs domain $\left(\M_\Omega(\mu_0),cg(\mu_0)\right)$ does not admit a K\"ahler immersion into $l^2(\C)$. Further by Theorem \ref{chcn}, it does not admit a K\"ahler immersion into $\CH^\infty$ for any value of $c>0$ either.
}
\end{remark}

We conclude this section with the following lemma which gives an explicit expression of the K\"ahler map of a Cartan-Hartogs domain into $\mathds{C}{\rm P}^\infty$.

\begin{lem}[A. Loi, M. Zedda, \cite{balancedch}]\label{immersion}
If $f\!:M_\Omega(\mu)\f \CP^\infty$ is a holomorphic map such that $f^*\omega_{FS}=\alpha\,\omega(\mu)$ then up to unitary transformation of $\CP^\infty$
it is given by:
\begin{equation}\label{immf0}
f=\left[ 1, s, h_{\frac{\mu\, \alpha}{\gamma}},\dots,\sqrt{\frac{(m+ \alpha-1)!}{(\alpha-1)!m!}}h_{\frac{\mu(\alpha +m)}{\gamma}}w^m,\dots\right],
\end{equation}
where $s=(s_1,\dots, s_m,\dots)$ with: 
$$s_m=\sqrt{\frac{(m+ \alpha-1)!}{(\alpha-1)!m!}}w^m,$$
and $h_k=(h_k^1,\dots,h_k^j,\dots)$ denotes the sequence of holomorphic maps on $\Omega$ such that the immersion $\tilde h_k=(1,h_k^1,\dots, h_k^j,\dots)$, $\tilde h_k\!:\Omega\f\CP^\infty$, satisfies $\tilde h_k^*\omega_{FS}=k \omega_B$, i.e.:
\begin{equation}\label{ie}
1+\sum_{j=1}^{\infty}|h_k^j|^2=\frac{1}{N^{\gamma\, k}}.
\end{equation} 
\end{lem}
\begin{proof}
Since the immersion is isometric, by (\ref{diastM0}) we have $f^*\Phi_{FS}=-\alpha\log(N_{\Omega}^\mu(z,z)-|w|^2)$, which is equivalent to:
$$\frac{1}{(N^{\mu}-|w|^2)^\alpha}=\sum_{j=0}^\infty |f_j|^2,$$
for $f=[f_0,\dots, f_j,\dots]$.
If we consider the power expansion around the origin of the left hand side with respect to $w$, $\bar w$, we get:
\begin{equation}
\begin{split}
\sum_{k=1}^\infty \left[\frac{\de^{2k}}{\de w^k \de \bar w^k}\frac{1}{(N^{\mu}-|w|^2)^\alpha}\right]_{0}\frac{|w|^{2k}}{k!^2}=& \sum_{k=1}^\infty \left[\frac{\de^{2k}}{\de w^k \de \bar w^k}\frac{1}{(1-|w|^2)^\alpha}\right]_{0}\frac{|w|^{2k}}{k!^2}\\
=&\frac1{(1-|w|^2)^\alpha}-1.\nonumber
\end{split}
\end{equation}
The power expansion with respect to $z$ and $\bar z$ reads:
\begin{equation}
\begin{split}
\sum_{j,k}\left[\frac{\de^{|m_j|+|m_k|}}{\de z^{m_j} \de \bar z^{m_k}}\frac{1}{(N^{\mu}-|w|^2)^\alpha}\right]_{0}\frac{z^{m_j}\bar z^{m_k}}{m_j!m_k!}=&\sum_{j,k}\left[\frac{\de^{|m_j|+|m_k|}}{\de z^{m_j} \de \bar z^{m_k}}\frac{1}{N^{\mu\alpha}}\right]_{0}\frac{z^{m_j}\bar z^{m_k}}{m_j!m_k!}\\
=&\sum_{j=1}^\infty |h_{\frac{\mu\alpha}{\gamma}}^j|^2,\nonumber
\end{split}
\end{equation}
where the last equality holds since by (\ref{ie}) $\sum_{j=1}^\infty |h_{\frac{\mu\alpha}{\gamma}}^j|^2$ is the power expansion of $\frac{1}{N^{\mu\alpha}}-1$. 

Finally, the power expansion with respect to $z$, $\bar z$, $w$, $\bar w$ reads:
\begin{equation}
\begin{split}
&\sum_{m=1}^\infty\sum_{j,k}\left[\frac{\de^{|m_j|+|m_k|}}{\de z^{m_j} \de \bar z^{m_k}}\frac{\de^{2m}}{\de w^m \de \bar w^m}\frac{1}{(N^{\mu}-|w|^2)^\alpha}\right]_{0}\frac{z^{m_j}\bar z^{m_k}w^m\bar w^m}{m_j!m_k!m!^2}\\
=&\sum_{m=1}^\infty\sum_{j,k}\left[\frac{\de^{|m_j|+|m_k|}}{\de z^{m_j} \de \bar z^{m_k}}\frac{(m+\alpha-1)!}{(\alpha-1)!m!N^{\mu(\alpha+m)}}\right]_{0}\frac{z^{m_j}\bar z^{m_k}}{m_j!m_k!}|w|^{2m}\\
=&\sum_{m=1}^\infty\sum_{j=1}^\infty \frac{(m+\alpha-1)!}{(\alpha-1)!m!}|w|^{2m}|h_{\frac{\mu(\alpha+m)}{\gamma}}^j|^2,\nonumber
\end{split}
\end{equation}
where we are using (\ref{ie}) again. It follows by the previous power series expansions, that the map $f$ given by (\ref{immf0}) is a \K\ immersion  of $(M_{\Omega}(\mu), \alpha g(\mu))$ into $\CP^\infty$. By Calabi's rigidity Theorem \ref{local rigidityb} all other \K\ immersions are given by $U\circ f$, where $U$ is a unitary transformation of $\C P^{\infty}$.
 \end{proof}

\section{Bergman--Hartogs domains}\label{bhdomains}

Bergman--Hartogs domains are a generalization of Cartan-Hartogs domains where the base domain is not required to be symmetric but just homogeneous and endowed with its Bergman metric. To the authors knowledge, they have already been considered in \cite{hao2,hao3}.

For all positive real numbers $\mu$ a {\em Bergman-Hartogs domain} is defined by:
$$M_{\Omega}(\mu)=\left\{(z,w)\in \Omega\times\mathds{C},\ |w|^2<\tilde \Kj(z, z)^{-\mu}\right\},$$
where $\tilde \Kj(z, z)=\frac{\Kj(z,z)\Kj(0,0)}{|\Kj(z,0)|^2}$ with $\Kj$ the Bergman kernel of $\Omega$.
Consider on $M_{\Omega}(\mu)$ the metric $g(\mu)$  whose associated K\"ahler form $\omega(\mu)$ can be described by the (globally defined) K\"ahler potential centered at the origin:
$$
\Phi(z,w)=-\log(\tilde\Kj(z, z)^{-\mu}-|w|^2).
$$
The domain $\Omega$ is called  the {\em  base} of the Bergman--Hartogs domain 
$M_{\Omega}(\mu)$ (one also  says that 
$M_{\Omega}(\mu)$  is based on $\Omega$). 

In the previous section it is proven that when the base domain is symmetric $(M_{\Omega}(\mu),c\,g(\mu))$ admits a K\"ahler immersion into the infinite dimensional complex projective space if and only if $(\Omega, (c+m)\mu g_B)$ does for every integer $m\geq0$. As pointed out in \cite{hao}, a totally similar proof holds also when the base is a homogeneous bounded domain. This fact together with Theorem \ref{loimossaimm} proves that a Bergman--Hartogs domain $(M_{\Omega}(\mu),c\,g(\mu))$ is projectively induced for all large enough values of the constant $c$ multiplying the metric. Further, the immersion can be written explicitely as follows (cfr. Lemma \ref{immersion} in the previous section):
\begin{lem}\label{chimm}
Let $\alpha$ be a positive real number such that the Bergman--Hartogs domain $(M_{\Omega}(\mu),\alpha\, g(\mu))$ is projectively induced. Then, the K\"ahler map $f$ from $(M_{\Omega}(\mu),\alpha\,g(\mu))$ into $\mathds{C}{\rm P}^\infty$, up to unitary transformation of $\mathds{C}{\rm P}^\infty$, is given by:
\begin{equation}\label{immf}
f=\left[ 1, s, h_{\mu\, \alpha},\dots,\sqrt{\frac{(m+ \alpha-1)!}{(\alpha-1)!m!}}h_{\mu(\alpha +m)}w^m,\dots\right],
\end{equation}
where $s=(s_1,\dots, s_m,\dots)$ with
$$s_m=\sqrt{\frac{(m+ \alpha-1)!}{(\alpha-1)!m!}}w^m,$$
and $h_k=(h_k^1,\dots,h_k^j,\dots)$ denotes the sequence of holomorphic maps on $\Omega$ such that the immersion $\tilde h_k=(1,h_k^1,\dots, h_k^j,\dots)$, $\tilde h_k\!:\Omega\rightarrow\mathds{C}{\rm P}^\infty$, satisfies $\tilde h_k^*\omega_{FS}=k \omega_B$, i.e. 
\begin{equation}
1+\sum_{j=1}^{\infty}|h_k^j|^2=\tilde \Kj^{-k}.\nonumber
\end{equation} 
\end{lem}
\begin{proof}
The proof follows essentially that of \cite[Lemma 8]{balancedch} once considered that $\Phi(z,w)=-\log(\tilde \Kj(z, z)^{-\mu}-|w|^2)$ is the diastasis function for $(M_\Omega(\mu), g(\mu))$ as follows readily applying the definition of diastasis \eqref{diastdefinition}.
\end{proof}

Observe that such map is full, as can be easily seen for example by considering that for any $m=1,2,3,\dots,$ the subsequence $\{s_1,\dots, s_m\}$ is composed by linearly independent functions.

\section{Rotation invariant Hartogs domains}\label{rotinvhart}
The class of domains we are about to describe is a very rich class of examples. It has been considered in \cite{englis} in the context of Berezin quantization and in \cite{hartogsloi} in relation to the existence of a K\"ahler immersion into finite dimensional complex space forms (see also \cite{hartogsint, hartogs, hartogsijg, hartogsosaka} for other results on their Riemannian and K\"ahler geometry). 
 
Let $x_0 \in \R^+ \cup \{ + \infty \}$ and let $F: [0, x_0)
\rightarrow (0, + \infty)$ be a decreasing continuous function,
smooth on $(0, x_0)$. The Hartogs domain $D_F\subset
{\C}^{n}$ associated to the function $F$ is defined by:
$$D_F = \{ (z_0, z_1,\dots ,z_{n-1}) \in {\C}^{n} \; | \; |z_0|^2 < x_0, \ ||z||^2  < F(|z_0|^2)
\}, $$
where $||z||^2=|z_1|^2+\dots + |z_{n-1}|^2$.
We shall assume that the natural $(1, 1)$-form on $D_F$  given by:
\begin{equation}\label{omegaf}
\omega_F = \frac{i}{2} \partial \overline{\partial}
\log \left(\frac{1}{F(|z_0|^2) - ||z||^2 }\right),
\end{equation}
 is  a \K\ form on $D_F$. The following proposition gives some
conditions on $D_F$ equivalent to this assumption:
\begin{prop}[A. Loi, F. Zuddas \cite{hartogsosaka}]\label{Kmet}
Let $D_F$ be a Hartogs domain in ${\C}^{n}$. Then the
following conditions are equivalent:
\begin{itemize}
\item [(i)] the $(1, 1)$-form $\omega_F$  given by (\ref{omegaf})
is a \K\ form; \item [(ii)] the function $- \frac{x F'(x)}{F(x)}$
is strictly increasing, namely $-\left( \frac{x F'(x)}{F(x)}\right)' >0$
for every $x \in [0, x_0)$; \item [(iii)] the boundary of $D_F$ is
strongly pseudoconvex at all $z = (z_0, z_1,\dots,z_{n-1})$ with
$|z_0|^2 < x_0$.
\end{itemize}
\end{prop}
The \K\ metric $g_F$ associated to the \K\ form $\omega_F$ is the metric we will be dealing with in the present paper. It follows by (\ref{omegaf}) that a K\"ahler potential for this metric is given by:
$$\Phi_F=-\log \left(F(|z_0|^2) - ||z||^2\right).$$
Observe that this is also the diastasis function around the origin for $\omega_F$.
\begin{remark}\rm
It is worth pointing out that an Hartogs domain $(D_F, g_F)$ is either  homogeneous or  Einstein
if and only if $F(x)=1-x$, namely $D_F$ is the complex hyperbolic space equipped with the hyperbolic metric
(see Theorem 1.1 in \cite{hartogsosaka} for a proof).
\end{remark}

 In order to study the existence of a K\"ahler immersion into the complex projective space, we start considering that setting $|z_0|^2=x$ and $|z_j|^2=y_j$, $j=1,\dots, n-1$, we get:
\begin{equation}\label{derf}
\begin{split}
\frac{\partial^{2j}}{\partial z_0^j\partial \bar z_0^j}\frac{\partial^{2k_1}}{\partial z_1^{k_1}\partial \bar z_1^{k_1}}\cdots &\frac{\partial^{2k_{n-1}}}{\partial z_{n-1}^{k_{n-1}}\partial \bar z_{n-1}^{k_{n-1}}}\frac{1}{\left(F(|z_0|^2) - ||z||^2\right)^c}|_0\\
&=j!k_1!\cdots k_{n-1}!\frac{\partial^{j}}{\partial x^j}\frac{\partial^{k_1}}{\partial y_1^{k_1}}\cdots \frac{\partial^{k_{n-1}}}{ \partial y_{n-1}^{k_{n-1}}}\frac{1}{\left(F(x) - ||y||^2\right)^c}|_0\\
&=j!k_1!\cdots k_{n-1}!\frac{\Gamma(c+k_1+\dots+k_{n-1})}{\Gamma (c)}\frac{\partial^{j}}{\partial x^j}\frac{1}{\left(F(x)\right)^{c+k_1+\dots+k_{n-1}}}|_0.
\end{split}
\end{equation}
From Calabi's criterion Theorem \ref{critfb} it follows that a Hartogs domain $(D_F,c\,\omega_F)$ is projectively induced if and only if:
\begin{equation}\label{condhartimm}
\frac{\partial^{j}}{\partial x^j}\frac{1}{\left(F(x)\right)^{c+k}}|_0\geq 0,
\end{equation}
for all integers $j$, $k\geq 0$. This condition is of course strictly related to $F$. The following example, Prop. \ref{rhp} and exercises \ref{springer}, \ref{dnotsufficient} and \ref{notb}, show that there are cases when the immersion exists for all values of $c$, or only for integers values of $c$, or for no value. Observe that since such domains are rotation invariant, when the immersion exists it can be written as:
$$
f\!:D_F\rightarrow \mathds{C}{\rm P}^\infty,\quad f(z)=[\dots,f_{j,k_1,\dots, k_{n-1}},\dots]
$$
where:
$$
f_{j,k_1,\dots, k_{n-1}}= \sqrt{\frac{\Gamma(c+k_1+\dots+k_{n-1})}{j!k_1!\cdots k_{n-1}!\,\Gamma (c)}\frac{\partial^{j}}{\partial x^j}\frac{1}{\left(F(x)\right)^{c+k_1+\dots+k_{n-1}}}|_0}z_0^{j}z_1^{k_1}\cdots z_{n-1}^{k_{n-1}}.
$$
\begin{ex}\label{tp}\rm
Let $F(t)=(1-t)^p$, $p>0$, $x_0=1$ (for $p=1$ we recover $\CH^n$ described in Section \ref{csf}). The Hartogs domain associated to $F$ is given by:
$$D_F=\left\{(z_0,\dots,z_{n-1})\in\C^n\ | \ |z_0|^2+(||z||^2)^{1/p}<1\right\}.$$
Since:
$$
\frac{\partial^{j}}{\partial x^j}\frac{1}{\left(1-x\right)^{p(c+k)}}|_0=\frac{\Gamma(p(c+k)+j)}{\Gamma(p(c+k))},
$$
this domain is infinite projectively induced for any value of $c>0$ and thus (see Ex. \ref{constantell2c}) it also admits a full K\"ahler immersion into $l^2(\mathds{C})$.
%
\end{ex}


The next proposition gives an example of Hartogs domain not admitting a K\"ahler immersion into $\mathds{C}{\rm P}^\infty$ even when the metric is rescaled by a positive constant. 
\begin{prop}\label{rhp}
The Hartogs domain $(D_F,g_F)$ defined by
$$F(x)=\left(x-1\right)\left(x-\frac{11}4\right)\left(x+\frac34\right),$$
and with $x_0=1$ does not admit a K\"ahler immersion into $\mathds{C}{\rm P}^\infty$ even when the metric is rescaled by a positive constant $c$.
\end{prop}
\begin{proof}
By definition $F(x)>0$ in $[0,1)$. Further:
$$
F'(x)=3x(x-2)-\frac1{16}\leq -\frac{1}{16}<0.
$$
This domain does not admit a K\"ahler immersion into $\mathds{C}{\rm P}^\infty$ for any value of $c$. In order to prove it, consider that for any polynomial $P(x)=(x-a_1)\cdots(x-a_n)$, one has:
$$
\frac{\partial^{j}}{\partial x^j}P(x)^{-c}|_{0}=\frac{(-1)^jj!}{\Gamma(c)^n}\sum_{k_1+\dots+k_n=j}\frac{\Gamma(c+k_1)\cdots\Gamma(c+k_n)}{k_1!\cdots k_n!}\frac{1}{(-a_1)^{c+k_1}\cdots(-a_n)^{c+k_n}}.
$$
For $n=3$ and $a_1=1$, $a_2=\frac{11}4$, $a_3=-\frac34$, we get:
\begin{equation}\label{akappa}
\frac{\partial^{j}}{\partial x^j}P(x)^{-c}|_{0}=\frac{(-1)^jj!}{\Gamma(c)^3}\sum_{k_1+k_2+k_3=j}(-1)^{k_1+k_2}A(k_1,k_2,k_3),
\end{equation}
where we set:
\begin{equation}\label{ak}
A(k_1,k_2,k_3)=\frac{\Gamma(c+k_1)\Gamma(c+k_2)\Gamma(c+k_3)}{k_1!k_2!k_3!}\left(\frac{4}{11}\right)^{c+k_2}\left(\frac43\right)^{c+k_3}.
\end{equation}
When $j$ is odd and greater than $1$, the sign of $(-1)^{k_1+k_2}A(k_1,k_2,k_3)$ is positive for $k_3$ odd, and negative for $k_3$ even. Thus, in order to prove that for any $c>0$ there exists $j$ such that \eqref{akappa} is negative, it is enough to show that for all $h=1,\dots, j$:
$$
\sum_{k_1+k_2=j-h}A(k_1,k_2,h)>\sum_{k_1+k_2=j-h-1}A(k_1,k_2,h-1).
$$
By \eqref{ak}, this is equivalent to the following quantity being positive:
\begin{equation}\label{akk}
\begin{split}
\sum_{k=0}^{j-h}&\frac{\Gamma(c+k)\Gamma(c+j-h-k)\Gamma(c+h-1)}{k!(j-h-k)!(h-1)!}\left(\frac4{11}\right)^{c+k}\left(\frac43\right)^{c+h-1}\left[\frac43\frac{c+h-1}{h}-\frac{c+j-h-k}{j-h-k+1}\right]+\\
&-\frac{\Gamma(c+j-h+1)\Gamma(c)\Gamma(c+h-1)}{(j-h+1)!(h-1)!}\left(\frac4{11}\right)^{c+j-h+1}\left(\frac43\right)^{c+h-1}.
\end{split}
\end{equation}
We claim that for any fixed $h$, \eqref{akk} is positive as $j\rightarrow +\infty$. In order to prove the claim we observe that the second part of \eqref{akk} goes to zero as $j$ grows, while the first part does as $k$ grows. Thus, the sign of \eqref{akk} as $j$ goes to infinity is determined by finite values of $k$, i.e. those values that do not approach $j$. The claim then follows since for $k$ and $h$ fixed:
$$
\lim_{j\rightarrow +\infty}\left(\frac43\frac{c+h-1}{h}-\frac{c+j-h-k}{j-h-k+1}\right)=\frac43\frac{c+h-1}{h}-1>0.
$$
\end{proof}

The property of being not projectively induced even when rescaled is a not trivial property which we discuss in more details in Section \ref{cigarsection}.

\section{Exercises}

\begin{ExerciseList}
\Exercise\label{springer} Consider the Springer domain, i.e. the rotation invariant Hartogs domain defined by:
$$
D_F=\left\{(z_0,\dots,z_{n-1})\in\C^n\ | \ ||z||^2<e^{-|z_0|^2}\right\},
$$
with $F(t)=e^{-t}$, $x_0=+\infty$, and let $g_F$ the K\"ahler metric defined by the K\"ahler potential $-\log(F(|z_0|^2)-||z||^2)$. Prove that $(D_F,g_F)$ admits a full K\"ahler immersion into $\mathds{C}{\rm P}^\infty$ for any $c>0$, and thus into $l^2(\mathds{C})$.
\Exercise\label{dnotsufficient} Consider the rotation invariant Hartogs domain given by:
$$D_F=\left\{(z_0,\dots,z_{n-1})\in\C^n\ | \ ||z||^2<\frac{\alpha}{|z_0|^2+\alpha}\right\},$$
for $F(x)=\frac{\alpha}{x+\alpha}$, $\alpha>0$, $x_0=+\infty$ and let $g_F$ the K\"ahler metric associated to the K\"ahler potential $-\log(F(|z_0|^2)-||z||^2)$. Prove that $(D_F,c\,g_F)$ is projectively induced if and only if $c$ is a positive integer.
\Exercise\label{notb} Consider the rotation invariant Hartogs domain given by:
$$D_F(p)=\left\{(z_0,\dots,z_{n-1})\in\C^n\ | \ ||z||^2<\frac{1}{(|z_0|^2+1)^p}\right\},$$
for $F(x)=\frac{1}{(x+1)^p}$, $p>0$, $x_0=+\infty$ and let $g_F(p)$ the K\"ahler metric associated to the K\"ahler potential $-\log(F(|z_0|^2)-||z||^2)$. Prove that $(D_F(p),c\,g_F(p))$ is projectively induced if and only if $cp$ is a positive integer.

\Exercise\label{FBHdomains} For any value of $\mu>0$, a Fock--Bargmann--Hartogs domain $D_{n,m}(\mu)$ is a strongly pseudoconvex, nonhomogeneous unbounded domain in $\mathds{C}^{n+m}$ with smooth real-analytic boundary, given by:
$$
D_{n,m}(\mu):=\{(z,w)\in \mathds{C}^{n+m}: ||w||^2< e^{-\mu||z||^2}\}.
$$
One can define a K\"ahler metric $\omega(\mu;\nu)$, $\nu>-1$ on $D_{n,m}(\mu)$ through the globally defined K\"ahler potential:
$$
\Phi(z,w):=\nu\mu||z||^2-\log(e^{-\mu||z||^2}-||w||^2).
$$
Prove that the metric $g(\mu;\nu)$ on the Fock--Bargmann--Hartogs domain $D_{n,m}(\mu)$ admits a full K\"ahler immersion into $l^2(\mathds{C})$ for any value of $\mu>0$ and $\nu>-1$ (see \cite{fbh} for more details and results on these domains).

%

\end{ExerciseList}

\chapter{Relatives}
We say that two K\"ahler manifolds (finite or infinite dimensional) $M_1$ and $M_2$ are relatives if they share a complex K\"ahler submanifold
$S$, i.e. if there exist two K\"ahler immersions $h_1\!:S\rightarrow M_1$ and
$h_2\!:S\rightarrow M_2$. Otherwise, we say that $M_1$ and $M_2$ are not relatives. Further, we say that two K\"ahler manifolds are {\em strongly not relatives} if they are not relatives even when the metric of one of them is rescaled by the multiplication by a positive constant.

This terminology has been introduced in \cite{relatives}, even if the problem of understanding when two K\"ahler manifolds share a K\"ahler submanifold has been firstly considered by M. Umehara  \cite{umeharafinite}, which solves the case of complex space forms with holomorphic sectional curvature of different sign and finite dimension, which we summarize in Section \ref{relumehara}.

In the remaining part of this chapter we pay particular attention to understanding whether or not a K\"ahler manifold $(M,g)$ is relative to a projective K\"ahler manifold, which is by definition a K\"ahler manifold admitting a K\"ahler immersion into a finite dimensional complex projective space $\CP^N$. In Section \ref{relhkmsection} we discuss the case when $(M,g)$ is homogeneous while in Section \ref{relhartogs} $(M,g)$ is a Bergman--Hartogs domain.

\section{Relatives complex space forms}\label{relumehara}

In \cite{umeharafinite} M. Umehara proved that two finite dimensional complex 
space forms with holomorphic sectional curvatures of different signs
do not share a common K\"ahler submanifold. His result should be compared to Bochner's Theorem  \ref{bochnerth} (see also Theorem \ref{bochnerthalt}), which shows that when the ambient space is allowed to be infinite dimensional, the situation is much different (see also \cite{cdsy} for the case when the complex space forms involved have curvatures of same sign).

The following theorems summarize Umehara's (and Bochner's) results. Recall that we denote by $\mathds{C}{\rm P}^{N}_{b}$, $b>0$, the complex projective space of holomorphic sectional curvature $4b$ and by $\mathds{C}{\rm H}^{N}_{b}$, $b<0$, the complex hyperbolic space of holomorphic sectional curvature $4b$ (cfr. Section \ref{csf}).
\begin{theor}\label{cncp}
Any K\"ahler submanifold of $\mathds{C}^N$, $N\leq \infty$, admits a full K\"ahler immersion into $\mathds{C}{\rm P}^{\infty}_b$, for any value of $b>0$. 
\end{theor}
\begin{proof}
Let $(M,g)$ be a K\"ahler submanifold of $\mathds{C}^N$, for $N\leq \infty$. Then for any $p\in  M$, there exist a neighbourhood $U$ and Bochner coordinates $(z_1,\dots, z_n)$ centered at $p$, such that $g$ is described by the K\"ahler potential:
$$
\dd_0(z)=\sum_{j=1}^N|f_j(z)|^2,
$$ 
where $f\!:U\rightarrow \mathds{C}^N$, $f(z)=(f_{1}(z),\dots, f_N(z))$ and $f(0)=0$. If $U$ admits a K\"ahler immersion $h\!:U\rightarrow \mathds{C}{P}^{N'}_b$, then we also have:
$$
\dd_0(z)=\frac1b\log\left(1+b\sum_{k=1}^{N'}|h_k(z)|^2\right),
$$ 
and thus:
$$
\sum_{k=1}^{N'}|h_k(z)|^2=\frac{\exp\left(b\sum_{j=1}^N|f_j(z)|^2\right)-1}b.
$$
Since:
$$
\exp\left(b\sum_{j=1}^N|f_j(z)|^2\right)=\sum_{k=1}^\infty \frac{\left(b\sum_{j=1}^N|f_j(z)|^2\right)^k}{k!}=\sum_{k=1}^\infty {b^k}\sum_{|m_j|=k}\frac{1}{m_j!}|f^{m_{j}}(z)|^2,
$$
the map $\psi$ defined by $\psi_j=\sqrt{b^k}f^{m_{j}}/\sqrt{m_j!}$, $|m_j|=k$, is a full holomorphic and isometric map from $U$ to $l^2(\mathds{C})\subset\mathds{C}{\rm P}^\infty$. By Calabi's Rigidity Theorem \ref{local rigidityb} we get $N'=\infty$. 
\end{proof}

\begin{cor}
There are not K\"ahler submanifolds of both the complex Euclidean space $\mathds{C}^{N<\infty}$ and the complex projective space $\mathds{C}{\rm P}^{N'<\infty}_b$.
\end{cor}
\begin{proof}
By Theorem \ref{cncp} any K\"ahler submanifold of $\mathds{C}^N$ admits a full K\"ahler immersion into $\CP_b^\infty$. Conclusion follows by Calabi Rigidity's Theorem \ref{local rigidityb}.
\end{proof}
\begin{theor}\label{chcn}
Any K\"ahler submanifold of $\mathds{C}{\rm H}_b^{N\leq \infty}$ admits a full K\"ahler immersion into $l^2(\mathds{C})$. 
\end{theor}
\begin{proof}
Let $(M,g)$ be a K\"ahler submanifold of $\mathds{C}{\rm H}^N$, for $N<\infty$. Then for any $p\in  M$, there exist a neighbourhood $U$ and Bochner coordinates centered at $p$, such that $g$ is described by the K\"ahler potential:
$$
\dd_0(z)=\frac1b\log\left(1+b\sum_{k=1}^{N}|h_k(z)|^2\right).
$$ 
where $h\!:U\rightarrow \mathds{C}{\rm H}^N$, $h(z)=(h_{1}(z),\dots, h_N(z))$ and we assume $h(0)=0$. If $U$ admits a K\"ahler immersion $f\!:U\rightarrow \mathds{C}^{N'}$, then we also have:
$$
\dd_0(z)=\sum_{j=1}^{N'}|f_j(z)|^2,
$$ 
and thus:
$$
\sum_{k=1}^{N'}|f_k(z)|^2=\frac1{b}\log\left(1-b\sum_{k=1}^{N}|h_k(z)|^2\right)=\sum_{k=1}^\infty \frac{(-b)^{k-1}}{k}\left(\sum_{j=1}^{N}|h_j(z)|^2\right)^k.
$$
Since:
$$
\sum_{k=1}^\infty \frac{(-b)^{k-1}}{k}\left(\sum_{j=1}^{N'}|h_j(z)|^2\right)^k=\sum_{k=1}^\infty \frac{(-b)^{k-1}}{(k-1)!}\sum_{|m_j|=k}\frac{|h^{m_j}|^2}{m_j!}.
$$
It follows that the map $\psi$ defined by $\psi_j=\sqrt{\frac{(-b)^{k-1}}{(k-1)!}}h^{m_{j}}/\sqrt{m_j!}$, $|m_j|=k$, is a full holomorphic and isometric map from $U$ to $l^2(\mathds{C})$. By Calabi's Rigidity Theorem \ref{local rigidity} we get $N'=\infty$. 
\end{proof}

\begin{cor}
There are not K\"ahler submanifolds of both the complex Euclidean space $\mathds{C}^{N<\infty}$ and the complex hyperbolic space $\mathds{C}{\rm H}^{N'<\infty}_b$.
\end{cor}
\begin{proof}
By Theorem \ref{chcn} any K\"ahler submanifold of $\CH^N_b$ admits a full K\"ahler immersion into $\ell^2(\mathds{C})$. Conclusion follows by Calabi Rigidity's Theorem \ref{local rigidity}.
\end{proof}
\begin{cor}
There are not K\"ahler submanifolds of both the complex hyperbolic space $\mathds{C}{\rm H}_b^{N<\infty}$ and the complex projective space $\mathds{C}{\rm P}_{b'}^{N'<\infty}$. 
\end{cor}
\begin{proof}
By Theorem \ref{cncp} and Theorem \ref{chcn} follows that any K\"ahler submanifold of $\mathds{C}{\rm H}^N_b$, $N\leq \infty$, admits a full K\"ahler immersion into $\mathds{C}{\rm P}^{\infty}_{b'}$. Conclusion follows by Calabi Rigidity's Theorem \ref{local rigidityb}.
\end{proof}

\section[H.K.m. are not relative to projective ones]{Homogeneous K\"ahler manifolds are not relative to projective ones}\label{relhkmsection}

In this section we discuss when a homogeneous K\"ahler manifold is relative to a projective one. Recall that a projective K\"ahler manifold is (by definition) a K\"ahler manifold which admits a K\"ahler immersion into a finite dimensional complex projective space $\CP^N$.  
 We begin with the following theorem.

\begin{theor}[A. J. Di Scala, A. Loi, \cite{relatives}]\label{bergmanrel}
A bounded domain $D$ of $\mathds{C}^n$ endowed with its Bergman metric and a projective K\"ahler manifold are not relatives.
\end{theor}
\begin{proof}
Observe first that by Prop. \ref{induceddiast}, it is enough to show that $(D,g_B)$ is not relative to $\mathds{C}{\rm P}^N$ for any finite $N$.
Since $D$ is bounded, $L_{hol}^2(D)$ contains all polynomials in the variables $z_1,\dots, z_n$. In particular, we can choose an orthonormal basis containing $\lambda_kz_1^k$, for any $k=0,1\dots$ and some suitable constants $\lambda_k$ and a full K\"ahler immersion $F\!:D\rightarrow\mathds{C}{\rm P}^\infty$ is given by $[\lambda_0,\lambda_1z_1,\dots, \lambda_kz_1^k,\dots, \tilde F]$, where $\tilde F$ is the sequence of holomorphic functions which complete $\{\lambda_kz_1^k\}$ as orthonormal basis of $L_{hol}^2(D)$.

Assume by contradiction that $S$ is a $1$-dimensional common K\"ahler submanifold of $\mathds{C}{\rm P}^N$ and $(D,g_B)$ and denote by $\alpha\!:S\rightarrow D$ and $\beta\!:S\rightarrow \mathds{C}{\rm P}^N$ the K\"ahler immersion. By Prop. \ref{induceddiast} and Theorem \ref{local rigidityb}, it is enough to show that $F\circ\alpha$ is a full K\"ahler map from $S$ to $\mathds{C}{\rm P}^\infty$. Let $\alpha=(\alpha_1,\dots, \alpha_n)$. Since the role of $z_1$ in the above construction can be switched to any other $z_j$, we can assume without loss of generality  that $\frac{\partial \alpha_1}{\partial \xi}\neq 0$, where $\xi$ is the coordinate on $S$. Conclusion follows since $\{\lambda_k\alpha_1^k\}$ is a subsequence of $\{(F\circ \alpha)_j\}$ composed by linear independent functions.
\end{proof}

Observe that such result does not hold for multiples of the Bergman metric, since it makes use of the existence of a full K\"ahler immersion of $(D,g_B)$ into $\mathds{C}{\rm P}^\infty$, which is in general not guaranteed when one multiplies $g_B$ by a positive constant $c$. 
As we see in a moment, an improvement in this direction can be achieved by adding the assumption of $D$ to be symmetric or homogeneous.

Dealing with the symmetric case, observe that since a Hermitian symmetric space of compact type with integral K\"ahler form admits a K\"ahler immersion into some finite dimensional complex projective space, then due to theorems \ref{cncp} and \ref{chcn} and their corollaries it does not share a common K\"ahler submanifold with either the complex flat space or the complex hyperbolic space of finite dimensions. It is still an open question if a Hermitian symmetric space of compact type is relative to $l^2(\mathds{C})$ or to $\mathds{C}{\rm H}^\infty$ and what happens when the metric is rescaled to be not integral.

Consider now a bounded symmetric domain $\Omega$ of $\mathds{C}^n$ and let $g_B$ denote its Bergman metric. The complex flat space and $(\Omega, g_B)$ are not relatives due to a result by X. Huang and Y. Yuan \cite{relct}. Further, due to Theorem \ref{bergmanrel} nor are $(\Omega, g_B)$ and a projective K\"ahler manifold. Observe that when we deal with flat spaces we can forget about rescaling the metric by a positive constant. Although, the situation is different dealing with projective spaces. 
We ask what happens when the metric is rescaled by the multiplication to a positive constant $c$. The following theorem answers this question.
\begin{theor}[A. J. Di Scala, A. Loi, \cite{relatives}]\label{symmrel}
A bounded symmetric domain $(\Omega,cg_B)$ endowed with a positive multiple of its Bergman metric is not relative to any projective K\"ahler manifold.
\end{theor}
 This result
 follows from Theorem \ref{bergmanrel} and the following general lemma, which proves that when the K\"ahler manifold considered is {\em regular}, in the sense that it is projectively induced when rescaled by a great enough constant, the property of not being relative to any projective K\"ahler manifold is invariant by the multiplication of the metric by a positive constant.
 \begin{lem}[M. Zedda \cite{chrel}]\label{consequences}
Assume that $(M,\beta g)$ is infinite projectively induced for any $\beta>\beta_0\geq 0$. Then, if $(M,g)$ and $\mathds{C}{\rm P}^n$ are not relatives for any $n<\infty$, then the same holds for $(M,cg)$, for any $c>0$.
\end{lem}
\begin{proof}
For any $c>0$, we can choose a positive integer $\alpha$ such that  $c\alpha>\beta_0$. Denote by $\omega$ the K\"ahler form associated to $g$. Let $F\!:M\rightarrow \mathds{C}{\rm P}^\infty$ be a full K\"ahler map such that $F^*\omega_{FS}=c\alpha\, \omega$. Then $\sqrt{\alpha}F$ is a K\"ahler map of $(M,c\, g)$ into $ \mathds{C}{\rm P}^\infty_\alpha$.
 Let $S$ be a $1$-dimensional common K\"ahler submanifold of $(M,c\,g)$ and $\mathds{C}{\rm P}^n$. Then by Theorem \ref{induceddiast} for any $p\in S$ there exist a neighbourhood $U$ and two holomorhic maps $f\!:U\rightarrow M$ and $h\!:U\rightarrow \mathds{C}{\rm P}^n$, such that $f^*(c\omega)|_U=(\sqrt\alpha F\circ f)^*\omega_{FS}|_U=h^*\omega_{FS}|_U$.
 
Thus, by \eqref{diastcp} one has:
$$
\log\left(1+\sum_{j=1}^n|h_j|^2\right)=\frac1\alpha\log\left(1+\sum_{j=1}^\infty|(F\circ f)_j)|^2\right).
$$
i.e.:
\begin{equation}\label{contradiction}
\alpha\log\left(1+\sum_{j=1}^n|h_j|^2\right)=\log\left(1+\sum_{j=1}^\infty|(F\circ f)_j)|^2\right).
\end{equation}
If this last equality holds, then $U$ is a common K\"ahler submanifold of both $\mathds{C}{\rm P}^\infty$ and $\mathds{C}{\rm P}^n_{1/\alpha}$. This is a contradiction, for $F\circ f$ is full, since otherwise $U$ would be a K\"ahler submanifold of both $(M,g)$ and a finite dimensional complex projective space, and by Calabi rigidity Theorem \ref{local rigidityb}, from \eqref{contradiction} since $\alpha$ is integer we get $n=\infty$.
\end{proof}

Observe that in general there are not reasons for a K\"ahler manifold which is not relative to another K\"ahler manifold to remain so when its metric is rescaled. For example, consider that the complex projective space $(\mathds{C}{\rm P}^2,c\,g_{FS})$ where $g_{FS}$ is the Fubini--Study metric, for $c=\frac23$ is not relative to $\mathds{C}{\rm P}^2$, while for positive integer values of $c$ it is (see \cite{cdsy} for a proof).

 We refer the reader to \cite{relatives} for a proof that Hermitian symmetric spaces of compact and noncompact type are not relatives to each others.

We conclude this section with the following result due to R. Mossa \cite{mossarel}, which generalises Theorem \ref{symmrel} to bounded homogeneous domains:
\begin{theor}[R. Mossa, \cite{mossarel}]\label{mossarel}
A bounded homogeneous domain $(\Omega,g)$ and a projective K\"ahler manifold are not relatives.
\end{theor}
\begin{proof}
Observe first that by Prop. \ref{induceddiast}, it is enough to show that $(\Omega,g)$ is not relative to $\mathds{C}{\rm P}^N$ for any finite $N$.
Let $(\Omega,g)$ be a homogeneous bounded domain of $\mathds{C}^n$. Then by Th. \ref{loimossaimm}, there exists $\beta_0>0$ such that for any $\beta \geq \beta_0$, $\beta g$ is projectively induced. Denote $\tilde g=\beta_0 g$.  Due to Th. \ref{consequences}, it is enough to show that $\tilde g$ is not relative to $\mathds{C}{\rm P}^N$ to get the same assertion for a homogeneous metric not necessarily projectively induced. Since $(\Omega,\tilde g)$ is pseudoconvex, the K\"ahler metric $\tilde g$ admits a globally defined K\"ahler potential $\Phi$, i.e. $\tilde \omega=\frac i2\partial\bar \partial \Phi$, where we denote by $\tilde \omega$ the K\"ahler form associated to $g$. Denote by $\mathcal H_{\Phi}$ the weighted Hilbert space of square integrable holomorphic functions on $\Omega$, with weight $e^{-\Phi}$:
$$
\mathcal H_{\Phi}=\left\{ f\in \mathcal O(\Omega)|\, \int_\Omega e^{-\Phi}|f|^2\frac{\tilde\omega^n}{n!}<\infty\right\}.
$$
In \cite{loimossaber,mossarel} it is proven that $\mathcal H_{\Phi}\neq \{0\}$ and a K\"ahler immersion $F\!:\Omega\rightarrow \mathds{C}{\rm P}^\infty$, $F^*\omega_{FS}=\tilde \omega$ is given through one of its orthonormal bases (a K\"ahler metric satisfying such property is called {\em balanced}, see Remark \ref{balanced} for references). Since $\Omega$ is bounded, $\mathcal H_{\Phi}$ contains all the monomials $\{\lambda_kz_j^k\}$ for $j=1,\dots, n$ and $k=0,1,2,\dots$. Thus, a orthonormal basis of $\mathcal H_{\Phi}$ and hence the K\"ahler map $F$, can be written as $F=[P(z_1) ,f]$, where $P=[\dots, \lambda_kz_1^k,\dots]$ and $f$ is a sequence obtained by deleting from the basis $\{F_j\}$ the sequence $\{\lambda_kz_j^k\}$.
The proof is now totally similar to that of Th. \ref{bergmanrel}. Namely, assume by contradiction that $S$ is a $1$-dimensional common K\"ahler submanifold of $\mathds{C}{\rm P}^N$ and $(\Omega,\tilde g)$ and denote by $\alpha\!:S\rightarrow \Omega$ and $\beta\!:S\rightarrow \mathds{C}{\rm P}^N$ the K\"ahler immersion. By Prop. \ref{induceddiast} and Th. \ref{local rigidityb}, it is enough to show that $F\circ\alpha$ is a full K\"ahler map from $S$ to $\mathds{C}{\rm P}^\infty$. Let $\alpha=(\alpha_1,\dots, \alpha_n)$. Since the role of $z_1$ in the above construction can be switched to any other $z_j$, we can assume without loss of generality  that $\frac{\partial \alpha_1}{\partial \xi}\neq 0$, where $\xi$ is the coordinate on $S$. Conclusion follows since $\{\lambda_k\alpha_1^k\}$ is a subsequence of $\{(F\circ \alpha)_j\}$ composed by linear independent functions.
\end{proof}

\section[BH domains are not relative to a projective K\"ahler manifold]{Bergman--Hartogs domains are not relative to a projective K\"ahler manifold}\label{relhartogs}
We begin this section with a general result which somehow generalizes the peculiarity of the K\"ahler maps described in theorems \ref{bergmanrel} and \ref{mossarel}. In order to state it, consider a $d$-dimensional K\"ahler manifold $(M,g)$ which admits global coordinates $\{z_1,\dots, z_d\}$ and denote by $M_{j}$ the $1$-dimensional submanifold of $M$ defined by:
$$
M_j=\{ z\in M|\, z_1=\dots=z_{j-1}=z_{j+1}=\dots=z_d=0\}.
$$
When exists, a K\"ahler immersion $f\!:M\rightarrow \mathds{C}{\rm P}^\infty$ is said to be {\em transversally full} when for any $j=1,\dots, d$, the immersion restricted to $M_{j}$ is full into $\mathds{C}{\rm P}^\infty$.

\begin{theor}[M. Zedda \cite{chrel}]\label{trfull}
Let $(M, g)$ be a K\"ahler manifold infinite projectively induced through a transversally full map. If for any $\alpha\geq \alpha_0>0$, $(M,\alpha\, g)$ is infinite projectively induced then $(M,g)$ is strongly not relative to any projective K\"ahler manifold.
\end{theor}
\begin{proof}
Due to Lemma \ref{consequences} and Theorem \ref{induceddiast} we need only to prove that a if a K\"ahler manifold is infinite projectively induced through a  transversally full immersion then it is not relative to $\mathds{C}{\rm P}^n$ for any $n$. 
Assume that $S$ is a $1$-dimensional K\"ahler submanifold of both $\mathds{C}{\rm P}^n$ and $(M, g)$.
Then around each point $p\in S$ there exist an open neighbourhood $U$ and  two holomorphic maps $\psi\!:U\rightarrow \mathds{C}{\rm P}^n$ and $\varphi\!:U\rightarrow M$, $\varphi(\xi)=(\varphi_1(\xi),\dots,\varphi_d(\xi))$ where $\xi$ are coordinates on $U$, such that $\psi^*\omega_{FS}|_U=\varphi^*(c\omega)|_U$. Without loss of generality we can assume $\frac{\partial\varphi_1(\xi)}{\partial \xi}(0)\neq 0$.
 Let $f\!:M\rightarrow \mathds{C}{\rm P}^\infty$ be a K\"ahler map from $(M, g)$ into $\mathds{C}{\rm P}^\infty$. Since by assumption $f$ is transversally full, $f=[f_0,\dots, f_j,\dots]$ contains for any $m=1,2,3,\dots$, a subsequence $\left\{f_{j_1},\dots, f_{j_m}\right\}$ of functions which restricted to $M_1$ are linearly independent.
The map $f\circ\varphi\!:U\rightarrow \mathds{C}{\rm P}^\infty$ is full, in fact $f|_{M_1}\circ \varphi$ is full since $\varphi_1(\xi)$ is not constant and for any $m=1,2,3,\dots$, $\left\{f_{j_1}(\varphi_1(\xi)),\dots, f_{j_m}(\varphi_1(\xi))\right\}$ is a subsequence of $\{f|_{M_1}\circ \varphi\}$ of linearly independent functions.  Conclusion follows by Calabi's rigidity Theorem \ref{local rigidityb}.
\end{proof}

Combining  Theorems \ref{bergmanrel} and \ref{trfull} with Lemmata \ref{chimm} and \ref{consequences}, we get the following:
\begin{cor}\label{chrel}
For any $\mu>0$, a Bergman--Hartogs domain $(M_\Omega(\mu), g(\mu))$ is strongly not relative to any projective manifold. 
\end{cor}
\begin{proof}
Observe first that due to Theorem \ref{induceddiast} it is enough to prove that $(M_\Omega(\mu),\alpha g(\mu))$ is not relative to $\mathds{C}{\rm P}^n$ for any finite $n$. Further, by Lemma \ref{consequences} and Theorem \ref{bergmanrel}, a common submanifold $S$ of both $(M_\Omega(\mu),\alpha g(\mu))$ and $\mathds{C}{\rm P}^n$ is not contained into $(\Omega,\alpha g(\mu)|_\Omega)$, since $\alpha g(\mu)|_\Omega=\frac{\alpha\mu}\gamma g_B$ is a multiple of the Bergman metric on $\Omega$. Thus, due to arguments totally similar to those in the proof of Th. \ref{trfull}, it is enough to check that the K\"ahler immersion $f\!:M_\Omega(\mu)\rightarrow \mathds{C}{\rm P}^\infty$ is transversally full with respect to the $w$ coordinate. Conclusion follows then by \eqref{immf}. 
\end{proof}


\section{Exercises}
\begin{ExerciseList}
\Exercise Prove that for any integer $m>0$ the Hartogs domain $(D_F,m\,g_F)$ described in Exercise \ref{dnotsufficient} is relative to $\CP^1$.
\Exercise Consider the Hartogs domain $(D_F(p),c\,g_F(p))$ described in Exercise \ref{notb}. Prove that if  both $p$ and $c$ are positive integers then $(D_F(p),c\,g_F(p))$ is relative to $\CP^1$.
\Exercise Prove that the Hartogs domain described in Exercise \ref{springer} is strongly not relative to any projective K\"ahler manifold.
\Exercise Prove that for any $\mu$, $\alpha>0$, $\nu>-1$, a Fock--Bargmann--Hartogs domain (see Ex. \ref{FBHdomains} for a definition) $\left(D_{n,m}(\mu),\alpha\,\omega(\mu;\nu)\right)$ admits a transversally full K\"ahler immersion into $\mathds{C}{\rm P}^\infty$. 
\Exercise Prove that for any $\mu>0$, a Fock--Bargmann--Hartogs domain (see Ex. \ref{FBHdomains} for a definition) $\left(D_{n,m}(\mu),\omega(\mu;\nu)\right)$ is strongly not relative to any projective manifold. 
\end{ExerciseList}

\chapter{Further examples and open problems}
In this chapter we describe three K\"ahler manifolds with interesting properties. The first section summarizes the results in \cite{cigar} showing that the complex plane $\C$ endowed with the Cigar metric does not admit a local K\"ahler immersion into any complex space form even when the metric is rescaled by a positive constant. The importance of this example relies on the fact that there are not topological and geometrical obstructions for the existence of such an immersion.
In the second section we describe a complete and not locally homogeneous metric introduced by Calabi in \cite{calabi}. The diastasis function associated to this metric is not explicitely given and it makes very difficult to say something about the existence of a K\"ahler immersion into complex space forms.
Finally, in the third and last section we discuss a $1$-parameter family of nontrivial Ricci--flat metrics on $\C^2$, called Taub-NUT metrics. The diastasis associated to these metrics is rotation invariant, i.e. depends only on the module of the variables, but it is not explicitely given and it is still unknown whether or not they are projectively induced for small values of the parameter.
\section{The Cigar metric on $\mathds{C}$} \label{cigarsection}
The metric we describe in this section is an example of K\"ahler metric whose associated diastasis is globally defined and nonnegative but nevertheless it does not admit a K\"ahler immersion into any complex space form. 

The Cigar metric $g$ on $\mathds{C}$ has been introduced by Hamilton in \cite{Hamilton} as first example of K\"ahler--Ricci soliton on non-compact manifolds. It is defined by:
$$
g=\frac{dz\otimes d\bar z}{1+|z|^2}.
$$
A (globally defined) K\"ahler potential for this metric is given by (see also \cite{Suzuki}):
$$
D_0(|z|^2)=\int_{0}^{|z|^2}\frac{\log(1+s)}{s}ds,
$$
whose power series expansion around the origin reads:
\begin{equation}\label{expdiast}
D_0(|z|^2)= \sum_{j=1}^\infty (-1)^{j+1}\frac{|z|^{2j}}{j^2}.
\end{equation}
By duplicating the variable in this last expression, by the very definition of diastasis function \eqref{diastdefinition} we get:
\begin{equation}\label{diastdef}
D^g(z,w)=\sum_{j=1}^\infty\frac{(-1)^{j+1}}{j^2}\left(|z|^{2j}+|w|^{2j}-(z\bar w)^{2j}-(w\bar z)^{2j}\right).
\end{equation}
The following lemma proves that $D^g(z,w)$ is everywhere nonnegative and globally defined on $\mathds{C}\times\C$ (the fact that $D^g(z,w)$ is globally defined was also observed in \cite{Suzuki}).
\begin{lem}[A. Loi, M. Zedda, \cite{cigar}]\label{pos}
The diastasis function \eqref{diastdef} of the Cigar metric is globally defined and nonnegative.
\end{lem}
\begin{proof}
If we denote by ${\rm Li}_n(z)$ the polylogarithm function, defined for $|z|<1$ by:
$$
{\rm Li}_n(z)=\sum_{j=1}^\infty\frac{z^j}{j^n},
$$
and by analytic continuation otherwise, from \eqref{diastdef} we can write $D^g(z,w)$ as:
$$
D(z,w)=-{\rm Li}_2(-|z|^2)-{\rm Li}_2(-|w|^2)+{\rm Li}_2(-z\bar w)+{\rm Li}_2(-w\bar z).
$$
Write $z=\rho_1e^{i\theta_1}$ and $w=\rho_2e^{\theta_2}$ and let $\alpha=\theta_1-\theta_2$. Then:
\begin{equation}
\begin{split}\label{diastcigarli}
D(z,w)=&-{\rm Li}_2(-\rho_1^2)-{\rm Li}_2(-\rho_2^2)+{\rm Li}_2(-\rho_1\rho_2\,e^{i\alpha})+{\rm Li}_2(-\rho_1\rho_2\,e^{-i\alpha})\\
=&-{\rm Li}_2(-\rho_1^2)-{\rm Li}_2(-\rho_2^2)+2{\rm ReLi}_2(-\rho_1\rho_2\,e^{i\alpha}),
\end{split}
\end{equation}
where we are allowed to take the real parts since $D^g(z,w)$ is real. 
In order to simplify the term ${\rm ReLi}_2(-\rho_1\rho_2\,e^{i\alpha})$, we recall the following formula due to Kummer (see \cite{kummer} or \cite[p.15]{lewin}):
\begin{equation}
\begin{split}
{\rm ReLi}_2(\rho e^{i\theta})=&\frac12\left({\rm Li}_2(\rho e^{i\theta})+\overline{{\rm Li}_2(\rho e^{i\theta})}\right)\\
=&-\frac12\left(\int_0^\rho\frac{\log(1-y e^{i\theta})}{y}dy+\int_0^\rho\frac{\log(1-y e^{-i\theta})}{y}dy\right)\\
=&-\frac12\int_0^{\rho} \frac{\log(1-2y\cos\theta+y^2)}{y}dy
\end{split}\nonumber
\end{equation}
i.e.:
$$
{\rm ReLi}_2(-\rho e^{i\alpha})=-\frac12\int_0^{\rho} \frac{\log(1+2y\cos(\alpha)+y^2)}{y}dy.
$$ 
Since $1+2y\cos(\alpha)+y^2$ is decreasing for $0<\alpha<\pi$ and increasing for $\pi<\alpha<2\pi$, $\alpha=\pi$ is a minimum. Thus:
\begin{equation}\label{reli}
{\rm ReLi}_2(-\rho e^{i\alpha})\geq-\int_0^{\rho} \frac{\log(|1-y|)}{y}dy,
\end{equation}
where:
$$
-\int_0^{\rho} \frac{\log(|1-y|)}{y}dy=\begin{cases}{\rm Li}_2(\rho)& \textrm{if}\ \rho\leq 1\\
\frac{\pi^2}{6}-{\rm Li}_2(1-\rho)-\ln(\rho-1)\ln(\rho)& \textrm{otherwise}.
\end{cases}
$$
Thus, when $\rho_1\rho_2\leq1$ from \eqref{diastcigarli} and \eqref{reli}, we get:
$$
D(z,w)\geq-{\rm Li}_2(-\rho_1^2)-{\rm Li}_2(-\rho_2^2)+2{\rm Li}_2(\rho_1\rho_2)\geq 0,
$$
where the last equality follows since all the factors in the sum are positive. When $\rho_1\rho_2>1$, from \eqref{diastcigarli} and \eqref{reli}, we get:
\begin{equation}\label{lastd}
D(z,w)\geq-{\rm Li}_2(-\rho_1^2)-{\rm Li}_2(-\rho_2^2)+\frac{\pi^2}{3}-2{\rm Li}_2(1-\rho_1\rho_2)-2\ln(\rho_1\rho_2-1)\ln(\rho_1\rho_2).
\end{equation}
The RHS is positive for $1<\rho_1\rho_2\leq2$ since it is sum of positive factors. When $\rho_1\rho_2>2$, since all the factors are monotonic, it is enough to consider the limit as $\rho_1$ goes to $+\infty$ of $-D(z,w)/{\rm Li}_2(-\rho_1^2)$. By \eqref{lastd} above we get:
$$
\lim_{\rho_1\rightarrow+\infty}\frac{D(z,w)}{-{\rm Li}_2(-\rho_1^2)}\geq\frac52,
$$
and we are done.
\end{proof}

In the previous chapters we have seen how the multiplication of a K\"ahler metric by a positive constant $c$ affects its  being projectively induced. The interest of the Cigar metric relies on the fact that it does not admit a K\"ahler immersion into $\mathds{C}{\rm P}^\infty$ for any value of $c$ (and thus due to theorems \ref{bochnerth} and \ref{bochnerthalt} into any other complex space form).
Observe that Calabi himself provides in \cite{Cal} examples of metrics which have the same property.
\begin{ex}\label{negativex}\rm
Consider on $\mathds{C}$ the metric $g$ whose associate K\"ahler form $\omega$ is given by:
$\omega=\left(4\cos(z-\bar z)-1\right)dz\wedge d\bar z$. The associated (globally defined) diastasis:
$$
D(p,q)=4\left[\cos(p-\bar p)+\cos(q-\bar q)-\cos(p-\bar q)-\cos(q-\bar p)\right]-|p-q|^2,
$$
takes negative values, e.g. for $q=p+2\pi$.
\end{ex}
\begin{ex}\rm
Consider the product $\CP^1\times\CP^1$ endowed with the metric $g=b_1 g_{FS}\oplus b_2 g_{FS}$, with $b_1$, $b_2$ positive real numbers such that $b_2/b_1$ is irrational. Then $(\CP^1\times\CP^1,cg)$ does not admit a K\"ahler immersion into $\CP^\infty$ for any value of $c$.
In fact, in \cite[Theorem 13]{Cal}, Calabi proves that $(\CP^n,cg_{FS})$ admits a K\"ahler immersion into $\CP^\infty$ iff $1/c$ is a positive integer, and this property cannot be fulfilled by both $1/cb_1$ and $1/cb_2$.
\end{ex} 
Althought, both those metrics present geometrical obstructions to the existence of a K\"ahler immersion into $\CP^\infty$ that put aside the role of $c$. More precisely, in the first example
the diastasis associated to $g$ is negative at some points, while in the second one  the  \K\ form  $\omega$ associated to $g$ is not integral. In this sense, the Cigar metric is important not only because it cannot be  \K\ immersed into any (finite or infinite dimensional) complex space form for any $c>0$ but also since its associated K\"ahler form is integral and its diastasis is globally defined on $\mathds{C}\times \mathds{C}$ and positive (cfr. Lemma \ref{pos} above).

In order to prove the nonexistence of a K\"ahler immersion of $(\mathds{C},c\,g)$ into $\mathds{C}{\rm P}^\infty$ we need the following definition and properties of Bell polynomials. The partial Bell polynomials $B_{n,k}(x):=B_{n,k}(x_1,\dots, x_{n-k+1})$ of degree $n$ and weight $k$ are defined by (see e.g. \cite[p. 133]{comtet}):
\begin{equation}\label{bjkdef}
B_{n,k}(x_1,\dots, x_{n-k+1})=\sum_{\pi(k)}\frac{n!}{s_1!\dots s_{n-k+1}!}\left(\frac{x_1}{1!}\right)^{s_1}\left(\frac{x_2}{2!}\right)^{s_2}\cdots \left(\frac{x_{n-k+1}}{(n-k+1)!}\right)^{s_{n-k+1}},
\end{equation}
where the sum is taken over the integers solutions of:
$$
\begin{cases}s_1+2s_2+\dots+ks_{n-k+1}=n\\ s_1+\dots+s_{n-k+1}=k.\end{cases}
$$
Bell polynomials satisfy the following equalities (the second one has been firstly pointed out in \cite{dc}):
\begin{equation}\label{propr}
B_{n,k}(trx_1,tr^2x_2,\dots,tr^{n-k+1}x_{n-k+1})=t^kr^nB_{n,k}(x_1,\dots, x_{n-k+1}).
\end{equation}
\begin{equation}\label{prop2}
\begin{split}
B_{n,k+1}(x)=\frac{1}{(k+1)!}\sum_{\alpha_1=k}^{n-1}\sum_{\alpha_2=k-1}^{\alpha_1-1}\cdots \sum_{\alpha_k=1}^{\alpha_{k-1}-1}&{n\choose \alpha_1}{\alpha_1\choose \alpha_2}\cdots{\alpha_{k-1}\choose \alpha_k} \cdot\\
&\cdot x_{n-\alpha_1}x_{\alpha_1-\alpha_2}\cdots x_{\alpha_{k-1}-\alpha_k}x_{\alpha_k}.
\end{split}
\end{equation}

The complete Bell polynomials are given by:
$$
Y_n(x_1,\dots, x_n)=\sum_{k=1}^nB_{n,k}(x),\quad Y_0:=0,
$$
and the role they play in our context is given by the following formula \cite[Eq. 3b, p.134]{comtet}:
\begin{equation}\label{exp}
\frac{d^n}{dx^n}\left(\exp\left(\sum_{j=1}^\infty a_j\frac{x^j}{j!}\right)\right)|_0=Y_n(a_1,\dots, a_n).
\end{equation}
Observe that from \eqref{propr} it follows:
\begin{equation}\label{proprY}
Y_n(rx_1,r^2x_2,\dots,r^{n}x_{n})=r^nY_n(x_1,\dots, x_{n}).
\end{equation}

\begin{theor}\label{maincigar}
Let $g=\frac{1}{1+|z|^2}dz\otimes d\bar z$ be the Cigar metric on $\C$. Then the diastasis function of the metric  $g$ is globally defined and positive  on $\C\times \C$ and  $(\C, c g)$ cannot be (locally) \K\ immersed   into any complex space form for any $c>0$.
\end{theor}
\begin{proof}

Observe first that if $(M,cg)$ does not admit a K\"ahler immersion into $\CP^\infty$ for any value of $c>0$, then it does not either into any other space form. 
In fact, if $(M,cg)$ admits a K\"ahler immersion into $l^2(\mathds{C})$ then by Th. \ref{bochnerth}, it also does into $\CP^\infty$, and in particular since the multiplication by $c$ is harmless when one considers K\"ahler immersion into flat spaces, it does for any value of $c>0$. Further, Th. \ref{bochnerthalt} implies that it does not admit a K\"ahler immersion into $\CH^\infty$ either.
Thus, it is enough to show that $(\mathds{C},cg)$ does not admit a K\"ahler immersion into $\CP^\infty$ for any $c>0$. 

Further, the diastasis function associated to the Cigar metric is globally defined and positive by Lemma \ref{pos} above. 

Then, by Calabi's criterions, it remains only to show that there exists $n$ such that:
$$
\frac{\partial^{2n}\exp\left(cD_0(|z|^2)\right)}{\partial z^n\partial \bar z^n}|_0< 0,
$$
where $D_0(|z|^2)$ is the K\"ahler potential defined in \eqref{expdiast}.
Observe first that setting:
\begin{equation}\label{aj}
\tilde a_j:=-c\frac{j!}{j^2},
\end{equation}
by \eqref{exp} and \eqref{proprY} we get:
\begin{equation}
\begin{split}
\frac{\partial^{2n}\exp\left(cD_0(|z|^2)\right)}{\partial z^n\partial \bar z^n}|_0=&\frac{1}{n!}\frac{d^{n}\exp\left(cD_0(x)\right)}{dx^n}|_0\\
=&\frac{1}{n!}Y_n\left(-\tilde a_1,(-1)^2\tilde a_2,\dots, (-1)^n\tilde a_n\right)\\
=&\frac{(-1)^n}{n!}Y_n\left(\tilde a_1,\dots, \tilde a_n\right).
\end{split}\nonumber
\end{equation}
We wish to prove that for any $c>0$ there exists $n$ big enough such that:
$$
Y_{2n}\left(a_1,\dots, a_{2n}\right)<0.
$$
Observe first that since $\tilde a_j=-c\, a_j$ with $a_j=j!/j^2$, we get:
\begin{equation}
\begin{split}
{Y_{2n}(\tilde a)}=&\sum_{k=1}^{2n}(-1)^{k}c^kB_{2n,k}( a)\\
=&\,\frac{(2n)!c}{(2n)^2}\left(-1+\frac{c(2n)^2B_{2n,2}( a)}{(2n)!}-\frac{c^2(2n)^2B_{2n,3}( a)}{(2n)!}+\dots+\frac{c^{2n-1}(2n)^2}{(2n)!}\right).
\end{split}\nonumber
\end{equation}
Thus, we need to prove that for any value of $c$ there exists $n$ large enough such that the following inequality holds:
\begin{equation}\label{eqtobeproven}
\frac{c(2n)^2B_{2n,2}( a)}{(2n)!}-\frac{c^2(2n)^2B_{2n,3}( a)}{(2n)!}+\dots+\frac{c^{2n-1}(2n)^2}{(2n)!}<1.
\end{equation}
Since (see \cite[Lemma 4]{cigar}):
$$
\lim_{n\rightarrow\infty}\frac{(2n)^2B_{2n,k+1}(a)}{(2n)!}=\frac{k+1}{(k+1)!}\sum_{j_1=1}^\infty\frac{1}{j_1^2}\cdots\sum_{j_{k-1}=1}^\infty\frac{1}{j_{k}^2},
$$
then:
$$
\lim_{n\rightarrow+\infty}\frac{(2n)^2B_{2n,k+1}( a)}{(2n)!}=\frac{k+1}{(k+1)!}\sum_{j_1=1}^\infty\frac{1}{j_1^2}\sum_{j_2=1}^\infty\frac{1}{j_2^2}\cdots\sum_{j_k=1}^\infty\frac{1}{j_k^2}.
$$
Further, by:
$$
\sum_{j=1}^\infty\frac{1}{j^2}=\frac{\pi^2}{6},
$$
we get:
$$
\lim_{n\rightarrow+\infty}\frac{(2n)^2B_{2n,k+1}( a)}{(2n)!}=\frac{1}{k!}\left(\frac{\pi^2}{6}\right)^k.
$$
Plugging this into \eqref{eqtobeproven}, we get that as $n$ goes to infinity the left hand side converge to:
$$
\sum_{k=1}^\infty\frac{(-1)^{k+1}c^k}{k!}\left(\frac{\pi^2}{6}\right)^k=1-e^{-\frac{c\,\pi^2}{6}},
$$
and conclusion follows by observing that $1-e^{-\frac{c\,\pi^2}{6}}$ is strictly increasing as a function on $c$ and its limit value as $c$ grows is $1$.
\end{proof}
\begin{remark}\rm
It is worth pointing out that the cigar metric has positive sectional curvature. Hence, in view of Theorem \ref{maincigar},  it is  interesting to see if there exist examples  
of  negatively curved real analytic \K\ manifolds $(M ,g)$  with  globally  defined diastasis function which is positive and  such that 
$(M, cg)$ cannot be locally \K\ immersed into any complex space form for all $c>0$. 
\end{remark}

\section{Calabi's complete and not locally homogeneous metric}

Consider the complex tubular domain $M_n=\frac{1}{2}D\oplus i\R^n\subset\C^n$, where $D$ denotes any connected, open subset of $\R^n$. Let $g_n$ be the metric on $M_n$ whose associated K\"ahler form is given by:
\begin{equation}
\omega_n=\frac{i}{2}\de\bar\de F(z),\nonumber
\end{equation}
with:
\begin{equation}
F(z)=f(z_1+\bar z_1, \dots,z_n+\bar z_n),\nonumber
\end{equation}
where $f\!:D\f \R$ is a radial function $f(x_1,\dots,x_n)=y(r)$, with $r=(\sum_{j=1}^nx_j^2)^{1/2}$, satisfying the differential equation:
\begin{equation}\label{basic}
\left(\frac{y'}{r}\right)^{n-1}y''=e^y,
\end{equation}
with initial conditions:
\begin{equation}\label{incond}
y'(0)=0,\ y''(0)=e^{y(0)/n}.
\end{equation}
This metric introduced by Calabi \cite{calabi} is the first example of complete and not locally homogeneous K\"ahler--Einstein metric. In \cite{wolf} J. A. Wolf gives a stronger more straight-forward version of Calabi's result, namely if $n\geq 2$ and $g_n$ is an $E(n)$-invariant K\"ahler metric on $M_n$, where $E(n)=\R^n\cdot\text{SO}(n)$, then $(M_n,g_n)$ cannot be both complete and locally homogeneous. Moreover, $E(n)$ is the largest connected group of holomorphic isometries of $(M_n,g_n)$.

It is still an open question if this metric admits or not a K\"ahler immersion into some complex space form (except with the case exposed in Exercise \ref{nlhcm}). Although, the following lemma guarantees that the metric is smooth around the origin and we are able to write its diastasis function.
\begin{lem}
If $y(r)$ is a function satisfying (\ref{basic}), then $y(r)$ is smooth at $r=0$.
\end{lem}
\begin{proof}
Let $y(r)$ be a solution of (\ref{basic}). From $((y')^n)'=n(y')^{n-1}y''$ and (\ref{basic}) we get:
\begin{equation}
y'(r)=\left(n\int_0^rt^{n-1}e^{y(t)}dt\right)^{1/n},\nonumber
\end{equation}
and by substituting $t=rs$ we have:
\begin{equation}
y'(r)=r\left(n\int_0^1s^{n-1}e^{y(rs)}ds\right)^{1/n}.\nonumber
\end{equation}
Since $\int_0^1s^{n-1}e^{y(rs)}ds$ is not zero at $r=0$, the last equation implies that $y'(r)\in C^k(0)$ whenever $y(r)\in C^k(0)$. By (\ref{incond}), $y(r)\in C^2(0)$ and we are done.
\end{proof}
By recursion from (\ref{basic}) and (\ref{incond}), one obtains that for all $m\in\N$:
$$
y^{(2m+1)}(0)=0,
$$
thus the power expansion of $y(r)$ around the origin is of the form:
\begin{equation}
y(r)=y(0)+\frac{y''(0)}{2}\, 
r^2+\frac{y^{(4)}(0)}{4!}\, r^4+\frac{y^{(6)}(0)}{6!}\, r^6+\dots,\nonumber
\end{equation}
and we have:
\begin{equation}
\begin{split}
F(z,\bar z)=F(0,0)+\,&\frac{y''(0)}{2}\,\sum_{j=1}^n (z_j+\bar z_j)^2+\frac{y^{(4)}(0)}{4!}\,\left(\sum_{j=1}^n (z_j+\bar z_j)^2\right)^2+\\
+\,&\frac{y^{(6)}(0)}{6!}\left(\sum_{j=1}^n (z_j+\bar z_j)^2\right)^3+\dots.
\end{split}\nonumber
\end{equation}
Thus by \eqref{diastdefinition}, we have:
\begin{equation}
D_0(z)=F(z,\bar z)+F(0,0)-F(0,\bar z)-F(z,0),\nonumber
\end{equation}
i.e.:
\begin{equation}
\begin{split}
D_0(z)=&\, \frac{y''(0)}{2}\,\left(\sum_{j=1}^n (z_j+\bar z_j)^2-\sum_{j=1}^n\bar z_j^2-\sum_{j=1}^n z_j^2\right)+\\
&+\frac{y^{(4)}(0)}{4!}\,\left(\left(\sum_{j=1}^n (z_j+\bar z_j)^2\right)^2-\left(\sum_{j=1}^n\bar z_j^2\right)^2-\left(\sum_{j=1}^n z_j^2\right)^2\right)+\\
&+\frac{y^{(6)}(0)}{6!}\,\left(\left(\sum_{j=1}^n (z_j+\bar z_j)^2\right)^3-\left(\sum_{j=1}^n\bar z_j^2\right)^3-\left(\sum_{j=1}^n z_j^2\right)^3\right)+\\
&+\dots\\
&+\frac{y^{(2k)}(0)}{(2k)!}\left(\left(\sum_{j=1}^n (z_j+\bar z_j)^2\right)^k-\left(\sum_{j=1}^n\bar z_j^2\right)^k-\left(\sum_{j=1}^n z_j^2\right)^k\right)+\\
&+\dots.
\end{split}\nonumber
\end{equation}
Observe that we can assume $y''(0)=1$ and the coefficients $\frac{y^{(2k)}(0)}{(2k)!}$ can be computed from \eqref{basic} and \eqref{incond}. Although, the matrices of coefficients in the power expansions \eqref{ddconv}, \eqref{powexdiastcp} and \eqref{powexdiastch} are not diagonal and it is not easy to find a negative eigenvalue or to prove they are positive semidefinite.

\section{Taub-NUT metric on $\mathds{C}^2$}
In \cite{lebrun} C. Lebrun constructs the following family of K\"ahler forms on $\C^2$ defined by $\omega_m=\frac{i}{2}\de\bar\de \Phi_m$, where:
\begin{equation}\label{Phim}
\Phi_m(u,v)=u^2+v^2+m(u^4+v^4),\ \textrm{for}\ m\geq 0,
\end{equation}
and $u$ and $v$ are implicitly defined by:
\begin{equation}\label{x1z1}
|z_1|=e^{m(u^2-v^2)}u,\ |z_2|=e^{m(v^2-u^2)}v.
\end{equation}
For $m=0$ one gets the flat metric, while for $m>0$ each of the metrics of this family represents the first example of complete Ricci--flat (non-flat) metric on $\C^2$ having the same volume form of the flat metric $\omega_0$, i.e. $\omega_m\wedge\omega_m=\omega_0\wedge\omega_0$. Moreover, for $m>0$, these metrics are isometric (up to dilation and rescaling) to the Taub-NUT metric.
\begin{lem}[A. Loi, M. Zedda, F. Zuddas, \cite{taubnut}]\label{pind}
Let $m\geq 0$,   $g_m$ be the  Taub--NUT metric on $\C^2$ and $\alpha$ be a positive real number. Then  $\alpha g_m$   is not projectively induced for   $m>\frac{\alpha}{2}$.
\end{lem}
\begin{proof}
Assume by contradiction that $\alpha g_m$ is projectively induced, namely that  there exists $N\leq\infty$ and   a \K\  immersion
of $(\C^2, \alpha g_m)$ into $\C P^N$. Then, it does exist also a \K\ immersion into $\CP^N$
of  the   \K\ submanifold of $(\C^2, \omega_m)$  defined by $z_2 = 0$, $z_1 = z$, endowed with the induced metric, having potential $\tilde \Phi_m = u^2 + m u^4$, where $u$ is defined implicitly by $z\bar{z} = e^{2 m u^2} u^2$. Observe that $\tilde \Phi_m$ is the diastasis function \eqref{diastdefinition} for this metric, since it is a rotation invariant potential centered at the origin.

Consider the power expansion around the origin of the function $e^{\alpha\tilde \Phi_m}-1$, that, by (\ref{Phim}) and (\ref{x1z1}), reads:
\begin{equation}
e^{\alpha\tilde\Phi_m}-1=\alpha |z|^2+\frac{\alpha}{2}(\alpha-2 m)|z|^4+\dots.\nonumber
\end{equation}
Since $\alpha-2 m\geq0$ if and only if $m\leq\frac{\alpha}{2}$, it follows by Calabi's criterion Th. \ref{localcritb} that $\alpha g_m$ can not admit a K\"ahler immersion into  $\CP^N$ for any $m>\frac{\alpha}{2}$.
\end{proof}

In \cite{taubnut} the authors state the following conjecture:
\begin{conj}
The Taub--NUT metric $\alpha g_m$ on $\mathds{C}^2$ is not projectively induced for any $m>0$.
\end{conj}

\section{Exercises}
\begin{ExerciseList}
\Exercise\label{nlhcm} Prove that for $n=2$, Calabi's complete not locally homogeneous metric $(M_2,\omega_2)$ does not admit a K\"ahler immersion into $\CH^N$, $N\leq \infty$.
\Exercise Verify that the Taub--NUT metric $(\C^2, \alpha g_m)$ cannot be \K\ immersed into the complex hyperbolic space $\C H^N$, $N\leq \infty$.
\Exercise Prove that if the Taub-NUT metric $(\C^2,  g_m)$ admits a \K\ immersion into
$\C^N$, $N\leq\infty$, then $m=0$.
\end{ExerciseList}

\backmatter

\end{document}